\providecommand{\href}[2]{#2}

\documentclass[12pt,a4paper,reqno]{amsart}
\usepackage[T1]{fontenc}
\usepackage[utf8]{inputenc}
\usepackage{color}
\usepackage{amsthm}
\usepackage{tikz-cd}
\usepackage{tikz}
\usepackage{tikz-qtree}
\usepackage{soul}
\usepackage{algorithm}
\usepackage{float}
\DeclareUnicodeCharacter{221E}{\infty}

\usepackage{pgfplots}
\pgfplotsset{compat=1.15}
\usepackage{mathrsfs}

\usepackage{algpseudocode}
\usepackage[colorinlistoftodos]{todonotes}
\usetikzlibrary{tqft, calc, arrows, decorations.markings, positioning}
\usepackage{varioref}
\usepackage{setspace}
\usepackage{pgfplots,caption}
\usepackage{subcaption}
\pgfplotsset{compat=1.9}
\usepackage{amsbsy}
\usepackage{graphicx}
\usepackage{amstext}
\usepackage{amssymb}
\usepackage{esint}
\usepackage{textcomp}
\usepackage[framemethod=TikZ]{mdframed}
\usepackage{lmodern}
\usepackage{setspace}
\usepackage{stackengine}
\usepackage{geometry}
\usepackage[unicode=true,pdfusetitle,
 bookmarks=true,bookmarksnumbered=false,bookmarksopen=false,
 breaklinks=false,pdfborder={0 0 0},backref=false,colorlinks=true]
 {hyperref}
\usepackage{breakurl}

\makeatletter

\numberwithin{equation}{section}
\numberwithin{figure}{section}
\theoremstyle{plain}
\newtheorem{thm}{\protect\theoremname}[section]
\surroundwithmdframed[outerlinewidth=0.5pt,
  innerlinewidth=0.5pt,
  middlelinewidth=1.5pt,
  middlelinecolor=white,
  bottomline=false,topline=false,rightline=false]{thm}
\theoremstyle{plain}
\newtheorem{theorem}{\protect\theoremname}

\surroundwithmdframed[outerlinewidth=0.5pt,
  innerlinewidth=0.5pt,
  middlelinewidth=1.5pt,
  middlelinecolor=white,
  bottomline=false,topline=false,rightline=false]{theorem}
  \theoremstyle{remark}
  \newtheorem{rem}[thm]{\protect\remarkname}
  
  \AtBeginEnvironment{rem}{%
  \pushQED{\qed}%
}
\AtEndEnvironment{rem}{\popQED\endexample}
  \theoremstyle{remark}
  \newtheorem{example}[thm]{\protect\examplename}
    \AtBeginEnvironment{example}{%
  \pushQED{\qed}%
}
\AtEndEnvironment{example}{\popQED\endexample}
  \theoremstyle{remark}
  
  \theoremstyle{remark}
  \newtheorem*{acknowledgement*}{\protect\acknowledgementname}
  \theoremstyle{plain}
  \newtheorem{prop}[thm]{\protect\propositionname}
  \theoremstyle{definition}
  \newtheorem{defn}[thm]{\protect\definitionname}
  \theoremstyle{plain}
  \newtheorem{lem}[thm]{\protect\lemmaname}

\newcommand{\Z}{\mathbb{Z}}
\newcommand{\R}{\mathbb{R}}

\newcommand{\id}{\mathrm{id}}

\newcommand{\lagmd}{\mathcal{L}^{m,\mathbf{d}}}
\newcommand{\col}{\textnormal{col}}

\newcommand{\red}{\textnormal{red}}

\newcommand{\Crit}{\mathrm{Crit}}

\newcommand{\beq}{\begin{equation}}
\newcommand{\beqn}{\begin{equation}\nonumber}
\newcommand{\eeq}{\end{equation}}

\newcommand{\bea}{\begin{equation}\begin{aligned}}
\newcommand{\bean}{\begin{equation}\begin{aligned}\nonumber}
\newcommand{\eea}{\end{aligned}\end{equation}}

\makeatother

  \providecommand{\acknowledgementname}{Acknowledgement}
  \providecommand{\corollaryname}{Corollary}
  \providecommand{\definitionname}{Definition}
  \providecommand{\examplename}{Example}
  \providecommand{\lemmaname}{Lemma}
  \providecommand{\propositionname}{Proposition}
  \providecommand{\remarkname}{Remark}
\providecommand{\theoremname}{Theorem}
\usepackage{lineno}
\usepackage[style=alphabetic]{biblatex}
\hypersetup{
    colorlinks=true,
    linkcolor=red,
    urlcolor=orange,
    linktoc=all,
    citecolor=orange,
          }

\begin{document}
\title{Filtered Fukaya categories}
\date{\today}
\author{Giovanni Ambrosioni}
\thanks{The author was partially supported by the Swiss National Science Foundation (grant number 200021-204107).}
\address{Giovanni Ambrosioni\\ ETH Z\"urich}
\email{giovanni.ambrosioni@math.ethz.ch}
\maketitle
\thispagestyle{empty}
\begin{abstract}
    We upgrade the natural weakly-filtered structure of Fukaya categories discussed in \cite{Biran2021LagrangianCategories} to a genuinely filtered one. The main tools are a Morse-Bott, or `cluster', model for Fukaya categories and a particular choice of class of perturbation data. We also include the construction of continuation $A_\infty$-functors following \cite{Sylvan2019OnCategories} in the context of filtered Fukaya categories.
\end{abstract}



\addtocontents{toc}{\protect\setcounter{tocdepth}{0}}

\section*{Introduction}
\subsection*{Background}
The application of the methods of persistence homology to symplectic geometry, in particular Floer theory, has lead to many important results. The basic fact is that in many cases Floer complexes naturally admit the structure of filtered chain complexes, meaning that the standard Floer differential preserves a filtration induced by the action functional, so that Floer persistence homology is a well-defined object, both in the closed string (or Hamiltonian) case as well as in the open (or Lagrangian) one. However, defining an adequate filtered structure on the $A_\infty$-categorification of Lagrangian Floer homology, the so called Fukaya category of a symplectic manifold, has been proven to be difficult, as discussed at length in \cite{Biran2021LagrangianCategories}, due to the presence of domain dependent Hamiltonian perturbations in the definition of the $A_\infty$-maps. On the other hand, in \cite{Biran2021LagrangianCategories}, the authors showed that Fukaya categories naturally admit the structure of a homologically unital \textit{weakly}-filtered $A_\infty$-category, meaning that the $A_\infty$-maps are filtered up to a uniform error depending on the order (which may eventually diverge) and representatives of the identities generally lie at positive filtration levels. If the uniform errors vanish and the representatives of units lie at  filtration level zero, we say that the $A_\infty$-category is filtered. In this paper we explain how to construct filtered Fukaya categories by slightly modifying the definition of the $A_\infty$-maps to count Floer `clusters' instead of Floer polygons, and by constructing adequate Hamiltonian perturbations. In particular, via our construction the derived Fukaya category of a symplectic manifold inherits the structure of a triangulated persistence category in the sense of \cite{Biran2023TriangulationCategories}.

\subsection*{Main results}
Let $(M,\omega)$ be a closed or convex at infinity symplectic manifold, and denote by $\mathcal{L}^*(M)$ a family of closed Lagrangian submanifolds of $M$ belonging to some class $*$ (e.g. $*=m$ for monotone Lagrangians, $*=w\text{-}ex$ for weakly exact Lagrangians, $*=ex$ for exact Lagrangians). In this paper we will work with monotone Lagrangians, but our construction works for weakly exact Lagrangians and exact Lagrangians in Lioville manifolds too.\\ We will refer to the model of the Fukaya category developed in Seidel's book \cite{Seidel2008FukayaTheory} as the \textit{standard model for the Fukaya category of $M$} (see also e.g. \cite{Biran2013LagrangianI, Biran2014LagrangianCategories, Biran2021LagrangianCategories} for this model adapted to the monotone setting). In this standard model, morphisms sets are Floer complexes $CF(L,L')$ where $L,L'\in \mathcal{L}^*(M)$ and the $A_\infty$-maps $$\mu_d: CF(L_0,L_1)\otimes \cdots \otimes CF(L_{d-1},L_d)\to CF(L_0,L_d)$$ are defined for any $d\geq 1$ and any tuple $(L_0,\ldots, L_d)$ of Lagrangians in $\mathcal{L}^*(M)$ by counting Hamiltonian perturbed pseudoholomorphic polygons with Lagrangian boundary conditions, so-called Floer polygons. The representatives of homological units are constructed by counting Floer `$1$-gons'. The action functional (defined in Section \ref{wfs}) induces an increasing real filtration on Floer complexes, that is a choice of subspaces $CF^{\leq \alpha}(L,L')$ for any $\alpha \in \R$ such that $$CF^{\leq \alpha}(L,L')\subset CF^{\leq \beta}(L,L') \textnormal{   for any }\alpha\leq \beta.$$ It is well-known that $\mu_1$ preserves the above defined filtration, while in general higher order maps do not, and representatives of the units lie at positive filtration levels (cfr. \cite{Biran2021LagrangianCategories}). Our main result, Theorem \ref{mainthm}, asserts that it is possible to construct a model for the Fukaya category admitting filtration-preserving $A_\infty$-maps as well as units lying in filtration level $\leq 0$. The following is a rough version of our main result.
\vspace{1.5mm}
\begin{theorem}\label{thmA} There is a non-empty space $E_\text{reg}$ of perturbation data such that associated to any $p\in E_\text{reg}$ there is a \textit{strictly} unital and \textit{filtered }$A_\infty$-category $\mathcal{F}uk(M;p)$, whose set of objects is $\mathcal{L}^*(M)$ and whose filtered structure is induced by the action functional. When ignoring filtrations this $A_\infty$-category is quasi-equivalent to the standard Fukaya category via an $A_\infty$-functor which is the identity on objects.
\end{theorem}

\noindent We will actually define the space $E_\text{reg}$ by defining (non-empty) subspaces $E^{\varepsilon}_\text{reg}$ for any $\varepsilon>0$ and then setting $E_\text{reg}:= \bigcup_{\varepsilon>0}E^{\varepsilon}_\text{reg}$. Here, $\varepsilon>0$ will act as a uniform bound to the size of the Hamiltonian Floer data one can associate to elements in the space $E^\varepsilon_\text{reg}$. The subscript `reg' has to do with the fact that at some point we will have to restrict ourselves to regular perturbation data in order to have well-defined $A_\infty$-categories (as in \cite{Seidel2008FukayaTheory}).

To prove Theorem \ref{thmA} we have to tackle two problems: define $A_\infty$-maps that do not shift filtration, and find represenatives of the unit lying at vanishing filtration levels. To solve the first problem we will construct the special space $E^\varepsilon$ of perturbation data.
The idea for construction of such spaces is very simple and we sketch it here: we pick `flat enough' Hamiltonian Floer data for any couple of Lagrangians in $\mathcal{L}^*(M)$ and use coordinates induced by strip-like ends on Floer polygons to interpolate between the Hamiltonian Floer datum and the zero function by a particular class of monotone homotopies, while asking that the Hamiltonian perturbation data vanish away from the strip-like ends. The definition of the spaces $E^\varepsilon$ of so-called $\varepsilon$-perturbation data is given in Section \ref{actualconstruction}. However, even when working with perturbation data from some family $E^\varepsilon$, the Fukaya category of $M$ may still be not filtered, because of the second problem listed above. Indeed, in Section \ref{addend} we argue that it is not possible to find representatives of the units at zero filtration levels by using the standard model for the Fukaya category. To cope with this fact, we use a Morse-bott, or `cluster', model for Fukaya categories, by defining a new $A_\infty$-category whose self-morphisms spaces are pearl complexes (in the sense of \cite{Biran2009LagrangianHomology, Biran2008RigiditySubmanifolds}), instead of Floer ones, and by considering `clusters' of pearly-Morse trees and Floer polygons. This model was first introduced in \cite{Cornea2005ClusterHomology} and later outlined in \cite{Biran2023TriangulationCategories} in the exact case; we work out more details of this model in Section \ref{mbfuk}.

After developing the machinery of $\varepsilon$-perturbation data and the associated filtered Fukaya categories we introduce continuation $A_\infty$-functors, an $A_\infty$-extension of continuation chain maps, in Section \ref{contfunc}, following \cite{Sylvan2019OnCategories}. The main result from Section \ref{contfunc} may be roughly stated as follows.
\vspace{1.5mm}
\begin{theorem}\label{thmB}
    Let $\varepsilon_1,\varepsilon_2>0$, $p\in E^{\varepsilon_1}_\text{reg}$ and $q\in E^{\varepsilon_2}_\text{reg}$. There is a unital $A_\infty$-quasi equivalence $$\mathcal{F}^{p,q}\colon \mathcal{F}uk(M;p)\to \mathcal{F}uk(M;q)$$ canonical up to quasi-isomorphism of $A_\infty$-functors which shifts filtration by $\leq \varepsilon_2-\frac{\varepsilon_1}{2}$.
\end{theorem}
\noindent The bound in the shift on filtration of those functors is obtained by constructing perturbation data for continuation functors following the same principles as in the construction of the spaces $E^\varepsilon$ mentioned above.

\subsection*{Relation to previous work} The weakly-filtered structure of Fukaya categories in the closed monotone case has been discussed in \cite{Biran2021LagrangianCategories}, which, for the application contained in that paper, did not need to be refined to a genuine filtered structure. The cluster model for the Fukaya category (frist introduced in \cite{Cornea2005ClusterHomology}) has been recently discussed in \cite{Charest2012SourceComplexes} in the case of a single Lagrangian, in \cite{Biran2023TriangulationCategories} in the case of a finite family of transversely intersecting exact Lagrangians, and in \cite{Sheridan2011OnPants} in the exact case. The difference between our cluster model and that contained in \cite{Sheridan2011OnPants}, which was not motivated by the study of filtrations on Fukaya categories, is that we work in the monotone case and do not pick Hamiltonian perturbations for smooth disks, as those would interfere with the filtered structure. Our approach generalizes the methods in \cite{Biran2023TriangulationCategories} by introducing $\varepsilon$-perturbation data and hence allowing to get rid of the restrictive finiteness and transversality assumptions.

\subsection*{Future work} As it will be clear from the construction, the filtered equivalence class of our Fukaya categories depends on the choice of perturbation data (and in particular on the choice of the parameter $\varepsilon>0$). In many applications it is an interesting question to investigate the behaviour of invariants constructed via the filtered structure when $\varepsilon\to 0$. In Section \ref{contfunc} we define continuation functors between Fukaya categories adapting previous work of Sylvan \cite{Sylvan2019OnCategories} to our setting, which we plan to exploit in future work to introduce quantitative invariants of Fukaya categories.

A first, immediate consequence of our construction is clear: working with filtered $A_\infty$-categories is handier than with categories that are only weakly-filtered; in particular this simplifies many proofs of statements that appeal to the weakly-filtered structure of Fukaya categories as in \cite{Biran2021LagrangianCategories}. For these applications, our construction is not an absolutely necessary tool. However, most importantly, our genuinely filtered Fukaya category allows one to state new results about the structure of the Fukaya category itself, mainly due to two reasons: our model is what one would call a `minimal energy' model (due to the fact that we define self-Floer complexes via the pearl model, which carries no Hamiltonian perturbations at all), and it allows nice control over the (positive or negative) shifts in filtration of maps that are defined on the Fukaya category, and some of its associated invariants, such as Hochschild homology, or the so-called open-closed maps.

\subsection*{Organization of the paper.} In Section \ref{preliminaries} we recall the definition of (weakly)-filtered $A_\infty$-categories and functors. We also go over monotone Lagrangians and introduce the combinatorial objects that we will use to describe compactifications of moduli spaces of Floer clusters. In Section \ref{mbfuk} we work out the `cluster' model for Fukaya categories and in Section \ref{wfs} we adapt the work of \cite{Biran2021LagrangianCategories} to show that our Fukaya categories naturally inherit the structure of weakly-filtered $A_\infty$-categories with filtered unit. In Section \ref{actualconstruction} we construct the classes of $\varepsilon$-perturbation data, which ensure that the associated Fukaya categories are genuinely filtered. In Section \ref{addend} we argue why the cluster model seems necessary to have a filtered Fukaya category. In Section \ref{contfunc} we develop the machinery of continuation functors between filtered Fukaya categories, building on \cite{Sylvan2019OnCategories}.

\subsection*{Acknowledgments} I would like to thank Paul Biran for numerous discussions about filtrations and Fukaya categories, and for mentioning the idea of a cluster model to me. Many thanks to the anonymous referee for their careful reading of a first draft of this manuscript and for many valuables suggestions and comments.

\addtocontents{toc}{\protect\setcounter{tocdepth}{2}}
\newpage 
{ \ \small \doublespacing  \tableofcontents }
\newpage
\section{Preliminaries}\label{preliminaries}

\subsection{Filtered $A_\infty$-categories}\label{ainfpre}
In this section we recall the basics about weakly-filtered $A_{\infty}$-categories and weakly-filtered $A_\infty$-functors as introduced in \cite{Biran2021LagrangianCategories}.

Let $(C,d)$ be a chain complex. A filtration on $(C,d)$ is a choice of subspaces $C^{\leq \alpha}\subset C$ for any $\alpha \in \R$ such that $C^{\leq \alpha}\subset C^{\leq \beta}$ for any $\alpha \leq \beta$ and $d(C^{\leq \alpha})\subset C^{\leq \alpha}$ for any $\alpha\in \R$. We refer to the second property by saying that the differential $d$ preserves the filtration $(C^{\leq \alpha)})_\alpha$ on $C$. A chain complex endowed with a filtration is called a filtered chain complex. Consider two filtered chain complexes $(C,d)$ and $(\overline{C}, \overline{d})$, a chain map $\varphi: (C,d)\to (\overline{C},\overline{d})$ between them and a non-negative real number $r\geq 0$. We say that $\varphi$ shifts filtration by $\leq r$ if $\varphi(C^{\leq \alpha})\subset \overline{C}^{\alpha+r}$ for any $\alpha \in \R$. If $\varphi$ shifts filtration by $\leq 0$ we say that it is a filtered chain map.

Consider now a homologically unital $A_\infty$-category $\left(\mathcal{A}, \mu=(\mu_d)_{d\geq 1}\right)$. We will use homological notation as in \cite{Biran2021LagrangianCategories}. Given objects $X,Y\in \text{Ob}(\mathcal{A})$ we write $\mathcal{A}(X,Y):=\text{hom}_\mathcal{A}(X,Y)$ for the morphism space between $X$ and $Y$. We recall that for any $d\geq 1$ the map $\mu_d$ is, for any tuple $X_0, \ldots, X_d$ of objects of $\mathcal{A}$, a linear map $$\mu_d\colon \mathcal{A}(X_0,X_1)\otimes \cdots \otimes \mathcal{A}(X_{d-1}, X_d)\to \mathcal{A}(X_0,X_d)$$ satisfying the $A_\infty$-equation \begin{equation}\label{ainfequation}
    \mu^1\circ\mu^d + \sum_{\substack{i+j=d-1\\i,j\geq 0}}\mu^d(id^i\otimes \mu^1\otimes \id^j) = \sum_{\substack{i+j+k=d+1\\i,j\geq 0, \ k \geq 1}}\mu^{i+j+1}(id^i\otimes \mu^k\otimes id^j)
\end{equation}  for any $d\geq 1$.\\
Let $u^\mathcal{A}\geq 0$ be a non-negative real number and $\boldsymbol{\varepsilon}^\mathcal{A}=(\varepsilon_d^\mathcal{A})_{d\geq 2}$ be a sequence of non-negative real numbers. A weakly filtered structure on $(\mathcal{A},\mu)$ with discrepancy $(\boldsymbol{\varepsilon}^\mathcal{A}, u^\mathcal{A})$ is a choice of filtration on the chain complexes $(\mathcal{A}(X,Y),\mu^1)$ for any two objects $X,Y\in \textnormal{Ob}(\mathcal{A})$ such that:
\begin{enumerate}
    \item for any $d\geq 2$, any tuple $(X_0,\ldots,X_d)$ of objects of $\mathcal{A}$ and any real numbers $\alpha_1,\ldots, \alpha_d\in \R$ we have $$\mu_d\left(\bigotimes_{i=1}^d\mathcal{A}^{\leq \alpha_i}(X_{i-1},X_i)\right)\subset \mathcal{A}^{\leq \sum_{i=1}^d \alpha_i + \varepsilon^\mathcal{A}_d}(X_0,X_d),$$
    \item for any object $X\in \textnormal{Ob}(\mathcal{A})$ there exists a representative $e_X\in \mathcal{A}(X,X)$ of the unit of $X$ such that $$e_X\in \mathcal{A}^{\leq u^\mathcal{A}}(X,X).$$
\end{enumerate}
\noindent An $A_\infty$-category endowed with a weakly-filtered structure will be called a weakly-filtered $A_\infty$-category. A weakly-filtered category $(\mathcal{A},\mu_d)$ is said to be filtered if we can take $u^\mathcal{A}=0$ and $\varepsilon_d^\mathcal{A}=0$ for any $d\geq 2$.

Let $(\mathcal{A}, \mu^\mathcal{A}_d)$ and $(\mathcal{B}, \mu_d^\mathcal{B})$ be filtered $A_\infty$-categories, $\mathcal{F}:\mathcal{A}\to \mathcal{B}$ be an $A_\infty$-functor between them and $\boldsymbol{\varepsilon}^\mathcal{F}:=(\varepsilon^\mathcal{F}_i)_{i\geq 1}$ be a sequence of non-negative real numbers. We say that $\mathcal{F}$ is a weakly-filtered $A_\infty$-functor with discrepancy $\boldsymbol{\varepsilon}^\mathcal{F}$ if for any $d\geq 1$, any tuple $(X_0,\ldots, X_d)$ of objects of $\mathcal{A}$ and any real numbers $\alpha_1,\ldots, \alpha_d\in \R$ we have $$\mathcal{F}_d\left(\bigotimes_{i=1}^d\mathcal{A}^{\leq \alpha_i}(X_{i-1},X_i)\right)\subset \mathcal{B}^{\leq\sum_{i=1}^d\alpha_i+\varepsilon^\mathcal{F}_d}\left(\mathcal{F}(X_0),\mathcal{F}(X_d)\right)$$
\noindent A weakly-filtered functor $\mathcal{F}$ is said to be filtered if we can take $\varepsilon_i^\mathcal{F}=0$ for any $d\geq 1$. $\mathcal{F}$ is said to shift filtration by $\rho\geq 0$ if we can take $\varepsilon_d^\mathcal{F}=d\rho$ for any $d\geq 1$. In \cite{Fukaya2021Gromov-HausdorffTheory}, such functors are called `with energy loss $\rho$'.


\subsection{Monotone Lagrangians}\label{monotonelags}
Let $(M,\omega)$ be a closed symplectic manifold, fixed for the rest of the paper.
Let $L$ be a Lagrangian submanifold of $M$. The Maslov class of $L$ induces a map $$\mu: \pi_2(M,L)\to \Z.$$ We say that $L$ is monotone if there is a positive constant $\tau>0$ such that $$\omega = \tau \mu,$$ where we see $\omega$ as a map on $\pi_2(M,L)$, and if $N_L\geq 2$ where $N_L$ generates the image of $\mu$ in $\Z$. We refer to $\tau$ as the monotonicity constant of $L$.\\
We will denote the standard Novikov field over $\Z_2$ as $$\Lambda:=\left\{ \sum_{k=0}^\infty a_kT^{\lambda_k}: \ a_k\in \Z_2, \ \lambda_k\in \mathbb{R} \ \lim_{k\rightarrow \infty}\lambda_k= \infty\right\}$$ and by $\Lambda_0$ the positive Novikov ring over $\Z_2$, that is $$\Lambda_0:=\left\{ \sum_{k=0}^\infty a_kT^{\lambda_k}: \ a_k\in \Z_2, \ \lambda_k\geq 0, \ \lim_{k\rightarrow \infty}\lambda_k= \infty\right\}$$ 
Let $L$ be a monotone Lagrangian and assume in addition that it is closed. Then for a generic choice of almost complex structure $J$ on $M$ the count of $J$-holomorphic disks with boundary on $L$, Maslov index equal to $2$ and passing through a generic point $p\in L$ weighted by symplectic area is well-defined, and independent from the choice of $J$ and of the point $p\in L$. We denote this count by $\mathbf{d}_L\in \Lambda_0$ (see \cite{Lazzarini2011RelativeCurves}). Let $\mathbf{d}\in \Lambda_0$, then we define $\lagmd(M,\omega)$, where the $m$ stands for \textit{monotone}, as the set of closed, connected and monotone Lagrangians $L$ og $M$ with $\mathbf{d}_L=\mathbf{d}$. Note that if $\mathbf{d}\neq 0$ all Lagrangians in $\lagmd(M,\omega)$ share the same monotonicity constant.


\subsection{Tuples of Lagrangians}\label{tuplesoflags}  In this subsection, we introduce many definitions and notations that will turn out to be very useful when defining our geometric structures in Section \ref{mbfuk} and in Section \ref{contfunc}.

Let $d\geq 1$ and pick a tuple $(L_0,\ldots,L_d)$ of Lagranians in $M$. In the following, we will often denote such a tuple by $\vec{L}$.

\begin{defn}
    We say that $\vec{L}$ is made of cyclically different Lagrangians \label{cicldiff} (or, simply, is a cyclically different tuple) if $L_i\neq L_{i+1}$ for each $i\in \{0,\ldots,d\}$\footnote{Here and in the following definitions like these have to be taken modulo $d$. In this case this means $L_d\neq L_{d+1}=L_0$.}, while we say that it is made of almost cyclically different Lagrangians (or, simply, is an almost cyclically different tuple) if $L_i\neq L_{i+1}$ for any $i\in \{0,\ldots,d-1\}$ but $L_0=L_d$.
\end{defn} 
\begin{defn}\label{defreducedtuple} Assume now that the tuple $\vec{L}$ is not cyclically different: each time there are consecutive indices $i,i+1,\ldots,i+k$ indexing the same Lagrangian, we subtract $k$ from all the indices bigger than $i+k+1$ and write the new tuple as $L_0,\ldots,L_{i-1}, (\overline{L}_i, k), \overline{L}_{i+1},\ldots,\overline{L}_{d-k}$, with $L_{i-1}\neq \overline{L}_i\neq \overline{L}_{i+1}$ (we always work modulo $d+1$); we repeat this process until we get a tuple $$ \vec{L}_\textnormal{red}:=\left((\overline{L}_0,m_0),\ldots,(\overline{L}_{d^R},m_{d^R})\right)$$ of $(d^R+1)$-many cyclically different Lagrangians with multiplicities $m_i\geq 1$, which we will call the \textit{reduced tuple} of $\vec{L}$.
\end{defn}

\noindent Of course the number $d^R$ satisfies $0\leq d^R\leq d$. Notice that the multiplicity $m_0\geq 1$ can be split as a sum $m_0= m_0^b+m_0^e$ where $m_0^b\geq 1$ is the number of subsequent Lagrangians equal to $L_0$ at the beginning of the tuple $\vec{L}$, while $m_0^e\geq 0$ is the number of subsequent Lagrangians equal to $L_0$ at the end of the tuple. \label{defofmis} For notational convenience we will write $\overline{m_0}:=m_0^b, \overline{m_i}:=m_i$ for $i=1,\ldots,d^R$ and $\overline{m_{d^R+1}}:=m_0^e$. In the following we will often omit multiplicities from the notation of the reduced tuple $\vec{L}_\textnormal{red}$. 

\begin{defn}\label{deffundamentaltuple}
    Given $\vec{L}=(L_0,\ldots,L_d)$ we define a tuple $\vec{L^F}=(L_0^F,\ldots, L_{d^F}^F)$ of length $d^F+1$, where $d^F\geq 0$, as follows: $L_0^F:= L_0$ while $L_k^F:=\overline{L_m}$, where $$m:=\min\left\{l\in \{k,\ldots, d^R-1\} \ : \ \overline{L_l}\neq L_j^F \ \forall j\in \{0,\ldots, k-1\}\right\}$$ until there is no Lagrangian left to index. We call $\vec{L^F}$ the \textit{fundamental tuple} of $\vec{L}$.
\end{defn}

\noindent Of course the number $d^F$ satisfies $0\leq d^F\leq d^R\leq d$. In words, given a tuple $\vec{L}$, the reduced tuple $\vec{L}_\textnormal{red}$ is the tuple obtained by merging pairs of equal subsequent Lagrangians in $\vec{L}$, while the fundamental tuple $\vec{L}^F$ is obtained by keeping geometrically different Lagrangians only (for an example see the next Example and Figure \ref{fig:extree}).
\begin{example}
    Let $L_0,L_1,L_2,L_3,L_4$ denote five different Lagrangians in $(M,\omega)$ and consider the tuple $$\vec{L}=\bigl(L_0,L_0,L_1,L_1,L_3,L_3,L_3,L_2,L_4,L_1,L_3,L_2,L_0,L_0\bigr).$$ In this case $d=13$. The corresponding reduced tuple is $$\vec{L}_\textnormal{red} = \bigl((L_0,4), (L_1,2),(L_3,3),(L_2,1),(L_4,1),(L_1,1),(L_3,1),(L_2,1)\bigr)$$ that is, $m_0^b=2$ and $m_0^e=2$. In particular $$\overline{L}_0 = L_0, \ \overline{L}_1=L_1, \ \overline{L}_2= L_3, \ \overline{L}_3=L_2, \ \overline{L}_4=L_4, \ \overline{L}_5=L_1, \ \overline{L}_6=L_3, \overline{L}_7=L_2. $$ The fundamental tuple of $\vec{L}$ is then $$\vec{L}^F= (L_0,L_1,L_3,L_2,L_4)$$ that is $$L_0^F= L_0, \ L_1^F=L_1, \ L_2^F=L_3, \ L_3^F=L_2, \ L_4^F=L_4.$$
    In conclusion, in this case we have $d^F=5< d^R=7<d=13$.
\end{example}

Consider a tuple $\vec{L}=(L_0,\ldots,L_d)$ of Lagrangians in $M$. We will write $$\pi_2(M,\vec{L}):=\pi_2\left(M,\bigcup_{i=0}^dL_i\right)$$ where $\pi_2$ denotes the second fundamental group. Notice that for any subtuple $(\tilde{L_0},\ldots, \tilde{L_k})$ of $\vec{L}$ of any length we have a map $$\pi_2\left(M,\bigcup_{i=0}^k\tilde{L_i}\right)\to \pi_2(M,\vec{L})$$ induced by the inclusion $\cup \tilde{L_i}\to \cup L_i$. In the rest of the paper, we will see each possible $\pi_2(M,\cup \tilde{L_i})$ as a subset of $\pi_2(M,\vec{L})$ and omit the inclusions above.\label{piconvention}


\subsection{Trees}\label{sectiontrees}

Let $d\geq 2$. We define, as in \cite{Seidel2008FukayaTheory}, a $d$-leafed tree to be a properly embedded planar tree $T\subset \R^2$ with $d+1$ semi-infinite edges, which we call exterior edges, one of which is distinguished and called the root of $T$, denoted $e_T$, while the others are numbered clockwise starting from the root and denoted $$e_0(T)=e_T, e_1(T), \ldots,  e_{d}(T)$$ The unique vertex attached to the root will be called the root vertex and denoted by $v_T$. For us, all leafed trees are oriented from the non-root leaves to the root.  For $d=1$, a leafed tree is just an infinite edge, oriented from leaf to root.  \begin{rem}
    Note that, as in 
\cite{Seidel2008FukayaTheory} but contrary to most of the literature, we do not consider leaves to be vertices.
\end{rem}\noindent Notice that a $d$-leafed tree cuts $\R^2$ in $d+1$ connected components, which we number clockwise, starting from the one nearest to the root when moving clockwise. We abuse notation and denote those connected components also by $e_0(T),\ldots,e_{d}(T)$. Next we introduce a bunch of notations and definition for leafed trees. \begin{defn}\label{hugedegfortrees}
    Let $T$ be a $d$-leafed tree.
    \begin{enumerate}
        \item We denote by $V(T)$ the set of its vertices, by $E(T)$ the set of its edges and by $E^\textnormal{int}(T)\subset E(T)$ the subset of edges which are not exterior, and which we call interior.
        \item We write $|v|$ for the number of edges wich are attached to a vertex $v\in V(T)$, and call it the \textit{valency} of $v$, $|v|_\textnormal{int}$ for the number of interior edges attached to $v$ and $|v|_e:=|v|-|v|_\textnormal{int}$ for the number of exterior edges attached to $v$. We denote by $V^i(T)\subset V(T)$ the subset of vertices of $T$ having valency $i$. 
        \item Assume that $T$ has no vertices of valency equal to $1$, then a vertex $v\in V(T)$ touches $|v|$ connected components of $\R^2-T$: we number them clockwise starting from the one associated to the edge attached to $v$ which is nearest to the root and denote them by $$e_0(v),\ldots,e_{|v|-1}(v)$$ for any vertex $v\in V(T)$.
        \item $T$ is called \textit{stable} if the minimal valency of a vertex of $T$ is $3$, and it is called \textit{binary} if each vertex has valency equal to $3$. We denote by $\mathcal{T}^{d+1}$ the space of stable $d$-leafed trees, where two trees are identified if there exists an isomorphism of planar trees between them.
        \item A \textit{flag} of $T$ is a couple $(v,e)\in V(T)\times E(T)$ such that the edge $e$ is attached to the vertex $v$. Given a vertex $v\in V(T)$ we denote by $f_0(v)\in E(T)$ the unique edge that exits from $v$ (with respect to the orientation introduced above), and by $f_1(v),\ldots,f_{|v|-1}(v)\in E(T)$ the remaining edges attached to $v$, ordered in clockwise order starting from $f_0(v)$. Conversely, given an edge $e\in E(T)$ we define $t(e)\in V(T)$ as the start-vertex of $e$ and by $h(e)\in V(T)$ the end-vertex of $e$ (of course, one between $h$ or $t$ is not defined for exterior edges). We denote by $F(T)$ the set of flags of $T$ and by $F^\textnormal{int}(T)\subset F(T)$ the subset of flags made of interior edges. 
        \item A metric on $T$ is a map $\lambda: E(T)\to [0,\infty]$ such that  $$\lambda(e_i(T))=\infty \textnormal{ for any }i=0,\ldots,d \textnormal{   and   } \lambda(e)<\infty \textnormal{ for any }e\in E^\textnormal{int}(T)$$ We call a couple $(T,\lambda)$ a metric tree and we denote by $\lambda(T)$ the space of metrics of $T$. We also define the space $\overline{\lambda(T)}$ of maps $\lambda:E(T)\to [0,\infty]$ such that $$\lambda(e_i(T))=\infty \textnormal{ for any }i=0,\ldots,d$$ Let $k\in \Z$, then we define $\lambda^k(T)\subset \overline{\lambda(T)}$ to be the space of metrics with exactly $k$ interior edges of infinite length. Note that if $T\in \mathcal{T}^{d+1}$ then $\lambda^k(T)=\emptyset$ for any $k>d-1$.
        \item If $T'\subset T$ is a subtree of $T$ and $\lambda\in \lambda(T)$ is a metric on $T$, then $\lambda$ induces a metric on $T'$, which we still denote by $\lambda$, and call $(T',\lambda)$ a metric subtree of $(T,\lambda)$, even if $T'$ itself is not leafed or unstable. We denote by $\lambda(T')$ the space of metrics on $T'$ in the sense above, and by $\overline{\lambda(T')}$ the obvious analogous of $\overline{\lambda(T)}$.\\
    \end{enumerate}
\end{defn}

\subsection{Trees with Lagrangian labels}\label{sectionlagtrees} Let $\vec{L}=(L_0,\ldots,L_d)$ be a tuple of Lagrangians in $\lagmd(M,\omega)$ and consider a $d$-leafed tree $T$. 

\begin{defn}
The assignement $L_i\mapsto e_i(T)$, where $e_i(T)$ is seen as the $i$th connected component of $\R^2\setminus T$ in the sense of Section \ref{sectiontrees}, is called a \textit{labelling} by $L$ for $T$. We denote by $\mathcal{T}^{d+1}(\vec{L})$ the space of \textit{stable} $d$-leafed trees labelled by $\vec{L}$.
\end{defn} 
\noindent If $\vec{L^F}={L}$, i.e $\vec{L}$ contains the same Lagrangian $d+1$ times, we simply write $\mathcal{T}^{d+1}(L)$ for $\mathcal{T}^{d+1}(\vec{L})$. \\
Let now $T\in \mathcal{T}^{d+1}(\vec{L})$ be a labelled tree. Any vertex $v\in V(T)$ naturally inherits a labelling $\vec{L}_v$ by the conventions above\footnote{ That is, $L_i\in \vec{L}_v$ if and only if $v$ touches the connected component $e_i(T)$.}. Any edge $e\in E(T)$ also inherits a labelling by a couple of Lagrangians in an obvious way.
\begin{defn}
    An edge $e\in E(T)$ is said to be \textit{unilabelled} if it is labelled by a couple of equal Lagrangians. 
\end{defn}
\noindent We introduce the following notations for various subsets of the set of edges $E(T)$ of $T$:
\begin{itemize}
    \item we denote by $E_U(T)\subset E(T)$ the set of unilabelled edges of $T$,
    \item we denote by $E_F(T):=E(T)\setminus E_U(T)$ the set of non-unilabelled edges of $T$,
    \item for any $i\in\{0,\ldots, d^F\}$, we denote by $E_i(T)\subset E(T)$ the set of edges of $T$ that are (uni)labelled by the Lagrangian $L_i^F$ (see Definition \ref{deffundamentaltuple} for the definition of $\vec{L}_F$ and $d^F$),
    \item we set $E_U^\textnormal{int}(T):=E_U(T)\cap E^\textnormal{int}(T)$, $E_F^\textnormal{int}:=E_F(T)\cap E^\textnormal{int}(T)$ and $E_i^\textnormal{int}(T):=E_i(T)\cap E^\textnormal{int}(T)$ for any $i\in \{0,\ldots, d^F\}$
\end{itemize}

\begin{rem}
    The concept of unilabelled edge will be relevant for the following reason: as in \cite{Seidel2008FukayaTheory} in order to efficiently deal with Floer-curves in the definition of the Fukaya category in Section \ref{mbfuk} we will use the language of trees; however, in our case those curves will be a combination of Floer-polygons (to use the terminology from \cite{Seidel2008FukayaTheory}( and trees of Morse trajectories, where the latter will be controlled precisely by trees of unilabelled edges. 
\end{rem}\begin{defn}
    We define $\mathcal{T}^{d+1}_U(\vec{L})\subset \mathcal{T}^{d+1}(\vec{L})$ to be the set of stable trees labelled by $\vec{L}$ with only unilabelled interior edges, that is $$\mathcal{T}^{d+1}_U(\vec{L}):=\left\{ T\in \mathcal{T}^{d+1}(\vec{L}): \ E^\textnormal{int}_U(T)=E^\textnormal{int}(T)\right\}.$$Given a labelled tree $T\in \mathcal{T}^{d+1}(\vec{L})$ we define $T_\textnormal{red}$ to be $T$ with unilabelled edges removed.
 \end{defn}   \begin{defn}
    A metric $\lambda\in \lambda(T)$ on a labelled tree $T\in \mathcal{T}^{d+1}(\vec{L})$ is said to be unilabelled if $\lambda(e)=0$ for any non-unilabelled interior edge $e\in E^\textnormal{int}(T)$. We denote by $\lambda_U(T)$ the space of unilabelled metrics on $T$,  by $\overline{\lambda_U(T)}$ its closure in $\overline{\lambda(T)}$ and set $\lambda^k_U(T):=\lambda_U(T)\cap \lambda^k(T)$ for any $k$, where $\lambda^k(T)\subset\overline{\lambda(T)}$ is the subspace of metrics with exactly $k$ interior edges with infinite length (see Definition \ref{hugedegfortrees}). 
\end{defn} 
\noindent The stability requirement in the definition of trees in the subset $\mathcal{T}^{d+1}_U(\vec{L})$ will be crucial in order to deal with breaking of Floer-clusters correctly in Section \ref{mbfuk}. 
Note that if $T\in \mathcal{T}^{d+1}_U(\vec{L})$ is unilabelled, then $\lambda_U(T)=\lambda(T)$.
 In general, the subtree $T_\textnormal{red}$ may not be connected and its connected components may be unstable leafed trees. Notice that if $T$ lies in $\mathcal{T}^{d+1}_U(\vec{L})$, then $T_\textnormal{red}$ obviously does not have any interior edges. 
 Anyway, the fact that $T$ is a planar connected tree implies that the numbering of leaves of $T$ induces a unique numbering $$e_0(T_\textnormal{red}),\ldots,e_{d^R}(T_\textnormal{red})$$ of the leaves of (the subtrees of) $T_\textnormal{red}$, even if the latter is not connected (recall that the number $d^R$ comes from the definition of the reduced tuple $\vec{L}_\textnormal{red}$ of $\vec{L}$, see Definition \ref{defreducedtuple}). Assume now that $T\in \mathcal{T}_U^{d+1}(\vec{L})$, then  the connected components (i.e. the vertices) of $T_\textnormal{red}$ induce a splitting $$\{0,\ldots,d^R\}=\bigcup_{v\in V(T_\textnormal{red})}\Lambda_v$$ where $i\in \Lambda_v$ \label{Lambdav}  if and only if $\overline{L_i}\in \vec{L_v}$. Note that the above union is not disjoint in general. \\
 We have the following decomposition of $T\setminus T_\textnormal{red}$ indexed by the fundamental tuple $\vec{L}_F$ \label{fundamentaldecomposition}of $\vec{L}$ (see Definition \ref{deffundamentaltuple} for the definition of the fundamental tuple $\vec{L}_F$ of $\vec{L}$ and of the number $d^F$): $$T\setminus T_\red = \bigcup_{i=0}^{d^F} T_j^F$$ where for any $j\in \{0,\ldots, d^F\}$, $T_j^F$ is the union of all the subtrees of $T\setminus T_\red$ with edges unilabelled by $L_j^F$.
 
 \begin{defn}
     Given a labelled tree $T\in \mathcal{T}^{d+1}(\vec{L})$ we call the decomposition $$T=T_\red \cup \bigcup_{i=0}^{d^F} T_j^F$$ described above the \textit{fundamental decomposition} of the labelled tree $T$. Moreover, we define $T_\textnormal{uni}$ as the union of the trees $T^F_j$ from the fundamental decomposition of $T$.
 \end{defn}  \noindent Although $T_\textnormal{uni}$ is not a leafed tree in general, the planar structure of $T$ induces the following ordering of edges of $T_\textnormal{uni}$ which are exterior edges of $T$: we denote by $e^i_j(T_\textnormal{uni})$ as the $j$th edge (in clockwise order) of the subtree with label $\overline{L_i}$, that is
$$e_j^i(T_\textnormal{uni}):=e_{\sum_{k=0}^{i-1}\overline{m_k}+j}(T)$$ (see page \pageref{defofmis} for the definition of $\overline{m_k}$). In Figure \ref{fig:extree} we sketched an example of a labelled tree and of its reduced tree.
 \begin{center}
     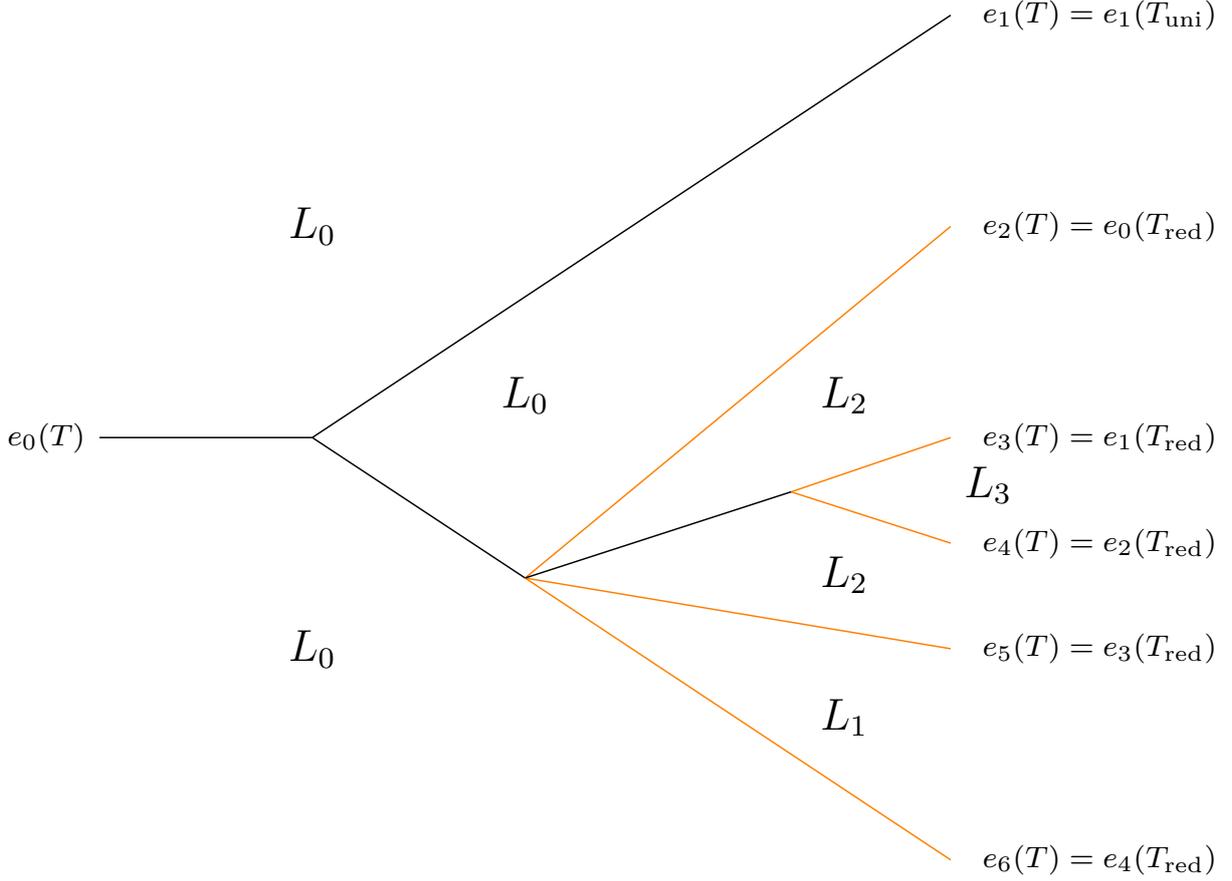
\begin{figure}
 \scalebox{1.4}{\begin{tikzpicture}
    \node at (0,-2) {$L_0$}; \node at (5,-2.65) {$L_1$}; \node at (5,-1.3) {$L_2$}; \node at (5, 0.4) {$L_2$}; \node at (6.35, -0.45) {$L_3$}; \node at (2,0.4) { $L_0$}; \node at (0,2) {$L_0$}; \node at (-2.5,0) {\tiny $e_0(T)$}; \node at (7.4,4) {\tiny $e_1(T)=e_1(T_\textnormal{uni})$}; \node at (7.4,2) {\tiny $e_2(T)=e_0(T_\textnormal{red})$}; \node at (7.4,0) {\tiny $e_3(T)=e_1(T_\textnormal{red})$}; \node at (7.4,-1) {\tiny $e_4(T)=e_2(T_\textnormal{red})$}; \node at (7.4,-2) {\tiny $e_5(T)=e_3(T_\textnormal{red})$}; \node at (7.4,-4) {\tiny $e_6(T)=e_4(T_\textnormal{red})$}; \draw[thick] (-2,0) -- (0,0); \draw[thick] (0,0) -- (6,4); \draw[thick] (0,0) -- (2,-1.33); \draw[thick,color=orange] (2,-1.33) -- (6,-4); \draw[thick, color=orange] (2,-1.33) -- (6,2); \draw[thick,color=orange] (2,-1.33) -- (6,-2); \draw[thick, color=orange] (4.5,-0.514) -- (6,-1); \draw[thick, color=orange] (4.5,-0.514) --(6,-0); \draw[thick] (2,-1.33)--(4.5,-0.514); \end{tikzpicture}}
         \caption{An element $T\in\mathcal{T}^{7}(\vec{L})$ for $\vec{L}=(L_0, L_0,L_2,L_3,L_2,L_1,L_0)$. We colored in orange the components of $T_\textnormal{red}$.  Note that in this case $\vec{L_\textnormal{red}}=((L_0,2+1),L_2,L_3,L_2,L_1)$ and $\vec{L^F}= (L_0,L_2,L_3,L_1)$.} \label{fig:extree}
     \end{figure}
    \end{center}
\begin{defn}
    We define $\lambda^{d+1}(\vec{L})$ as the space of metric trees $(T,\lambda)$ (see Section \ref{sectiontrees}), where $T\in \mathcal{T}^{d+1}(\vec{L})$ is a labelled tree and $\lambda \in \lambda(T)$, up to the relation that identifies identical metric trees, i.e. $(T_1,\lambda_1)\sim (T_2,\lambda_2)$ if and only if there is a planar isomorphism $$\varphi: T_1\setminus \{e\in E(T_1): \ \lambda_1(e)=0\}\to T_2\setminus\{e\in E(T_2): \ \lambda_2(e)=0\}$$ such that $\lambda_2(\varphi(e))=\lambda_1(e)$ for any remaining edge. 
\end{defn}
\begin{rem}
    Note that here we use the letter $\lambda$ to denote a set of metric trees, and not only of metrics on a tree.
\end{rem}
\noindent We will often write elements of $\lambda^{d+1}(\vec{L})$ simply as $(T,\lambda)$. Note that $\lambda^{d+1}(\vec{T})$ has an obvious conformal structure. We define $\lambda_U^{d+1}(\vec{L})\subset \lambda^{d+1}(\vec{L})$ as the subspace containing metric trees which can be represented by a tuple $(T,\lambda)$ with $\lambda\in \lambda_U(T)$, or equivalently $T\in \mathcal{T}_U^{d+1}(\vec{L})$. Notice that the space $\lambda_U^{d+1}(\vec{L})$ may of course have boundary, depending on the nature of the tuple $\vec{L}$.\\


\subsection{Systems of ends for trees with Lagrangian labels}\label{seconsystemofends} We introduce the notion of system of ends (originally introduced in a different form in \cite{Charest2012SourceComplexes})\label{sistemofends}. 

\begin{defn} Let $\vec{L}=(L_0,\ldots, L_d)$ be a tuple of Lagrangians, $T\in \mathcal{T}^{d+1}(\vec{L})$ a labelled tree and $\lambda\in \lambda_U(T)$ a unilabelled metric on $T$.
    \begin{itemize}
        \item A system of ends for the metric tree $(T,\lambda)$ is a map $$s_{T,\lambda}: E(T)\to [0,\infty)$$ such that for any edge $e\in E(T)$ of $T$ it satisfies $s_{T,\lambda}(e)<\frac{\lambda(e)}{2}$ if $\lambda(e)>0$ and $s_{T,\lambda}(e)=0$ otherwise.
        \item  A system of ends for $T$ is a map $$s_T: E(T)\times \lambda_U(T)\to \R$$ such that $s_{T,\lambda}:=s_T(\lambda)$ is a system of ends for $(T,\lambda)$ and $s_T(e):\lambda(T)\to \R$ is smooth for any $e\in E_U(T)$.
        \item A system of ends for $\vec{L}$ is a smooth map\footnote{Here the map $s$ is well-defined if we think of any element of $\lambda_U^{d+1}$ as represented by a tree with $d-2$ internal edges, which is possible by definition (indeed, recall that we are dealing with stable trees here). The fact that $s$ lands in $\R^{2d-1}$ is due to the presence of the $d+1$ exterior edges.} $$s: \lambda^{d+1}_U(\vec{L})\to \R^{2d-1}$$ such that $s_{T,\lambda}:=s(T,\lambda)$ is a system of ends for any $(T,\lambda)\in \lambda^{d+1}_U(\vec{L})$.
        \item A universal choice of system of ends is a choice of system of ends for any tuple $\vec{L}$ of Lagrangians of any length $d+1$.
    \end{itemize}
\end{defn}

We recall that, given a tree $T$, in Section \ref{sectiontrees} we defined the space $\overline{\lambda(T)}$ by allowing metrics to take infinite value away from exterior edges. In the following, we will see an interior edge of $T$ with infinite length as a `breaking' of $T$ into two leafed trees. We will now describe gluing of genuine labelled trees (that is, not endowed with metrics). Let $\vec{L_1}$ and $\vec{L_2}$ be tuples of Lagrangians of length $d_1+1$ and $d_2+1$ respectively.
 \begin{defn} \label{compatibilitydef}
    We say that the tuple $\vec{L_1}$ is compatible with the tuple $\vec{L_2}$ if one of the following (mutually exclusive) conditions hold:
    \begin{enumerate}
        \item\label{firstcomp} we have $m_0^e(\vec{L_1})=0$ (see page \pageref{defofmis}) and there is $i\in \{1,\ldots, d_2\}$ such that $$L_0^1=L_i^2 \textnormal{   and   }L_{i+1}^2=L_{d_1}^1,$$
        \item\label{secondcomp} we have $m_0^e(\vec{L}_1)>0$, that is in particular $L_0^1=L_{d_1}^1$, and there is $i\in\{1,\ldots,d_2\}$ such that $$L_0^1=L_i^2=L_{i+1}^2=L_{d_1}^1.$$ 
    \end{enumerate}
    In both cases, we call such an $i\in \{1,\ldots, d_2\}$ where the compatibility is verified, an admissible index.
\end{defn}
\begin{rem}
    Note that formally the notion of compatibility is not symmetric.
\end{rem}
\noindent Assume that $\vec{L_1}$ is compatible with $\vec{L_2}$ and pick labelled trees $T_1\in \mathcal{T}^{d_1+1}(\vec{L_1})$ and $T_2\in \mathcal{T}^{d_2+1}(\vec{L_2})$. Let $i\in \{1,\ldots, d_2\}$ be an admissible index, then we define the tree $T_1\#_iT_2$ as the tree obtained by gluing the root $e_0(T_1)$ of $T_1$ to the $i$th exterior edge $e_i(T_2)$ of $T_2$ g. We call $T_1\#_iT_2$ the tree obtained by gluing $T_1$ to $T_2$ at the admissible index $i$. Note that we have $T_1\#_iT_2\in \mathcal{T}^{d_1+d_2+1}(\vec{L_1}\#_i\vec{L_2})$, where \label{cancelletto} $$\vec{L^1}\#_i\vec{L^2}:=\left(L_0^2,\ldots, L_i^2,L_1^1,\ldots, L_{d_1-1}^1,L_{i+1}^2,\ldots, L_{d_2}^2\right).$$ Moreover, we can see $T_1\setminus e_0(T_1)$ and $T_2\setminus e_i(T_2)$ as subsets of $T_1\#_iT_2$ in an obvious way; the new edge resulting from the gluing we just described will be denoted by $e_g\in E(T_1\#_iT_2)$.\\

We next describe gluing of metrics on labelled trees. Let $\lambda_1\in \lambda(T_1)$ and $\lambda_2\in \lambda(T_2)$ and consider $\rho\in [-1,0)$. Suppose again that $\vec{L_1}$ is compatible with $\vec{L_2}$ and let $i\in \{1,\ldots, d_2\}$ be an admissible index. Then we can define a metric $$\gamma^{T_1,T_2;i}(\rho, \lambda_1,\lambda_2)\in \lambda(T_1\#_iT_2)$$ on the labelled tree $T_1\#_iT_2$ via $$\gamma^{T_1,T_2;i}(\rho, \lambda_1,\lambda_2)|_{T_1\setminus e_0(T_1)} :\equiv \lambda_1, \ \ \gamma^{T_1,T_2;i}(\rho, \lambda_1,\lambda_2)|_{T_2\setminus e_i(T_2)} :\equiv \lambda_2 \textnormal{   and   } \gamma^{T_1,T_2;i}(\rho, \lambda_1,\lambda_2)(e_g):=-\ln(-\rho)$$ This way we defined for any admissible $i\in\{0,\ldots, d_2\}$ a map $$\gamma^{T_1,T_2;i}: [-1,0)\times \lambda(T_1)\times \lambda(T_2)\to \lambda(T_1\#_iT_2)$$ Note that such maps extend to maps $$\overline{\gamma^{T_1,T_2;i}}:  [-1,0]\times \lambda(T_1)\times \lambda(T_2)\to\overline{\lambda(T_1\#_iT_2)}$$ by declaring $\overline{\gamma^{T_1,T_2;i}}|_{0\times \lambda(T_1)\times \lambda(T_2)}$ to be the trivial gluing. This explains what we meant above by seeing interior edges of infinite length as `broken edges'.\\
Let now $\vec{L}$ be a tuple of Lagrangians of length $d+1$. By considering all the decompositions of all trees $T\in \mathcal{T}^{d+1}(\vec{L})$ into two and more leafed trees and by packing all the maps of the form $\overline{\gamma^{T_1,T_2;i}}$ as above we get boundary charts for $\overline{\lambda^{d+1}(\vec{L})}$ by looking at `small enough neighbourhoods of the trivial gluing', similarly to what is done for moduli spaces of punctured disks in \cite{Seidel2008FukayaTheory}. We skip the details of this construction. As it will be apparent from the constructions in Section \ref{mbfuk} we are particularly interested in $\overline{\lambda_U^{d+1}(\vec{L})}$, that is the closure of $\lambda_U^{d+1}(\vec{L})$ in $\overline{\lambda^{d+1}(\vec{L})}$, which inherits conformal structure and boundary charts from $\overline{\lambda^{d+1}(\vec{L})}$.\\

We can now introduce the notion of consistency for system of ends.

\begin{defn}
    A system of ends $s$ on $\vec{L}$ is said to be consistent if it extends smoothly to a map $\overline{s}$ on $\overline{\lambda^{d+1}_U(\vec{L})}$.
\end{defn}\noindent Note that the fact that our trees are labelled is not of central importance for the notion of consistency. In particular, the following result can be proved by constructing an explicit system of ends as in \cite{Charest2012SourceComplexes}.
\begin{lem}
    Universal choices of consistent system of ends exist.
\end{lem}
\begin{proof}
An example is constructed in \cite[Section 1.3]{Charest2012SourceComplexes}.
\end{proof}

Lastly, given a metric labelled tree $T\in \mathcal{T}^{d+1}(\vec{L})$ and a metric $\lambda\in \overline{\lambda(T)}$ on it, we will identify edges with intervals according to the metric and the orientation described above, that is:
\begin{itemize}
    \item The non-root leaves will be identified with $(-\infty,0]$, while the root with $[0,\infty)$;
    \item interior edges $e\in E^\textnormal{int}(T)$ with $\lambda(e)<\infty$ will be identified with $[0,\lambda(e)]$;
    \item interior edges $e\in E^\textnormal{int}(T)$ with $\lambda(e)=\infty$ will be identified with the disjoint union $[0,\infty)\sqcup(-\infty,0]$.
\end{itemize}
Given $e\in E(T)$ we will often write the point $t$ in the interval representation of $e$ as $e(t)$ for notational convenience. Given a system of ends $s$ on $\vec{L}$, we will abuse notation and often identiy $s_T(e)\in \mathbb R$ as an interval in the following way: $$s_T(e) = \left[\frac{\lambda(e)}{2}-s_T(e), \frac{\lambda(e)}{2}+s_T(e)\right]\subset e$$ if $e$ has finite length, $s_T(e) = [s_T(e), \infty)\subset e$ if $e$ is the root $e_0(T)$ and $s_T(e)=(-\infty, -s_T(e)]\subset e$ if $e$ is a non-root exterior edge.

\newpage

\section{A Morse-Bott model for $\mathcal{F}uk(X)$ and its weakly filtered structure}\label{mbfuk}
In this section, we construct a Morse-Bott model for the Fukaya category. Similar construction have appeared in \cite{Cornea2005ClusterHomology, Sheridan2011OnPants, Charest2012SourceComplexes}. The idea of using such a model to construct filtered Fukaya categories already appeared in \cite{Biran2023TriangulationCategories}. We fix once and for all a closed and connected symplectic manifold $(M,\omega)$. Recall from Section \ref{monotonelags} that given a positive Novikov serie $d\in \Lambda_0$ the set $\lagmd(M,\omega)$ consists of closed, connected and monotone Lagrangians $L\subset M$ with $d_L:=d$.


\subsection{Source spaces: moduli spaces of clusters}\label{sourcespaces}
Let $d\geq 2$. We recall the definition of moduli spaces of (configurations of) disks with $d$ boundary marked points $\mathcal{R}^{d+1}$ following \cite[Chapter 9]{Seidel2008FukayaTheory}.\\ We denote by $D:=\{z\in \mathbb C \ : \ |z|\leq 1\}$ and $\partial D:=\{z\in \mathbb C \ : \ |z|=1\}$ the unit disk and the unit circle in the complex plane. We define $$\textnormal{conf}_{d+1}(\partial D)\subset (\partial D)^{d+1}$$ as the space of ordered tuples of $d+1$ distinct points on $\partial D$. An element of $\textnormal{conf}_{d+1}(\partial D)$ will be usually denoted as $(z_0,\ldots,z_d)$. We then define $$\mathcal{R}^{d+1}:=\frac{\textnormal{conf}_{d+1}(\partial D)}{\textnormal{Aut}(D)}$$ where $\textnormal{Aut}(D)$ denotes the group of automorphisms of the unit disk $D$, which acts on $\textnormal{conf}_{d+1}(\partial D)$ in the standard way. Recall that $\mathcal{R}^{d+1}$ is a smooth manifold of dimension $d-2$ and admits a compactification into a manifold with corners $\overline{\mathcal{R}^{d+1}}$ which realizes Stasheff's associahedron.

Let $\vec{L}=(L_0,\ldots,L_d)$ be a tuple of Lagrangians in $\mathcal{L}^{m,\textbf{d}}(M,\omega)$.

\begin{defn}
    We denote by $\mathcal{R}^{d+1}(\vec{L})$ the space of disks $r\in \mathcal{R}^{d+1}$ equipped with the Lagrangian label $\vec{L}$, that is, for any $i=0,\ldots,d$ we view the boundary arc of $r$ between $z_i$ and $z_{i+1}$ as labelled by $L_i$. An element $r\in \mathcal{R}^{d+1}$ will be called a (labelled) disk configuration.
\end{defn}\noindent The difference between $\mathcal{R}^{d+1}(\vec{L})$ and $\mathcal{R}^{d+1}$ is hence purely formal, but the a-priori specification of Lagrangian labels will be of help when defining moduli spaces of clusters.
\begin{defn}\label{defoftype}
    Let $r\in \mathcal{R}^{d+1}(\vec{L})$ be a disk configuration: the marked point $z_i$ will be called of type I if $L_{i-1}\neq L_i$, and of type II otherwise.
\end{defn}
\noindent Over $\mathcal{R}^{d+1}(\vec{L})$ we have a bundle $$\pi^{d+1}(\vec{L}): \mathcal{S}^{d+1}(\vec{L})\to \mathcal{R}^{d+1}(\vec{L})$$ where the fiber $S_r:=\left(\pi^{d+1}(\vec{L})\right)^{-1}(r)$ over $r\in \mathcal{R}^{d+1}(\vec{L})$ is the equivalence class of punctured disk representing the configuration $r$, with punctures at points of type I and smooth marked points at points of type II (see \cite{Seidel2008FukayaTheory} for the formal definition of disk with punctures). We call the family $\left(\pi^{d+1}(\vec{L})\right)_{d,\vec{L}}$ over all tuples of Lagrangians in $\lagmd(M,\omega)$ of any length a \textit{universal family of disks}.\\

In the remaining of this subsection we will describe a partial compactification for $\pi^{d+1}(\vec{L})$, which requires the notion of strip-like ends. We briefly recall from \cite{Seidel2008FukayaTheory} what strip-like ends are. Let $S$ be a punctured disk and let $z$ be either a point on $\partial S$ or a puncture (viewed in the compactification of $S$). A positive strip-like\label{striplikends} end for $S$ at $z$ is a proper holomorphic embedding $\epsilon: [0,\infty)\times [0,1]\to S$ such that $$\epsilon^{-1}(\partial S)= [0,\infty)\times \{0,1\} \text{    and   }\lim_{s\to \infty}\epsilon(s,\cdot)= z$$ A negative strip-like end is a strip-like end modeled on $(-\infty,0]$.  A choice of strip-like ends for $\pi^{d+1}(\vec{L})$ consists of proper embeddings $\epsilon_0: \mathcal{R}^{d+1}(\vec{L})\times [0,\infty)\times [0,1]\to \mathcal{S}^{d+1}(\vec{L})$ and $\epsilon_i: \mathcal{R}^{d+1}(\vec{L})\times (-\infty,0]\times [0,1]\to \mathcal{S}^{d+1}(\vec{L})$ for $i\in \{1,\ldots, d\}$ that restrict to strip-like ends on fibers at the associated $z_i(r)$ such that the images are pairwise disjoint. Given a choice $(\epsilon_0,\ldots, \epsilon_d)$ of strip-like ends on $\pi^{d+1}(\vec{L})$ we denote by $(\overline{\epsilon_1}, \ldots, \overline{\epsilon_{d^R}})$ the set of negative strip-like ends at the points $z_i$ of type I (i.e. at punctures). We will often omit strip-like ends from the notation and identify half-strips as subsets of disks. Moreover, for any $c\geq 0$ we will write $|s|\geq c$ for the subset $[c,\infty]\times [0,1]$ (resp. $(-\infty,-c]\times [0,1]$) of the standard positive half-strip (resp. negative half-strip). There is a notion of consistency for a universal choice of strip-like ends on the universal family $\left(\pi^{d+1}(\vec{L})\right)$ which requires strip-like ends to be compatible with breaking and gluing of disks (see \cite[Section 9g]{Seidel2008FukayaTheory}). We do not recall here the definition but require any of our choices of strip-like ends to be consistent.

Let $T\in \mathcal{T}^{d+1}(\vec{L})$ be a $d$-leafed tree labeled by $\vec{L}$. Recall (see Section \ref{sectiontrees}) that elements of $\mathcal{T}^{d+1}(\vec{L})$ are assumed to be stable (that is, with no vertices of valency less than $3$). A $T$-disk configuration is a tuple $$\left((r_v)_{v\in V(T)}, (z_{h(e)},z_{t(e)})_{e\in E^\textnormal{int}(T)}\right)$$ where $r_v\in \mathcal{R}^{|v|}(\vec{L_v})$ and $z_{h(e)}$ (resp. $z_{t(e)})$ is a distinguished point in the tuple $r_{h(e)}$ (resp. $r_{t(e)}$). We denote by $\mathcal{R}^T$ the space of $T$-disk configurations and we set $$\overline{\mathcal{R}^{d+1}(\vec{L})}:=\bigcup_{T\in \mathcal{T}^{d+1}(\vec{L})}\mathcal{R}^T.$$
The following result is Lemma 9.2 in \cite{Seidel2008FukayaTheory}.
\begin{lem}\label{disc2}
    The space $\overline{\mathcal{R}^{d+1}(\vec{L})}$ admits the structure of a smooth manifold with corners. Moreover, it realizes Stasheff's $(d-2)$-associahedron.
\end{lem}
\noindent The partial compatictification of $\pi^{d+1}$ is  $$\overline{\pi^{d+1}(\vec{L})}: \overline{\mathcal{S}^{d+1}(\vec{L})}\to \overline{\mathcal{R}^{d+1}(\vec{L})}$$ where fibers are disjoint unions of nodal disks representing elements of the base.\\

We mimic an idea contained in \cite{Charest2012SourceComplexes} in order to construct moduli spaces of cluster of punctured disks with marked points. Basically, what we do to define source spaces for the Floer maps defining the $A_\infty$-maps of our Fukaya category is adding a collar neighbourhood to certain boundary components of $\overline{\mathcal{R}^{d+1}(\vec{L})}$. Recall from Section \ref{sectionlagtrees} that given a labelled tree $T\in \mathcal{T}^{d+1}(\vec{L})$ we defined the space $$\lambda_U(T)=\{\lambda\in \lambda(T) \ : \ \lambda(e)= 0 \textnormal{ for any non unilabelled }e\in E^\textnormal{int}(T)\}.$$ Moreover, we defined $\mathcal{T}^{d+1}_U(\vec{L})\subset \mathcal{T}^{d+1}(\vec{L})$ as the subset of labelled trees with no non unilabelled interior edges. We define $$\mathcal{R}^{d+1}_C(\vec{L}):=\bigcup_{T\in \mathcal{T}^{d+1}_U(\vec{L})}\mathcal{R}^T\times \lambda(T).$$ An element of $\mathcal{R}^{d+1}_C(\vec{L})$ will be usually denoted as $(r,T,\lambda)$, where $r\in \mathcal{R}^T$ is a $T$-disk configuration and $\lambda\in \lambda(T)$ is a metric on the labelled tree $T$. Notice that $\mathcal{R}^{d+1}_C(\vec{L})$ is the interior of $$\bigsqcup_{T\in \mathcal{T}^{d+1}(\vec{L})} \mathcal{R}^T\times \lambda_U(T)$$ We also define $$\overline{\mathcal{R}^{d+1}_C(\vec{L})}:= \bigsqcup_{T\in \mathcal{T}^{d+1}(\vec{L})} \mathcal{R}^T\times \overline{\lambda_U(T)}.$$
\noindent To help the reader understand the difference between the various moduli spaces defined above, we add a graphical description of the construction of part of the boundary of $\mathcal{R}^{d}_C(\vec{L})$ in the case $d=4$ for a particular choice of tuple of Lagrangians in Figure \ref{fig:enter-label}.


\definecolor{ffzzqq}{rgb}{1,0.6,0}
\definecolor{ffqqqq}{rgb}{1,0,0}
\definecolor{qqqqff}{rgb}{0,0,1}
\definecolor{ffzzqq}{rgb}{1,0.6,0}
\definecolor{ffqqqq}{rgb}{1,0,0}
\definecolor{ffffff}{rgb}{1,1,1}
\hspace{-15cm}
\begin{figure}
    \centering
    \captionsetup[subfigure]{oneside,margin={0cm,-0.3cm}}
    \begin{subfigure}[b]{0.3\textwidth}
    \centering 
    \begin{tikzpicture}[line cap=round,line join=round,>=triangle 45,x=1cm,y=1cm]
\clip(-8.625682348780519,-1.0233369838667696) rectangle (-2.667751316241792,4.650381790935983);
\fill[line width=0pt,color=ffzzqq,fill=ffzzqq,fill opacity=0.1] (-3.28,-0.17) -- (-8.28,-0.17) -- (-9.825084971874734,-4.925282581475764) -- (-5.78,-7.86420884293813) -- (-1.7349150281252645,-4.925282581475766) -- cycle;
\draw [line width=1pt,color=ffqqqq] (-8.28,-0.17)-- (-3.28,-0.17);
\draw [line width=1pt,color=ffqqqq] (-3.28,-0.17)-- (-3.1256541807559914,-0.6455810846552836);
\draw [line width=1pt,color=ffqqqq] (-8.28,-0.17)-- (-8.427815730455455,-0.6476510335275327);
\draw [line width=1pt] (-7.260720578460138,0.8494742492285339)-- (-6.260720578460138,0.8494742492285339);
\draw [line width=1pt] (-6.260720578460138,0.8494742492285339)-- (-5.555826121443226,1.5587864544434398);
\draw [line width=1pt] (-6.260720578460138,0.8494742492285339)-- (-5.566086589378064,0.13011090486636423);
\draw [line width=1pt] (-5.9134035839191,0.48979257704744905)-- (-5.553582022437874,0.836964643742355);
\draw [line width=1pt] (-5.9134035839191,0.48979257704744905)-- (-5.4134035839191,0.48979257704744905);
\draw (-6.751986723706502,1.4556248771822508) node[anchor=north west] {\tiny $L_0$};
\draw (-6.741460343790355,0.7819359533455118) node[anchor=north west] {\tiny $L_1$};
\draw (-6.138297190666371,1.1903604688174016) node[anchor=north west] {\tiny $L_1$};
\draw (-5.620400273517401,0.9082531215426532) node[anchor=north west] {\tiny $L_2$};
\draw (-5.630927262636936,0.497724914016343) node[anchor=north west] {\tiny $L_3$};
\end{tikzpicture}
\caption{\small A detail of a graphical description of the moduli space $\overline{\mathcal{R}^5(\vec{L})}$ (cfr. \cite{Seidel2008FukayaTheory}).}
\end{subfigure}
\hspace{1cm}
\captionsetup[subfigure]{oneside,margin={0.8cm,-1.3cm}}
\begin{subfigure}[b]{0.3\textwidth} 
 \centering 
\begin{tikzpicture}[line cap=round,line join=round,>=triangle 45,x=1cm,y=1cm]
\clip(-9.543142683810132,-1.0617529782629214) rectangle (-0.7450823198346074,7.316611926653659);
\fill[line width=0pt,color=ffzzqq,fill=ffzzqq,fill opacity=0.1] (-3.28,-0.17) -- (-8.28,-0.17) -- (-9.825084971874734,-4.925282581475764) -- (-5.78,-7.86420884293813) -- (-1.7349150281252645,-4.925282581475766) -- cycle;
\fill[line width=0pt,fill=ffzzqq,fill opacity=0.1] (-8.28,-0.17) -- (-3.28,-0.17) -- (-3.275264570268226,1.8299943939184569) -- (-8.280002660886453,1.82999999999823) -- cycle;
\draw [line width=1pt, dotted, color=ffqqqq] (-8.28,-0.17)-- (-3.28,-0.17);
\draw [line width=1pt,color=qqqqff] (-3.28,-0.17)-- (-3.1256541807559914,-0.6455810846552836);
\draw [line width=1pt,color=qqqqff] (-8.28,-0.17)-- (-8.427815730455455,-0.6476510335275326);
\draw [line width=1pt] (-7.260720578460138,0.8494742492285339)-- (-6.260720578460138,0.8494742492285339);
\draw [line width=1pt] (-6.260720578460138,0.8494742492285339)-- (-5.555826121443226,1.5587864544434398);
\draw [line width=1pt] (-6.260720578460138,0.8494742492285339)-- (-5.566086589378064,0.13011090486636423);
\draw [line width=1pt] (-5.9134035839191,0.48979257704744905)-- (-5.553582022437874,0.836964643742355);
\draw [line width=1pt] (-5.9134035839191,0.48979257704744905)-- (-5.4134035839191,0.48979257704744905);
\draw (-6.752945094881498,1.4542426710447218) node[anchor=north west] {\tiny $L_0$};
\draw (-6.745172956750777,0.7848009961034031) node[anchor=north west] {\tiny $L_1$};
\draw (-6.13469348611009,1.1843949415152712) node[anchor=north west] {\tiny $L_1$};
\draw (-5.627175132096949,0.901930018942945) node[anchor=north west] {\tiny $L_2$};
\draw (-5.633757204057756,0.5393074766654788) node[anchor=north west] {\tiny $L_3$};
\draw [line width=1pt,color=qqqqff] (-8.28,-0.17)-- (-8.280002660886453,1.82999999999823);
\draw [line width=1pt,color=brown, <->] (-4.28,-0.17)-- (-4.280002660886453,1.82999999999823);
\draw (-4.280002660886453,1.15) node[anchor=west] {\color{brown} \tiny $\lambda_U(T)$};
\draw [line width=1pt,color=qqqqff] (-3.28,-0.17)-- (-3.275264570268226,1.8299943939184569);
\draw [line width=1pt,color=ffffff] (-8.280002660886453,1.82999999999823)-- (-3.275264570268226,1.8299943939184569);
\draw [line width=1pt, color=qqqqff] (-3.275264570268226,1.8299943939184569)-- (-8.280002660886453,1.82999999999823);
\end{tikzpicture}
\caption{\small A detail of a graphical description of the moduli space of clusters $\mathcal{R}^5_C(\vec{L})$. Note that this space is open. The dotted red line is there only to highlight where the boundary component of $\overline{\mathcal{R}^5(\vec{L})}$ was, before the addition of the $\lambda_U(T)$ factor (in brown).}
\end{subfigure}
\vspace{-3cm}

\hspace{-1cm}
\captionsetup[subfigure]{oneside,margin={0.8cm,-1.3cm}}
\begin{subfigure}[b]{0.3\textwidth}
\centering
\begin{tikzpicture}[line cap=round,line join=round,>=triangle 45,x=1cm,y=1cm]
\clip(-9.543142683810132,-1.0617529782629214) rectangle (-0.7450823198346074,7.316611926653659);
\fill[line width=0pt,color=ffzzqq,fill=ffzzqq,fill opacity=0.1] (-3.28,-0.17) -- (-8.28,-0.17) -- (-9.825084971874734,-4.925282581475764) -- (-5.78,-7.86420884293813) -- (-1.7349150281252645,-4.925282581475766) -- cycle;
\fill[line width=2pt,color=ffzzqq,fill=ffzzqq,fill opacity=0.1] (-8.28,-0.17) -- (-3.28,-0.17) -- (-3.275264570268226,1.8299943939184569) -- (-8.280002660886453,1.82999999999823) -- cycle;

\draw [line width=1pt, dotted, color=ffqqqq] (-8.28,-0.17)-- (-3.28,-0.17);
\draw [line width=1pt,color=ffqqqq] (-3.28,-0.17)-- (-3.1256541807559914,-0.6455810846552836);
\draw [line width=1pt,color=ffqqqq] (-8.28,-0.17)-- (-8.427815730455455,-0.6476510335275326);
\draw [line width=1pt] (-7.260720578460138,0.8494742492285339)-- (-6.260720578460138,0.8494742492285339);
\draw [line width=1pt] (-6.260720578460138,0.8494742492285339)-- (-5.555826121443226,1.5587864544434398);
\draw [line width=1pt] (-6.260720578460138,0.8494742492285339)-- (-5.566086589378064,0.13011090486636423);
\draw [line width=1pt] (-5.9134035839191,0.48979257704744905)-- (-5.553582022437874,0.836964643742355);
\draw [line width=1pt] (-5.9134035839191,0.48979257704744905)-- (-5.4134035839191,0.48979257704744905);
\draw (-6.752945094881498,1.4542426710447218) node[anchor=north west] {\tiny $L_0$};
\draw (-6.745172956750777,0.7848009961034031) node[anchor=north west] {\tiny $L_1$};
\draw (-6.13469348611009,1.1843949415152712) node[anchor=north west] {\tiny $L_1$};
\draw (-5.627175132096949,0.901930018942945) node[anchor=north west] {\tiny $L_2$};
\draw (-5.633757204057756,0.5393074766654788) node[anchor=north west] {\tiny $L_3$};
\draw [line width=1pt,color=ffqqqq] (-8.28,-0.17)-- (-8.280002660886453,1.82999999999823);
\draw [line width=1pt,color=ffqqqq] (-3.28,-0.17)-- (-3.275264570268226,1.8299943939184569);
\draw [line width=1pt,color=qqqqff] (-3.275264570268226,1.8299943939184569)-- (-8.280002660886453,1.82999999999823);
\end{tikzpicture}
\caption{\small A detail of a graphical description of the moduli space $\sqcup_{T\in \mathcal{T}^5(\vec{L})}\mathcal{R}^T\times \lambda_U(T)$.}
\end{subfigure}
\hspace{2cm}
\begin{subfigure}[b]{0.3\textwidth}
\centering
\begin{tikzpicture}[line cap=round,line join=round,>=triangle 45,x=1cm,y=1cm]
\clip(-9.543142683810132,-1.0617529782629214) rectangle (-0.7450823198346074,7.316611926653659);
\fill[line width=0pt,color=ffzzqq,fill=ffzzqq,fill opacity=0.1] (-3.28,-0.17) -- (-8.28,-0.17) -- (-9.825084971874734,-4.925282581475764) -- (-5.78,-7.86420884293813) -- (-1.7349150281252645,-4.925282581475766) -- cycle;
\fill[line width=2pt,color=ffzzqq,fill=ffzzqq,fill opacity=0.1] (-8.28,-0.17) -- (-3.28,-0.17) -- (-3.275264570268226,1.8299943939184569) -- (-8.280002660886453,1.82999999999823) -- cycle;

\draw [line width=1pt, dotted, color=ffqqqq] (-8.28,-0.17)-- (-3.28,-0.17);
\draw [line width=1pt,color=ffqqqq] (-3.28,-0.17)-- (-3.1256541807559914,-0.6455810846552836);
\draw [line width=1pt,color=ffqqqq] (-8.28,-0.17)-- (-8.427815730455455,-0.6476510335275326);
\draw [line width=1pt] (-7.260720578460138,0.8494742492285339)-- (-6.260720578460138,0.8494742492285339);
\draw [line width=1pt] (-6.260720578460138,0.8494742492285339)-- (-5.555826121443226,1.5587864544434398);
\draw [line width=1pt] (-6.260720578460138,0.8494742492285339)-- (-5.566086589378064,0.13011090486636423);
\draw [line width=1pt] (-5.9134035839191,0.48979257704744905)-- (-5.553582022437874,0.836964643742355);
\draw [line width=1pt] (-5.9134035839191,0.48979257704744905)-- (-5.4134035839191,0.48979257704744905);
\draw (-6.752945094881498,1.4542426710447218) node[anchor=north west] {\tiny $L_0$};
\draw (-6.745172956750777,0.7848009961034031) node[anchor=north west] {\tiny $L_1$};
\draw (-6.13469348611009,1.1843949415152712) node[anchor=north west] {\tiny $L_1$};
\draw (-5.627175132096949,0.901930018942945) node[anchor=north west] {\tiny $L_2$};
\draw (-5.633757204057756,0.5393074766654788) node[anchor=north west] {\tiny $L_3$};
\draw [line width=1pt,color=ffqqqq] (-8.28,-0.17)-- (-8.280002660886453,1.82999999999823);
\draw [line width=1pt,color=ffqqqq] (-3.28,-0.17)-- (-3.275264570268226,1.8299943939184569);
\draw [line width=1pt,color=ffqqqq] (-3.275264570268226,1.8299943939184569)-- (-8.280002660886453,1.82999999999823);
\end{tikzpicture}
\caption{\small A detail of a graphical description of the compactified moduli space od clusters $\overline{\mathcal{R}^5_C(\vec{L})}$.}
\end{subfigure}
    \caption{A graphical description of the construction of the compactification $\overline{\mathcal{R}^5_C(\vec{L})}$ of the moduli space of cluster $\mathcal{R}^5_C(\vec{L})$ for $\vec{L}=(L_0,L_1,L_2,L_3,L_1)$ starting from $\overline{\mathcal{R}^5(\vec{L})}$ (cfr. the drawing on page 121 of \cite{Seidel2008FukayaTheory}). For simplicity, we zoomed only on the boundary component corresponding to the depicted tree, which has one interior edge unilabelled by the Lagrangian $L_1$. We adopted the convention that {\color{red} red} lines correspond to closed and {\color{blue} blue} correspond to open.}
    \label{fig:enter-label}
\end{figure}

We define the bundle of clusters disks labelled by $\vec{L}$ $$\pi^{d+1}_C(\vec{L}): \mathcal{S}^{d+1}_C(\vec{L})\to \mathcal{R}^{d+1}_C(\vec{L})$$ where the fiber $S_{r,T,\lambda}:=(\pi^{d+1}_C(\vec{L}))^{-1}(r,T,\lambda)$ over an element $(r,T,\lambda)\in \mathcal{R}^{d+1}_C(\vec{L})$ is obtained by modifying $S_r\in \mathcal{S}^{d+1}(\vec{L})$ in the following way: any nodal point of $S_r$ is replaced by a line segment of length $\lambda(e)$ oriented as $e$, while at any marked point (of type II) we attach a semi-infinite line segment $(-\infty,0]$ or $[0,\infty)$ depending on weather the marked point is an entry or an exit. \\

To describe partial compactifications of the universal families $\pi^{d+1}_C(\vec{L})$, we introduce strip-like ends for clusters and then define gluing. Let $d\geq 2$ and $\vec{L}=(L_0,\ldots,L_d)$ be a tuple of Lagrangians in $\lagmd(M,\omega)$. A choice of strip-like ends on $\pi^{d+1}_C(\vec{L})$ is a choice of strip-like ends on each $\pi^{|v|+1}(\vec{L_v})$ for any $T\in \mathcal{T}^{d+1}(\vec{L})$ and any $v\in V(T)$, which is smooth in the following sense: if $(r,T,\lambda)\in \mathcal{R}^{d+1}_C(\vec{L})$ is a configurations such that there is an interior edge $e\in E^\textnormal{int}(T)$ of $T$ such that $\lambda(e)=0$, then both marked points $z_{h(e)}\in S_{r,T,\lambda}(h(e))$ and $z_{t(e)}\in S_{r,T,\lambda}(t(e))$ do not lie in the image of a strip-like end. A universal choice of strip-like ends for clusters is a choice of strip-like ends on $\pi^{d+1}(\vec{L})$ for any tuple of Lagrangians in $\lagmd(M,\omega)$ of any finite length $d+1$. We will always assume that our universal choices of strip-like ends for clusters are consistent, that is, vertex-wise consistent. \\

We fix once and for all a consistent universal choice of system of ends as well as a consistent universal choice of strip-like ends for clusters on $\mathcal{L}^{m,\textbf{d}}(M,\omega)$. We can now adapt the gluing procedure for punctured disks described in \cite[Section (9e)]{Seidel2008FukayaTheory} to the case of clusters of disks. The aim of the following is not really to precisely describe gluing, but more to set up the necessary notation for later stages in this paper.\\
Let $d_1,d_2\geq 2$ and consider two tuples $$\vec{L^1}=(L_0^1,\ldots, L_{d_1}^1) \textnormal{   and   } \vec{L^2}=(L_0^2,\ldots, L_{d_2}^2)$$ of Lagrangians in $\lagmd(M,\omega)$, whose reduced tuples we denote by $\vec{L^1}_\textnormal{red}$ and $\vec{L^2}_\textnormal{red}$.  Consider $$(r_1,T_1,\lambda_1)\in \mathcal{R}^{d_1+1}(\vec{L^1}) \textnormal{   and   }(r_2,T_2,\lambda_2)\in  \mathcal{R}^{d_2+1}(\vec{L^2})$$ and let $\rho \in [-1,0)$. As in the case of gluing of labelled trees in Section \ref{seconsystemofends}, in order to define gluing of clusters we consider the cases $m_0^e(\vec{L^1})=0$ and $m_0^e(\vec{L^1})>0$ (see page \pageref{defofmis} for the definition of $m_0^e$) separately:

\begin{enumerate}
    \item First, we assume $m_0^e(\vec{L^1})=0$, that is $L_0^1\neq L_0^d$. Consider an admissible (see Definition \ref{compatibilitydef}) index $i\in \{1,\ldots, d_2\}$ and denote by $j\in \{1,\ldots, d_2^R\}$ the index such that $L_i^2=\overline{L_j^2}$ (that is, the position of $L_i^2$ in the reduced tuple $\vec{L^2}_\textnormal{red}$, see page \pageref{defreducedtuple}). The gluing of $(r_1,T_1,\lambda_1)$ with $(r_2,T_2,\lambda_2)$ at $i$ with length $\rho$ is defined as the tuple $$\left(r,T_1\#_iT_2,\gamma^{T_1,T_1;i}[-1,\lambda_1,\lambda_2)\right)\in \mathcal{R}^{d_1+d_2+1}(\vec{L_1}\#_i\vec{L^2})$$ where:\begin{itemize}
        \item $T_1\#_iT_2$ is the tree obtained by gluing the root of $T_1$ to the $i$th exterior edge of $T_2$ (see page \pageref{compatibilitydef}),
        \item $\gamma^{T_1,T_1,i}$ is the gluing map defined in Section \ref{sectionlagtrees} and 
        \item $r\in \mathcal{R}^{T_1\#_iT_2}$ is defined as follows: the map $\gamma^{T_1,t_2,i}$ applied with length $-1$ identifies (as it produces an edge with vanishing length by definition) the unique vertex $v_1$ of $T_1$ attached to the root $e_0(T_1)$ with the unique vertex $v_2$ of $T_2$ attached to the exterior edge $e_j((T_2)_\textnormal{red})$\footnote{Another characterization of $v_1$ and $v_2$ using the partition $(\Lambda_v)_v$ introduced on page \pageref{Lambdav} is the following: $v_1\in V((T_1)_\textnormal{red})$ is the unique vertex such that $0,d_1^R\in \Lambda_{v_1}$ and $v_2\in V((T_2)_\textnormal{red})$ is the unique vertex such that $j,j+1\in \Lambda_{v_2}$.} to produce a vertex $v_g$ of $T_1\#_iT_2$; the configuration $r$ is the cluster configuration that: \begin{itemize}
            \item agrees with $r_1$ on $T_1\setminus v_1\subset T_1\#_iT_2$,
            \item agrees with $r_2$ on $T_2\setminus v_2\subset T_1\#_iT_2$,
            \item and on $v_g$ represents the puntured disks obtained by gluing the exit of $S_{r_1,T_1,\lambda_1}(v_1)$ to the $i$th entry of $S_{r_2,T_2,\lambda_2}(v_2)$ with gluing length $\rho$ (see \cite[Section (9e)]{Seidel2008FukayaTheory}).
        \end{itemize} 
    \end{itemize}  
    \item Assume now $m_0^e(\vec{L^1})>0$, that is in particular $L_0^1=L_{d_1}^1$, and consider indices $i$ and $j$ as above. The gluing of $(r_1,T_1,\lambda_1)$ with $(r_2,T_2,\lambda_2)$ at $i$ with length $\rho$ is defined as the tuple $$\left(r,T_1\#_iT_2,\gamma^{T_1,T_1,i}(-\ln(-\rho),\lambda_1,\lambda_2)\right)$$ where $T_1\#_iT_2$ and $\gamma^{T_1,T_1,i}$ are as above and $r$ agrees with $r_1$ on $T_1\setminus e_0(T_1)\subset T_1\#_iT_2$ and agrees with $r_2$ on $T_2\setminus e_0(T_2)\subset T_1\#_iT_2$. 
\end{enumerate}

\noindent In summary, we defined maps $$\gamma^{\vec{L^1},\vec{L^2},i}: [-1.0)\times \mathcal{R}^{d_1+1}(\vec{L^1}) \times \mathcal{R}^{d_2+1}(\vec{L^2})\to \mathcal{R}^{d_1+d_2}_C(\vec{L^1}\#_i\vec{L^2})$$ for any tuples of Lagrangians as above and any admissible index $i\in \{1,\ldots, d_2\}$. It is easy to see that those maps extend to maps $$\overline{\gamma^{\vec{L^1},\vec{L^2},i}}: [-1,0]\times   \mathcal{R}^{d_1+1}(\vec{L^1}) \times \mathcal{R}^{d_2+1}(\vec{L^2})\to \overline{\mathcal{R}^{d_1+d_2}_C(\vec{L^1}\#_i\vec{L^2})}$$ by trivial gluing.\\

Let now $\vec{L}:=(L_0,\ldots,L_d)$ be a tuple of Lagrangians in $\lagmd(M,\omega)$, and pick a labelled tree $T\in \mathcal{T}^{d+1}(\vec{L})$ and a number $k\in \{0,\ldots,d-2\}$. We define a gluing map $$\gamma^{T,k}: [-1,0)^{|E^\textnormal{int}_F(T)|+k}\times \mathcal{R}^T\times \lambda^k_U(T)\to \mathcal{R}^{d+1}_C(\vec{L})$$ where $E^\textnormal{int}_F(T)= E(T)\setminus E_U(T)$ is the set of non-unilabelled edges of $T$ and $\lambda^k_U(T)$ is the set of unilabelled metrics on $T$ such that exactly $k$ interior edges have infinite length (both concepts were introduced on page \pageref{sectionlagtrees}), by composing various maps of the form $\gamma^{\vec{L^1},\vec{L^2},i}$ for $k-1$ admissible tuples of Lagrangian decomposing $\vec{L}$. For instance, if $T\in \mathcal{T}^{d+1}_U(\vec{L})$ is unilabelled, then there are tuples $\vec{L_1}$, of length $d_1+1$, and $\vec{L^2}$, of length $d_2+1$, of Lagrangians in $\lagmd(M,\omega)$ such that $d_1+d_2=d+2$ and an index $i\in \{0,\ldots, d_2\}$ such that $\gamma^{T,1}= \gamma^{\vec{L_0},\vec{L_1},i}$.\\
Note that the maps $\gamma^{T,k}$ are well-defined since we are working with consistent choices of strip-like end and of system of ends, and the two are independent concepts. Each of these maps extends to maps of the form $$\overline{\gamma^{T,k}}:[-1,0]^{|E^\textnormal{int}_F(T)|+k}\times \mathcal{R}^T\times \lambda^k_U(T)\to \overline{\mathcal{R}^{d+1}_C(\vec{L})}$$ again by trivial gluing.

By obvious considerations and repeatedly applying \cite[Lemma 9.2]{Seidel2008FukayaTheory} we see that for any $(d+1)$-tuple $\vec{L}$ of Lagrangians in $\lagmd(M,\omega)$ there are maps $\overline{\gamma^{\vec{L^1},\vec{L^2},i}}$ whose restriction to small enough neighbourhood of the trivial gluing in the domain define boundary charts for $\mathcal{R}^{d+1}_C(\vec{L})$.

\begin{lem}\label{lemcluster}
    The space $\mathcal{R}^{d+1}_C(\vec{L})$ admits the structure of a smooth manifold of dimension $d-2$. Moreover, $\overline{\mathcal{R}^{d+1}_C(\vec{L})}$ admits the structure of a smooth manifold with corners which realizes Stasheff's associahedron.
\end{lem}
\noindent Notice that the fact that our moduli spaces of clusters $\mathcal{R}^{d+1}_C(\vec{L})$ realize the associahedra comes for free from the fact that to define them we just added some collar neighbourhoods to the compactified moduli spaces of punctured disks with Lagrangian labels $\overline{\mathcal{R}^{d+1}(\vec{L})}$. We remark that for $T\in \mathcal{T}^{d+1}(\vec{L})$ and $k\in \{0,\ldots,d-2\}$, the subspace $\mathcal{R}^T\times \lambda^k(T)\subset \overline{\mathcal{R}^{d+1}_C(\vec{L})}$ is a boundary component of codimension $|E_F(T)|+k$ (where $E_F(T)$ denotes the subset of non unilabelled edges of $T$).\\

We now introduce the following piece of notation which will be useful later on. Given $(r,T,\lambda)\in \mathcal{R}^{d+1}_C(\vec{L})$ we define $$\partial S_{r,T,\lambda} = \partial_F S_{r,T,\lambda} \cup \partial_M S_{r,T,\lambda}$$ where $\partial_F S_{r,T,\lambda}$ is the union of the boundaries of the disk components of $S_{r,T,\lambda}$, while $\partial_MS_{r,T,\lambda}$ is the union of all the line components in $S_{r,T,\lambda}$. Moreover, we denote by $S_{r,T,\lambda}(v)$ the disk with punctures and marked points corresponding to the vertex $v\in V(T)$, define $$S_{r,T,\lambda}^p:=\bigsqcup_{v\in V(T)} S_{r,T,\lambda}(v)$$ and for any $i\in\{0,\ldots, d^R\}$ we denote by $\partial_iS_{r,T,\lambda}(v)$ to be the $i$th boundary component of $S_{r,T,\lambda}(v)$, that is the one corresponding to the boundary component $e_i(v)$ (see Section \ref{sectiontrees}), and hence labelled by $\overline{L^v_i}$ (here we are again differentiating between $\overline{L_i}$ and $\overline{L_j}$ for $i\neq j$ even if they agree as Lagrangians).\\
A universal choice of strip-like ends for $\pi^{d+1}_C(\vec{L})$ is the pullback from $\overline{\pi^{d+1}(\vec{L})}$ of the compactification of a universal and consistent choice of strip-like ends on $\pi^{d+1}(\vec{L})$. Hence, consistency of strip-like ends for clusters comes for free by definition. Given a choice of strip-like ends on $S_{r,T,\lambda}$ we number the strip-like ends following the numbering of the exterior edges of $T_\textnormal{red}$ and denote them by $\epsilon_0,\ldots,\epsilon_{d^R}$ or $\epsilon_1,\ldots.,\epsilon_{d^R}$ depending on weather or not $T_\textnormal{red}$ contains the root or not (i.e. weather $L_0\neq L_d$ or $L_0=L_d)$.\\
$\pi^{d+1}_C(\vec{L})$ admits a partial compactification $$\overline{\pi_C^{d+1}(\vec{L})}: \overline{\mathcal{S}_C^{d+1}(\vec{L})}\to \overline{\mathcal{R}_C^{d+1}(\vec{L})}$$ defined by allowing line segments between disks to have infinite length. Those edges will be identified with broken edges following the conventions defined in Section \ref{sistemofends}.\\


\subsection{Perturbation data for clusters}\label{pdgen}
We define the concept of perturbation data in our cluster-setup. Fix coherent choices of strip-like ends and of system of ends $(s_T)_T$ (see page \pageref{sistemofends}). 
Let $(L_0,L_1)$ be a couple of Lagrangians in $\mathcal{L}^{m,\textbf{d}}(M,\omega)$. 

\begin{defn}
    A Floer datum for $(L_0,L_1)$ consists of:\begin{itemize}
    \item if $L_0=L_1=L$: a triple $(f^L,g^L,J^L)$ where $(f^L,g^L)$ is a Morse-Smale pair on $L$ such that $f^L$ has a unique maximum and $J^L$ is an $\omega$-compatible almost complex structure on $M$;
    \item if $L_0\neq L_1$: a couple $(H^{L_0,L_1}, J^{L_0,L_1})$ where $H:M\times [0,1]\to \R$ is a time dependent Hamiltonian on $M$ such that $$\varphi^{L_0,L_1}_1(L_0)\pitchfork L_1$$ where $\varphi^{L_0,L_1}_t$ is the flow of the Hamiltonian vector field generated by $H^{L_0,L_1}$, and $J$ is a $[0,1]$-family of $\omega$-compatible almost complex structures on $M$ which equals $J^{L_0}$ near $0$ and $J^{L_1}$ near $1$.
\end{itemize}
\end{defn}

Given $L\in \lagmd(M,\omega)$ we denote by $\Crit(f^L)$ the (finite) set of critical points of a Morse function $f^L:L\to \R$ on $L$, and given $x\in \Crit(f^L)$ we denote by $|x|$ its Morse index, by $W^u(x)$ its unstable manifold and by $W^s(x)$ its stable manifold. Given a couple $(L_0,L_1)$ of different Lagrangians in $\lagmd(M,\omega)$ with a choice $H^{L_0,L_1}$ of Hamiltonian Floer datum we denote by $\mathcal{O}(H^{L_0,L_1})$ the (finite) set of orbits $\gamma:[0,1]\to M$ of the Hamiltonian vector field of $H^{L_0,L_1}$ such that $$\gamma(0)\in L_0 \textnormal{   and   }\gamma(1)\in L_1.$$
Fix a choice of Floer data for any couple of Lagrangians in $\mathcal{L}^{m,\textbf{d}}(M,\omega)$. In order to state the following definition in an easier way, we arbitrarily set $H^{L,L}:=0$ and $J^{L,L}:=J^L$ for any $L\in \lagmd(M, \omega)$ and keep this convention for the whole paper.
 Let $\vec{L}=(L_0,\ldots,L_d)$ be a tuple of Lagrangians in $\mathcal{L}^{m,\textbf{d}}(M,\omega)$ and let $(r,T,\lambda)\in \mathcal{R}^{d+1}_C(\vec{L})$ be a cluster configuration. \begin{defn}\label{maindefofpd}
     A perturbation datum for $\vec{L}$ on the cluster $S_{r,T,\lambda}$ consists of the following data:
\begin{itemize}\label{pdgen2}
    \item For any $i=0,\ldots,d^F$ (see page \pageref{defofmis} for the definition of $d^F$) a choice of tuple $$\left(f^{r,T,\lambda}_{(T_i^F,\lambda_i)}, g^{r,T,\lambda}_{(T_i^F,\lambda_i)}\right)$$ where the $(T_i^F,\lambda_i)$'s are the metric subtrees of $T$ coming from the fundamental decomposition of $T$ (see Definition \ref{fundamentaldecomposition}) and: \begin{itemize}
        \item $f^{r,T,\lambda}_{(T_i^F,\lambda_i)}:T_i^F\times \overline{L_i}\to \R$ is a map smooth on edges and continuous on vertices such that $$f^{r,T,\lambda}_{(T_i^F,\lambda_i)}(\tau):\overline{L_i}\to \R$$ is Morse for any\footnote{Here we identify system of ends as subintervals of edges via the convention introduced at the end of Section \ref{sectiontrees}.} $\tau\in T_i^F$ and $$f^{r,T,\lambda}_{(T_i^F,\lambda_i)}(\tau)=f^{\overline{L_i}}$$ for any $\tau\in s_{T_i^F}$;
        \item $g^{r,T,\lambda}_{(T_i^F,\lambda_i)}$ assigns to any $\tau\in T_i^F$ a Riemannian metric $g^{r,T,\lambda}_{(T_i^F,\lambda_i)}(\tau)$ on $L$ such that the pair $$(f^{r,T,\lambda}_{(T_i^F,\lambda_i)}(\tau),g^{r,T,\lambda}_{(T_i^F,\lambda_i)})$$ is Morse-Smale for any $\tau\in T_i^F$ and $$g^{r,T,\lambda}_{(T_i^F,\lambda_i)}(\tau)=g^{\overline{L_i}} $$ for anu $\tau\in s_{T_i^F}$.
    \end{itemize}
    \item For any $v\in V(T)$ a choice of couples $$(K^{r,T,\lambda}_v,J^{r,T,\lambda}_v)$$ such that:
    \begin{itemize}
        \item $K^{r,T,\lambda}_v\in \Omega^1(S_{r,T,\lambda}(v), C^\infty(M))$ is an Hamiltonian-valued one-form which vanishes identically if $v\notin T_\textnormal{red}$ (i.e. if $S_{r,T,\lambda}(v)$ does not have any punctures), while for $v\in T_\textnormal{red}$ is such that for any $i=0,\ldots,d^R$ it satisfies $$K^{r,T,\lambda}_v(\xi)|_{\overline{L_i}}=0 \textnormal{ for any } \xi\in T(\partial_iS_{r,T,\lambda}(v))$$ and for any $i=0,\ldots, d$ on the strip-like $\epsilon_i$ we have\footnote{See page \pageref{striplikends} for the definition of $|s|\geq 1$.} $$K^{r,T,\lambda}_v=H^{L_i,L_{i+1}} dt \textnormal{ for any }|s|\geq 1.$$
        \item $J^{r,T,\lambda}_v$ is a domain-dependent $\omega$ compatible almost complex structure which, if $v\in T_i^F\subset T\setminus T_\textnormal{red}$ for some $i\in \{0,\ldots, d^F\}$ (see Section \ref{sectionlagtrees}) it is identical to $J^{L_i^F}$, while if $v\in T_\textnormal{red}$, it is such that on the $i$th strip-like end $\epsilon_i$ we have $$J^{r,T,\lambda}_v=J^{L_i,L_{i+1}} \textnormal{ for any }|s|\geq 1$$ for any $i=0,\ldots, d$.
    \end{itemize}
\end{itemize}
We define a perturbation datum for $\vec{L}$ as a smooth choice of perturbation data for $\vec{L}$ on $\mathcal{R}^{d+1}_C(\vec{L})$ and denote it by $$\left((f^{\vec{L}},g^{\vec{L}}),(K^{\vec{L}},J^{\vec{L}})\right).$$ The couple $(f^{\vec{L}},g^{\vec{L}})$ will be called the Morse part of the perturbation datum, while $(K^{\vec{L}},J^{\vec{L}})$ will be called its Floer part. A universal choice of perturbation data is a choice of perturbation datum for any tuple $\vec{L}$.
 \end{defn}

We have the following consistency conditions for perturbation data (see \cite{Mescher2018PerturbedCohomology, Sheridan2011OnPants}). 
\begin{defn}
    We say that a choice of perturbation data is consistent if for any tuple $\vec{L}$ of Lagrangians in $\lagmd(M,\omega)$, say of length $d+1$, we have:\begin{enumerate}
    \item for any $T\in \mathcal{T}(\vec{L})$ and $k\in \{0,\ldots, d-2\}$, there is a subset $$U\subset \mathcal{R}^T\times \lambda^k_U(T)\times [-1,0)^{|E_F^\textnormal{int}(T)|+k}$$ whose closure is a neighbourhood of the trivial gluing where the gluing parameters are small such that the perturbation data for clusters over $U$ agree with perturbation data induced by gluing on thin parts;
    \item all perturbation data extend smoothly to $\partial\overline{\mathcal{R}^{d+1}_C(\vec{L})}$ and agree there with perturbation data coming from (trivial) gluing.
\end{enumerate}
\noindent We define the space $E$ of consistent universal choices of perturbation data for clusters (with respect to a fixed choice of strip-like ends and of system of ends).

\begin{rem}
   It is not hard to see that $E$ is non-empty. Indeed, there are plenty of consistent universal choices of perturbation data for puntured disks (\'a la \cite{Seidel2008FukayaTheory}) and of consistent universal choices of perturbation data for Morse trees (see \cite{Mescher2018PerturbedCohomology}). As those two types of perturbation data are independent enough (indeed, equality for glued perturbation data has to hold only on small neighbourhoods of the trivial gluing) in our definition of perturbation data for clusters, it follows that the space $E$ is non-empty.
\end{rem}
\end{defn} Given $p\in E$ and a tuple $\vec{L}$ of Lagrangians in $\lagmd(M,\omega)$ we will often write the perturbation data on $\vec{L}$ induced by $p$ as $$\mathcal{D}^{\vec{L}}_p= \left(\left(f^{\vec{L}},g^{\vec{L}}\right),\left(K^{\vec{L}},J^{\vec{L}}\right)\right).$$ Moreover, given $(r,T,\lambda)\in \mathcal{R}^{d+1}_C(\vec{L})$ we will write the associated perturbation datum for $\vec{L}$ on $S_{r,T,\lambda}$ as $\mathcal{D}^{\vec{L}}_p(r,T,\lambda)$.


\subsection{Moduli spaces of Floer clusters with Lagrangian boundary} Let $p\in E$ be a consistent universal choice of perturbation data for clusters as defined in the previous section. In the following, given a disk $u: D\to M$ in $M$ we write $[u]\in \pi_2(M,L)$ for the image of the fundamental class of $D$ under the pushforward of $u$.

First, we define moduli spaces of pearly-edges (see \cite{Biran2009LagrangianHomology}). Let $L$ be a monotone Lagrangian in $\mathcal{L}^{m,\textbf{d}}(M,\omega)$ and $d\geq 1$. Pick a tree $T\in \mathcal{T}^{d+1}(L)$ labelled by the single Lagrangian $L$, an interior edge $e\in E^\textnormal{int}(T)$ of $T$, points $a,b\in L\setminus\textnormal{Crit}(f^L)$ and a class $A\in \pi_2(M,L)$.
\begin{defn}
We define the moduli space $$\mathcal{P}(a,b;p;e;A)$$ of so-called pearly trajectories in the class $A$ modeled on the edge $e$ and joining $a$ to $b$ as the space of tuples $$\left(\lambda(e),(u_1,\ldots,u_k)\right)$$ where    \begin{enumerate}
    \item $\lambda\in \lambda(T)$ is a metric on $T$,
    \item for any $i=1,\ldots,k$, $u_i:D\to M$ is a non-constant $J$-holomorphic disk such that $u_i(\partial D)\subset L$,
    \item we have $\sum_{i=1}^k [u_i]=A$,
    \item there are $t_0,\ldots,t_k\in [0,\infty)$ such that $\sum_{i=0}^k t_i=\lambda(e)$ and we have the relations $$\varphi_{f^L_{(T.\lambda)|_e},g^L_{(T.\lambda)|_e}}^{t_0}(a)=u_1(-1) \textnormal{   and   } \varphi_{f^L_{(T,\lambda)}|_e,g^L_{(T,\lambda)}|_e}^{-t_k}(b)=u_k(1)$$ as well as $$\varphi_{f^L_{(T,\lambda)}|_e,g^L_{(T,\lambda)}|_e}^{t_i}(u_i(1))=u_{i+1}(-1)$$ for any $i=1,\ldots,k-1$, where $\varphi_{f^L_{(T,\lambda)}|_e,g^L_{(T,\lambda)}|_e}^t$ is the time $t$ flow map of the negative gradient of the time dependent map $f^L_{(T,\lambda)}|_e$ with respect to the time dependent Riemannian metric $g^L_{(T,\lambda)}|_e$, where $f^L_{(T,\lambda)}$ and $g^L_{(T,\lambda)}$ are Morse data for $L$ associated to $p\in E$, as defined in Section \ref{pdgen}.
\end{enumerate}
up to reparametrization, i.e. $(\lambda(e),u_1,\ldots,u_k)$ is identified with $(\lambda'(e),u_1',\ldots,u_l')$ if and only if $k=l$, $\lambda(e)=\lambda'(e)$ and there are automorphisms $\sigma_1,\ldots,\sigma_k\in \textnormal{Aut}(D)$ of the unit disk $D\subset \mathbb C$ fixing $-1$ and $1$ such that $u_i'=\sigma_i\circ u_i$ for any $i=1,\ldots,k$.\\  Moreover, we define $\mathcal{P}(a,a;p;e;A)=\emptyset$ for any choice of parameters. The definition of $\mathcal{P}(a,b;p;e;A)$ extends also to the case where $a,b$ are critical points of the function $f_L$ in a standard way: if for instance $a\in \Crit(f_L)$, then we ask that $t_0=\infty$ (i.e. $u_1(-1)\in W^u(a)$) and $\sum_{i=1}^k t_i\in [0,\infty)$.\end{defn} \noindent The virtual dimension of $\mathcal{P}(a,b;p;e;A)$ is $n+\mu(A)-1$, where $\mu$ denotes the Maslov index.  If both $a$ and $b$ are critical points, then the virtual dimension of $\mathcal{P}(a,b;p;e;A)$ is $|a|-|b|+\mu(A)-1$, where $|\cdot |$ denotes the Morse index.\\

Consider now a tuple $\vec{L}=(L_0,\ldots,L_d)$ of Lagrangians in $\mathcal{L}^{m,\textbf{d}}(M,\omega)$. We define moduli spaces of Floer clusters with boundary on $\vec{L}$.\\
Assume first $L_0\neq L_d$. Pick $\vec{x}_i:=(x^i_1,\ldots,x^i_{m_i-1})$ for any $i=0,\ldots,d^R$, where $x^i_j\in \Crit(f^{\overline{L^i}})$ are critical points, orbits $\gamma_j\in \mathcal{O}(H^{\overline{L_{j-1}},\overline{L_j}})$ for any $i=1,\ldots,d^R$ and $\gamma_+\in \mathcal{O}(H^{\overline{L_0},\overline{L_{d^R}}})$ and a class $A\in \pi_2(M,\vec{L})$. Recall that we defined $\pi_2(M,\vec{L}): =\pi_2(M,\cup_i L_i)$ in Section \ref{tuplesoflags}.
We define the moduli space $$\mathcal{M}^{d+1}\left(\vec{x}_0, \gamma_1, \vec{x}_1,\ldots,\gamma_{d^R}, \vec{x}_{d^R}; \gamma^+;A;p\right)$$ of Floer clusters joining $\vec{x}_0, \gamma_1, \vec{x}_1,\ldots,\gamma_{d^R}, \vec{x}_{d^R}$ to $\gamma^+$ in the class $A$ as the space of tuples $((r,T,\lambda),u)$ where $$u=\left((u_v)_{v\in V(T)}, (u_e)_{e\in E(T)}\right): S_{r,T,\lambda}\to M$$ satisfies \begin{enumerate}
    \item for any vertex $v\in V(T)$, $u_v:S_{r,T,\lambda}(v)\to M$ satisfies the $(K^v,J^v)$-Floer equation and the boundary conditions $u(\partial_iS_{r,T,\lambda})\subset \overline{L_i}$,
    \item for any $i=1,\ldots,d^R$ we have $$\lim_{s\to \-\infty}u_{h(e_i(T_\textnormal{red}))}(\overline{\epsilon_i}(s,t))=\gamma_i(t)$$ and $$\lim_{s\to \infty}u_{h(e_0(T_\textnormal{red}))}(\epsilon_0(s,t))=\gamma^+(t),$$
    \item for any $i=0,\ldots,d^R$ and any interior edge $e\in E_i^\textnormal{int}(T)$ (uni)labelled by $\overline{L_i}$ there is a class $B_e\in \pi_2(M,L)$ such that $$u_e\in \mathcal{P}(u_{t(e)}(z_{t(e)}), u_{h(e)}(z_{h(e)});p;e;B_e),$$
    \item for any $i=0,\ldots,d^R$ and any $j=1,\ldots,m_i$ (see page \pageref{defofmis}) there is a class $B_j^i\in \pi_2(M,\overline{L_i})$ such that $$u_{e^i_j(T_\textnormal{uni})}\in \mathcal{P}(x^i_{j}, u_{h(e_j(T_i^F))}(z_{h(e_j(T_i^F))});p;e^i_j(T_\textnormal{uni});B_j^i),$$
    \item We have the relation $$A=\sum_{v\in V(T)}[u_v]+\sum_{e\in E^\textnormal{int}(T)}B_e+\sum_{i=0}^{d^R}\sum_{j=1}^{m_i}B_j^i$$ on\footnote{Recall that we ignore inclusions on $\pi_2$, see Section \ref{tuplesoflags}.} $\pi_2(M,\vec{L})$.
\end{enumerate}

Assume now $L_0=L_d$. Pick $x^+\in \Crit(f^{\overline{L_0}})$, $\vec{x}_i:=(x_1^i,\ldots,x_{\overline{m_i}}^i)$ for $i=0,\ldots,d^R+1$, where $x^i_j\in \Crit(f^{\overline{L_i}})$ are critical points\footnote{Recall that we work modulo $d^R$, so that $\overline{L_{d^R+1}}=\overline{L_0}$.}, orbits $\gamma_j\in \mathcal{O}(H^{\overline{L_{j-1}},\overline{L_j}})$ for any $i=1,\ldots,d^R+1$ and a class $A\in \pi_2(M,\vec{L})$. We define the moduli space $$\mathcal{M}^{d+1}\left(\vec{x}_0, \gamma_1, \vec{x}_2,\ldots,\gamma_{d^R+1}, \vec{x}_{d^R+1}; x^+;A;p\right)$$ of Floer clusters joining $\vec{x}_0, \gamma_1, \vec{x}_1,\ldots,\gamma_{d^R+1}, \vec{x}_{d^R+1}$ to $x^+$ in the class $A$ as the space of tuples $((r,T,\lambda),u)$ where $$u=((u_v)_{v\in V(T)}, (u_e)_{e\in E(T)}): S_{r,T,\lambda}\to M$$ satisfies\begin{enumerate}
    \item for any vertex $v\in V(T)$, $u_v:S_{r,T,\lambda}(v)\to M$ satisfies the $(K^v,J^v)$-Floer equation and the boundary conditions $u(\partial_iS_{r,T,\lambda})\subset \overline{L_i}$,
    \item for any $i=1,\ldots,d^R$ we have $$\lim_{s\to \-\infty}u_{h(e_i(T_\textnormal{red}))}(\overline{\epsilon_i}(s,t))=\gamma_i(t),$$
    \item for any $i=0,\ldots,d^R$ and any interior edge $e\in E_i^\textnormal{int}(T)$ (uni)labelled by $\overline{L_i}$ there is a class $B_e\in \pi_2(M,L)$ such that $$u_e\in \mathcal{P}(u_{t(e)}(z_{t(e)}), u_{h(e)}(z_{h(e)});p;e;B_e),$$
    \item for any $i=0,\ldots,d^R$ and any $j=1,\ldots,m_i$ there is a class $B_j^i\in \pi_2(M,\overline{L_i})$ such that $$u_{e^i_j(T_\textnormal{uni})}\in \mathcal{P}(x^i_{j}, u_{h(e^i_j(T_\textnormal{uni}))}(z_{h(e^i_j(T_\textnormal{uni}))});p;e^i_j(T_\textnormal{uni});B_j^i)$$ and there is a class $B_0^0\in \pi_2(M,\overline{L_0})$ such that $$u_{e_0(T)}\in \mathcal{P}(u_{t(e_0(T))}, x^+;p;e_0(T);B_0^0),$$
    \item We have the relation $$A=\sum_{v\in V(T)}[u_v]+\sum_{e\in E^\textnormal{int}(T)}B_e+\sum_{i=0}^{d^R}\sum_{j=1}^{m_i}B_j^i+B_0^0$$ on $\pi_2(M,\vec{L})$.
\end{enumerate}
\noindent A schematic representation of a Floer cluster in the case $L_0=L_d$ is depicted in Figure \ref{figcluster}.

Given a Floer cluster $(r,T,\lambda,u)$ we write $u_F:=(u_v)_{v\in V(T_\textnormal{red})}$ for the collection of curves contributing to $u$ which are not purely pseudoholomorphic, and $$\omega(u):=\sum_{v\in V(T)}\omega(u_v)+\sum_{e\in E(T)}\omega(u_e)= \omega(A)+\omega(u_F)$$ fot its total symplectic area.\label{uf}


\subsection{Transversality for Floer clusters}\label{transversality for floer clusters} 

Let $d\geq 1$ and $\vec{L}$ be a tuple of pairwise different Lagrangians in $\lagmd(M,\omega)$, i.e. with $\vec{L^F}=\vec{L}$ in the notation introduced in Section \ref{tuplesoflags}, of length $d+1$. Consider Hamiltonian orbits $\gamma_i\in \mathcal{O}(H^{L_{i-1},L_i})$ for $i\in \{1,\ldots, d\}$ and $\gamma^+\in \mathcal{O}(H^{L_0,L_d})$ as well as a class $A\in \pi_2(M,\vec{L})$ (recall that we defined $\pi_2(M,\vec{L}):=\pi_2(M,\cup_iL_i)$, see page \pageref{piconvention}). Then, associated to any Floer polygon $u$ connecting $\gamma_1,\ldots, \gamma_d$ to $\gamma^+$ in the class $A$ there is a polygonal Maslov index $\mu(u)\in \Z$ (see \cite{Fukaya2009LagrangianI} or \cite[Chapter 13]{Oh2015SymplecticHomology}). It is known that this index only depends on the elements $\gamma_1,\ldots, \gamma_d,\gamma^+$ and on the class $A$, and hence will be denoted by $\mu(\gamma_1,\ldots, \gamma_d,\gamma^+;A)$. Moreover, it is additive under breaking and bubbling of Floer polygons and bubbling of pseudoholomorphic disks (and in this case reduces to the standard Maslov index). In particular, it follows directly that the Maslov index is defined for clusters too.

    \begin{figure}
    \centering
 \scalebox{2}{
    \begin{tikzpicture}
    \begin{scope}
    [decoration={markings, 
    mark= at position 0.5 with {\arrow{stealth}}}, roundnode/.style={circle, draw, minimum size=7mm, orange}]

    \pic[tqft/pair of pants, between incoming and outgoing/.style={draw, red}, between outgoing 1 and 2/.style={draw, red},name=a];
    \pic[tqft/cylinder, between incoming and outgoing/.style={draw, red}, at={(-4,+0.5)},name=b];

    \node[roundnode] (uppercircle) at (-2,0.5) {};
    
    \node[style={circle, draw, minimum size=10mm,orange}] (lower2) at (-2,-4) {};
    \node[style={circle, draw, minimum size=3mm,orange}] (lower1) at (1,-3.7) {};
    \node (lower11) at (1,-4.8) {};
    \draw (lower11.north) node {\large .};
    \node at (1,-4.95) {\tiny $x^5_1$};
    \node (lowercent) at (0,-3) {};
    \node (lower21) at (-3,-4) {};
    \draw (lower21.east) node {\large .};
    \node at (-3.15,-4) {\tiny $x^+$};
        \draw (lower2.north) node {\large .};
    \node[above] at (lower2.north) {\tiny $x_1^0$};

    \draw[postaction={decorate},orange] ($(a-between  outgoing 1 and 2)$) to[bend right=40] (lowercent.north)  ;
    \draw[postaction={decorate},orange] (lowercent.north) to[bend left=100] (lower2.south);
    \draw[postaction={decorate},orange] (lower1.north) to[bend right=10] (lowercent.north);
    \draw[postaction={decorate},orange]   (lower2.west) to(lower21);
    \draw[postaction={decorate},orange] (lower11) to (lower1.south);
    
    \draw[postaction={decorate},orange] (uppercircle.south) to[bend right=25] ($(a-between first incoming and first outgoing)$) ;
    \draw[postaction={decorate},orange] ($(b-between last incoming and last outgoing)$) to[bend left=20] (uppercircle.west) ;

    \draw node at (-1.3,-0.5){\color{brown} \tiny $L_2$};
    \draw node at (1.3,-0.5){\color{brown} \tiny $L_1$};
    \draw node at (0,-2){\color{brown} \tiny \ \ \ $L_0$};
    \draw node at (0,0.2){ $\downarrow$};
    \draw node at (0,0.53){ \tiny $\gamma_2$};
    \draw node at (-1,-2.2){ $\uparrow$};
    \draw node at (-1,-2.53){ \tiny $\gamma_5$};
    \draw node at (-4.7,-0.7){\color{brown} \tiny $L_3$};
    \draw node at (1,-2.2){ $\uparrow$};
    \draw node at (-4,-1.75){ $\uparrow$};
    \draw node at (-4,-2.08){\tiny $\gamma_4$};
    \draw node at (-4,0.7){ $\downarrow$};
    \draw node at (-4,1.03){\tiny $\gamma_3$};
    \draw node at (1,-2.53){ \tiny $\gamma_1$};
    \end{scope}
    \end{tikzpicture}}

    \caption{A schematic representation of a Floer cluster in the moduli space $\mathcal{M}^8(x_1^0,\gamma_1,\gamma_2,\gamma_3,\gamma_4,\gamma_5, x_1^{4+1}; x^+;A)$ for some class $A\in \pi_2(M,\vec{L})$ for the tuple $\vec{L} = (L_0,L_0,L_1,L_2,L_3,L_2,L_0,L_0)$. Morse trajectories and smooth disks are colored in orange, critical points in black and Floer polygons in red. Note that in this case $\vec{L_\textnormal{red}}=((L_0,2+2),L_1,L_2,L_3,L_2)$ and $\vec{L^F}= (L_0,L_1,L_2,L_3)$.}
    \label{figcluster}
\end{figure}
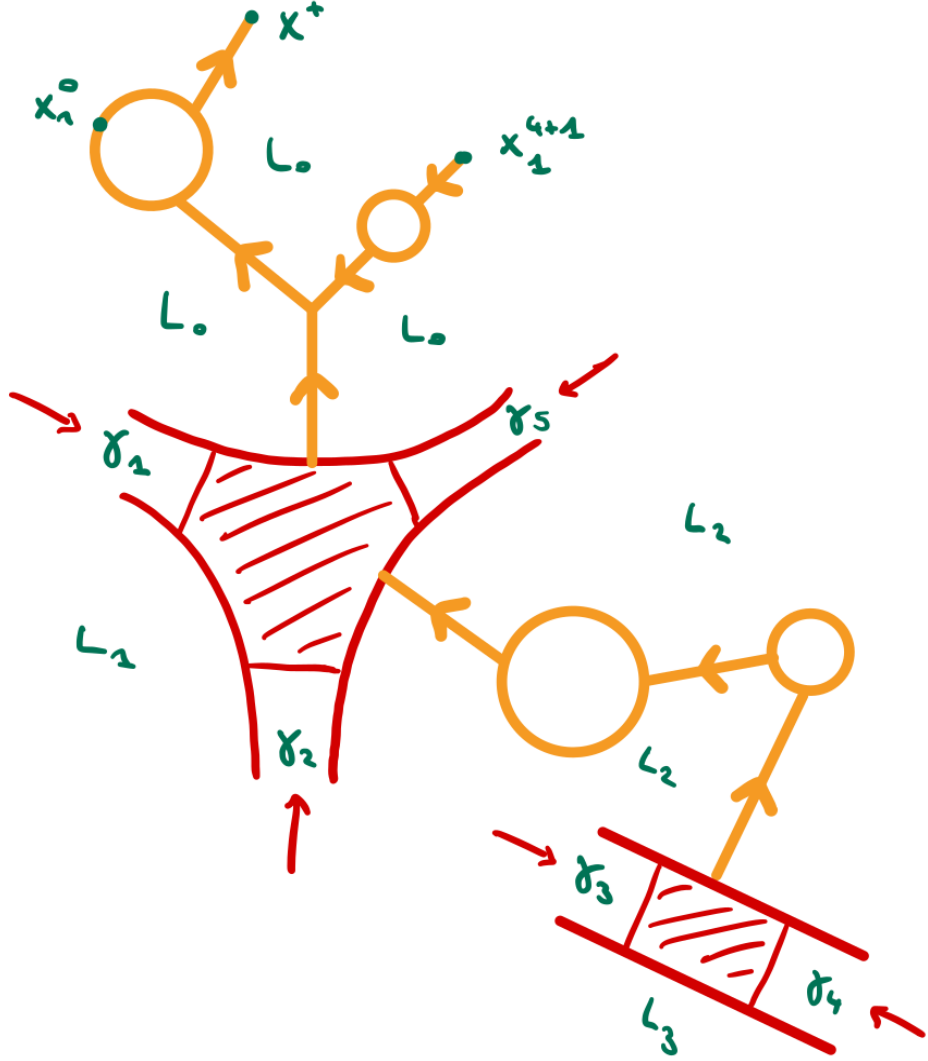
\noindent In this section we sketch the proof of the following result.

\begin{prop}\label{tc1}
    There exists a residual subset $E_\textnormal{reg}\subset E$ such that for any $p\in E_\textnormal{reg}$ and any tuple $\vec{L}=(L_0,\ldots,L_d)$ of Lagrangians in $\lagmd(M,\omega)$ of any length $d\geq 1$ the following holds:
    \begin{enumerate}
        \item if $L_0\neq L_d$ then for any $\vec{x}_i:=(x^i_1,\ldots,x^i_{m_i-1})$, $i=0,\ldots,d^R$, where $x^i_j\in \Crit(f^{\overline{L_i}})$ are critical points, any orbits $\gamma_j\in \mathcal{O}\left(H^{\overline{L_{j-1}},\overline{L_j}}\right)$, $i=1,\ldots,d^R$ and $\gamma^+\in \mathcal{O}\left(H^{\overline{L_0},\overline{L_{d^R}}}\right)$ and any class $A\in \pi_2(M,\vec{L})$ satisfying \begin{equation}\label{te1}
            d_{A}^{(\vec{x}_i)_i, (\gamma_j)_j,\gamma^+}:=\mu(\gamma_1,\ldots, \gamma_d,\gamma^+,A)+\sum_{i=0}^{d^R}\sum_{j=1}^{m_i-1}|x^i_j|-n(d-d^R-1)+d-2\leq 1
        \end{equation} then the moduli space $$\mathcal{M}^{d+1}\left(\vec{x}_0, \gamma_1, \vec{x}_1,\ldots,\gamma_{d^R}, \vec{x}_{d^R}; \gamma^+;A;p\right)$$ is a smooth manifold of dimension $d_{A}^{(\vec{x}_i)_i, (\gamma_j)_j,\gamma^+}$.
        \item if $L_0=L_d$ then for any $x^+\in \Crit(f^{\overline{L_0}})$, $\vec{x}_i:=(x_1^i,\ldots,x_{\overline{m_i}}^i)$ for $i=0,\ldots,d^R+1$, where $x^i_j\in \Crit(f^{\overline{L_i}})$, any orbits $\gamma_j\in \mathcal{O}\left(H^{\overline{L_{j-1}},\overline{L_j}}\right)$, $i=1,\ldots,d^R$ and any class $A\in \pi_2(M,\vec{L})$ satisfying 
        \begin{equation}\label{te2}
           d_{A}^{(\vec{x}_i)_i, (\gamma_j)_j,x^+}:=\mu(\gamma_1,\ldots, \gamma_d,\gamma^+,A)+\sum_{i=0}^{d^R}\sum_{j=1}^{m_i-1}|x^i_j|-|x^+|-n(d-d^R-1)+d-2\leq 1
        \end{equation}
     then the moduli space $$\mathcal{M}^{d+1}\left(\vec{x}_0, \gamma_1, \vec{x}_1,\ldots,\gamma_{d^R+1}, \vec{x}_{d^R+1}; x^+;A;p\right)$$ is a smooth manifold of dimension $d_{A}^{(\vec{x}_i)_i, (\gamma_j)_j,x^+}$.
    \end{enumerate}
\end{prop}
\begin{defn}\label{defEreg}
    The elements of $E_\textnormal{reg}\subset E$ will be called regular perturbation data.
\end{defn}
\noindent The proof of Proposition \ref{tc1} is analogous to the proof of Proposition 4.7 in \cite{Sheridan2011OnPants}, with the difference that we are not using Hamiltonian perturbations for pearly trees. As it will be apparent from the discussion below, this is the reason why we are dealing with moduli spaces of virtual dimension $\leq 1$ only (cfr. \cite{Biran2009LagrangianHomology}). We now explain how to account for this difference.\\
Fix $p\in E$ (which we will omit from the notation of Floer and perturbation data). Assume first that $\vec{L^F}=L$, that is $L=L_i$ for any $i$, and pick critical points $x_1,\ldots, x_d, x^+$ of $f^L$ and a class $A\in \pi_2(M,\vec{L})=\pi_2(M,L)$. Then the virtual dimension of $$\mathcal{M}^{d+1}(x_1,\ldots, x_d;x^+;A;((f^L_T,g^L_T)_T,J^L))$$ where $(f^L_T,g^L_T)_T$ is some choice of (Morse) perturbation datum for $L$ (and the notation makes sense as $d^F=1$) and $J^L$ is part of the Floer data for $L$, is $$\mu(A)+\sum_{i=1}^d(|x_i|-n)-|x^+|+d-2.$$ Moreover, assuming that the virtual dimension is $\leq 1$, then for a generic choice of the almost complex structure $J^L$ and (Morse) perturbation data, the Morse functions are in generic position, the evaluation maps from the holomorphic discs to the Lagrangian is transverse to stable and unstable manifolds of our Morse functions, and our moduli spaces only contain configurations made of simple and absolutely distinct disks joined by absolutely distinct Morse flowlines (see \cite{Mescher2018PerturbedCohomology} for the Morse part, and \cite{Charest2012SourceComplexes} for simpleness of holomorphic disks, proved by applying the results from \cite{Lazzarini2011RelativeCurves} following ideas from \cite{Biran2009LagrangianHomology}), which are the conditions needed in order to ensure regularity of moduli spaces (see \cite{Mcduff2012J-holomorphicTopology}). In this case, it is crucial that the minimal Maslov number $N_L$ of $L$ is at least $2$, as bubbling of pseudoholomorphic disks is a codimension $1$ phenomenon.\\
We now take care of the general situation. Consider $\vec{L}$ such that $m_0^e(\vec{L})=0$, i.e. $L_0\neq L_d$. Pick $\vec{x}_i:=(x^i_1,\ldots,x^i_{m_i-1})$, $i=0,\ldots,d^R$, $\gamma_j\in \mathcal{O}\left(H^{\overline{L_{j-1}},\overline{L_j}}\right)$, $i=1,\ldots,d^R$, $\gamma_+\in \mathcal{O}\left(H^{\overline{L_0},\overline{L_{d^R}}}\right)$ and $A\in \pi_2(M,\vec{L})$ such that $d_{A}^{(\vec{x}_i)_i, (\gamma_j)_j,\gamma^+}\leq 1$. Then, due to the fact that the Maslov index for polygonal maps respects breaking and disk bubbling (see \cite{Oh2015SymplecticHomology}), it follows from Estimate \ref{te1} that for any $i\in \{0,\ldots, d^F\}$ we have $$\mu\left(\sum_{v\in V(T_i^F)}[u_v]+ \sum_{e\in E_i^\textnormal{int}(T)} B_e + \sum_{\substack{k\in \{0,\ldots, d^R\} \\ k^F=i}}\sum_{j=1}^{\overline{m_{k}}}B_j^{k}\right) + \sum_{\substack{k\in \{0,\ldots, d^R\} \\ k^F=i}}\left(\sum_{j=1}^{\overline{m_k}}(|x_j^k|-1) + m_k\right)-2 \leq 1$$ (see page \pageref{fundamentaldecomposition} for the definition of the subtrees $T^F_i$).
In particular, as above, we can apply the results contained in \cite{Mescher2018PerturbedCohomology, Charest2012SourceComplexes} to any configuration associated to any connected component of $T_i^F$. This arguments fills in the gap between the situation in \cite{Sheridan2011OnPants} and ours.

\subsection{Definition of $\mathcal{F}uk(M)$} Let $p\in E_\textnormal{reg}$ be a regular perturbation datum. We recall that $\Lambda$ denotes the standard Novikov field over $\Z_2$ (see Section \ref{monotonelags}). Given a couple $(L_0,L_1)$ of Lagrangians in $\mathcal{L}^{m,\textbf{d}}(M,\omega)$ we define its $p$-Floer vector spaces $CF(L_0,L_1;p)$ as follows:
\begin{enumerate}
    \item if $L_0\neq L_1$ we define $$CF(L_0,L_1;p):=\bigoplus_{\gamma\in \mathcal{O}(H^{L_0,L_1}_p)}\Lambda\cdot \gamma$$ where $H^{L_0,L_1}_p$ is part of the Floer datum for $(L_0,L_1)$ prescribed by $p$;
    \item if $L_0= L_1=:L$ we define $$CF(L_0,L_1;p):=\bigoplus_{x\in \Crit(f^{L}_p)}\Lambda\cdot x$$ where $f^L_p$ is part of the Floer datum for $(L_0,L_1)$ prescribed by $p$.
\end{enumerate}
We will often suppress $p$ from the notation when there is no risk of confusion.\\
Let $d\geq 1$ and consider a tuple $\vec{L}=(L_0,\ldots,L_d)$ of Lagrangians in $\mathcal{L}^{m,\textbf{d}}(M,\omega)$. In the following we define a map $$\mu_d: CF(L_0,L_1)\otimes \cdots \otimes CF(L_{d-1},L_d)\to CF(L_0,L_d)$$ which depends on the choice of the perturbation datum $p\in E_\textnormal{reg}$. First, we assume $L_0\neq L_d$. Let $\vec{x}_i:=(x^i_1,\ldots,x^i_{m_i-1})$, $i=0,\ldots,d^R$, where $x^i_j\in \Crit(f^{\overline{L_i}})$ are critical points, $\gamma_j\in \mathcal{O}(H^{\overline{L_{j-1}},\overline{L_j}})$, $i=1,\ldots,d^R$ are orbits (see page \pageref{defofmis} for the definition of $d^R$ and $m_i$). Then we define $$\mu_d(\vec{x}_0, \gamma_1, \vec{x}_1,\ldots,\gamma_{d^R}, \vec{x}_{d^R}):=\sum_{\gamma^+,\, A}\sum_u T^{\omega(u)}\cdot \gamma^+$$ where the first sum runs over orbits $\gamma^+\in \mathcal{O}(H^{L_0,L_d})$ and classes $A\in \pi_2(M,\vec{L})$ such that $$ d_{A}^{(\vec{x}_i)_i, (\gamma_j)_j,\gamma^+}=0,$$ and second sum runs over Floer clusters $$u\in \mathcal{M}^{d+1}\left(\vec{x}_0, \gamma_1, \vec{x}_1,\ldots,\gamma_{d^R}, \vec{x}_{d^R}; \gamma^+;A;p\right).$$
Assume now $L_0=L_d=:L$. Let  $\vec{x}_i:=(x_1^i,\ldots,x_{\overline{m_i}}^i)$ for $i=0,\ldots,d^R+1$, where $x^i_j\in \Crit(f^{\overline{L_i}})$ are critical points, orbits $\gamma_j\in \mathcal{O}(H^{\overline{L_{j-1}},\overline{L_j}})$ for any $i=1,\ldots,d^R$. Then we define $$\mu_d(\vec{x}_0, \gamma_1, \vec{x}_1,\ldots,\gamma_{d^R}, \vec{x}_{d^R}):=\sum_{x^+, A}\sum_u T^{\omega(u)}\cdot x^+$$ where the first sum runs over critical points $x^+\in \Crit(f^{L_0})$ and classes $A\in \pi_2(M,\vec{L})$ such that $$d_{A}^{(\vec{x}_i)_i, (\gamma_j)_j,x^+}=0,$$ and the second sum runs over Floer clusters $$u\in \mathcal{M}^{d+1}\left(\vec{x}_0, \gamma_1, \vec{x}_1,\ldots,\gamma_{d^R}, \vec{x}_{d^R}; x^+;A;p\right).$$


Before defining the $p$-Fukaya category of $M$, we show that the maps $\mu_d$ above are well-defined and satisfy the expected properties, that is the $A_\infty$-equations (see Equation \ref{ainfequation}). We have the following standard compactness result (see \cite{Seidel2008FukayaTheory, Charest2012SourceComplexes, Biran2008RigiditySubmanifolds}). 
\begin{prop}\label{cc1}
    Let $p\in E_\textnormal{reg}$. Then: \begin{enumerate}
        \item assume $L_0\neq L_d$ and consider generators $\vec{x}_i:=(x^i_1,\ldots,x^i_{m_i-1})$, $i=0,\ldots,d^R$, $\gamma_j\in \mathcal{O}\left(H^{\overline{L_{j-1}},\overline{L_j}}\right)$, $i=1,\ldots,d^R$ and $\gamma_+\in \mathcal{O}\left(H^{\overline{L_0},\overline{L_{d^R}}}\right)$ and a class $A\in \pi_2(M,\vec{L})$ as in Proposition \ref{tc1}, then:
        \begin{enumerate}
            \item if $d_{A}^{(\vec{x}_i)_i, (\gamma_j)_j,\gamma^+}=0$, the moduli space $$\mathcal{M}^{d+1}\left(\vec{x}_0, \gamma_1, \vec{x}_1,\ldots,\gamma_{d^R}, \vec{x}_{d^R}; \gamma^+;A;p\right)$$ is compact,
            \item if $d_{A}^{(\vec{x}_i)_i, (\gamma_j)_j,\gamma^+}=1$ it admits a compactification into a manifold with  boundary $$\overline{\mathcal{M}}^{d+1}\left(\vec{x}_0, \gamma_1, \vec{x}_1,\ldots,\gamma_{d^R}, \vec{x}_{d^R}; \gamma^+;A;p\right)$$ whose boundary points are in one-to-one correspondence with the terms in the $A_\infty$-equation \ref{ainfequation} for $(\vec{x}_0, \gamma_1, \vec{x}_1,\ldots,\gamma_{d^R}, \vec{x}_{d^R}; \gamma^+)$
        \end{enumerate}
        \item assume $L_0=L_d$ and consider generators $x^+\in \Crit(f^{L_0})$, $\vec{x}_i:=(x_1^i,\ldots,x_{\overline{m_i}}^i)$ for $i=0,\ldots,d^R+1$, $\gamma_j\in \mathcal{O}\left(H^{\overline{L_{j-1}},\overline{L_j}}\right)$, $i=1,\ldots,d^R$ and a class $A\in \pi_2(M,\vec{L})$ as in Proposition \ref{tc1}, then:
        \begin{enumerate}
            \item if $d_{A}^{(\vec{x}_i)_i, (\gamma_j)_j,x^+}=0$, the moduli space $$\mathcal{M}^{d+1}\left(\vec{x}_0, \gamma_1, \vec{x}_1,\ldots,\gamma_{d^R}, \vec{x}_{d^R}; x^+;A;p\right)$$ is compact,
            \item if $d_{A}^{(\vec{x}_i)_i, (\gamma_j)_j,x^+}=1$ it admits a compactification into a manifold with boundary $$\overline{\mathcal{M}}^{d+1}\left(\vec{x}_0, \gamma_1, \vec{x}_1,\ldots,\gamma_{d^R}, \vec{x}_{d^R}; x^+;A;p\right)$$  whose boundary points are in one-to-one correspondence with the terms in the $A_\infty$-equation \ref{ainfequation} for $(\vec{x}_0, \gamma_1, \vec{x}_1,\ldots,\gamma_{d^R}, \vec{x}_{d^R}; x^+)$.
        \end{enumerate}
    \end{enumerate}
\end{prop}
\begin{rem}
    The statement about the boundary of $1$-dimensional components of moduli spaces of Floer clusters is not presented in the most accurate form to avoid notational complexity. However, the meaning should be clear as results of this kind are very standard in various construction of Fukaya categories (see for instance \cite[Section 9l]{Seidel2008FukayaTheory}).
\end{rem}
The proof of Proposition \ref{cc1} is a mix of standard arguments (for Morse trees \cite{Abouzaid2011APlumbings, Mescher2018PerturbedCohomology} and for Floer polygons \cite{Seidel2008FukayaTheory}) and the structure of the boundary of $1$-dimensional components comes from the fact that compactification of sources spaces for Floer clusters realize associahedra (see Section \ref{sourcespaces}) and from breaking of Morse flowlines along leaves as well as concentration of energy along strip-like ends. Notice that we do not have to care about shrinking of Morse trajectories between pseudoholomorphic disks as in \cite{Biran2007QuantumSubmanifolds}, as by construction the limit lies in the interior of our moduli spaces. There is just one case of bubbling that may a priori happen in $1$-dimensional case but we want to avoid in order for the boundary to look `as it should be' (that is, so that it realizes $A_\infty$-operations): what in \cite{Biran2007QuantumSubmanifolds} is referred to as `side bubbling', that is bubbling of pseudoholomorphic disks away from marked points. This cannot happen in our situation because of our assumption on the minimal Maslov number of Lagrangians in $\lagmd(M,\omega)$, and because the polygonal Maslov index respects bubbling and breaking, as it would give rise to a Floer cluster lying in a smooth manifold of negative dimension (exactly as in \cite{Biran2007QuantumSubmanifolds}).\\

We define the $p$-Fukaya category $\mathcal{F}uk^\textbf{d}(M,\omega;p)$ of $(M,\omega)$ as follows: the objects are monotone Lagrangians in $\lagmd(M,\omega)$, the morphism space between any two objects $L_0,L_1\in \lagmd(M,\omega)$ is the associated $p$-Floer complex $CF(L_0,L_1;p)$ and the $\mu_d$-maps are those defined above. When there is no risk of confusion we will drop $\textbf{d}$, $\omega$ and $p$ from the notation. The following is the main result of this section.
\begin{prop}
    For any $p\in E_\textnormal{reg}$, $\mathcal{F}uk^\mathbf{d}(M,\omega;p)$ is a strictly unital $A_\infty$-category.
\end{prop}
\begin{proof}
    The fact that $\mathcal{F}uk^\textbf{d}(M,\omega;p)$ is an $A_\infty$-category follows directly from Proposition \ref{cc1}. We show that it is a strictly unital one. Let $L\in \lagmd(M,\omega)$ and denote by $e_L\in \Crit(f^L_p)$ the (unique) maximum of the Morse function $f^L_p\colon L\to \R$. Note that for any critical point $x^+\in \Crit(f^L_p)$ and any non-trivial class $0\neq A\in \pi_2(M,L)$ the moduli space $\mathcal{M}^2(e_L,x^+;A;p)$ is at least one dimensional. From this and the well known fact that the top Morse homology of $L$ is isomorphic to the coefficient field $\Lambda$ it follows that $e_L$ is a cycle.\\ Let now $x^-\in \Crit(f^L_p)$. Assume there is a non-trivial class $0\neq A\in \pi_2(M,L)$ and another critical point $x^+\in \Crit(f^L_p)$ such that  $$|e_L|+|x^-|-|x^+|+\mu(A) -n = 0 \textnormal{   and   }\mathcal{M}^3(e_L,x^-;x^+;A;p)\neq \emptyset$$ As $|e_L|=n$ the first equation may be rewritten as $\mu(A)+|x^-|-|x^+|=0$. At this point however, the existence of a trajectory in $\mathcal{M}^3(e_L,x^-;x^+;A;p)$ implies the existence of a trajectory in $\mathcal{M}^2(x^-;x^+;A;p)$ which in turn implies $\mu(A)+|x^-|-|x^+|-1=0$, a contradiction. Assume now that there is a trajectory in $\mathcal{M}^3(e_L,x^-;x^+;0;p)$. This implies $|x^-|=|x^+|$ (as $\mu(0)=0$) and that there is a Morse trajectory from $x^-$ to $x^+$, so that $x^-=x^+$. We hence proved $\mu_2(e_L,x^-)=x^-$ for any $x^-\in \Crit(f^L_p)$. \\ Let now $L_1\neq L$ and consider an orbit $\gamma^-\in \mathcal{O}(H^{L,L_1}_p)$. Assume there is some orbit $\gamma^+\in \mathcal{O}(H^{L,L_1}_p)$ and some class $A\in \pi_2(M,\vec{L})$ such that $$\mu(\gamma^-,\gamma^+;A)+|e_L|-n=0\textnormal{   and   } \mathcal{M}^3_0(e_L,\gamma^-;\gamma^+;A;p)\neq 0$$ This implies in particular the existence of an unparametrized Floer $(H^{L,L_1}_p, J^{L,L_1}_p)$-strip of zero index from $\gamma^-$ to $\gamma^+$, which is possible if and only if $\gamma^-=\gamma^+$. We hence proved $\mu_2(e_L,\gamma^-)=\gamma^-$ for any $\gamma^-\in \mathcal{O}(H^{L,L_1}_p)$. In particular, $e_L$ is a strict unit for $L$, as claimed.    \end{proof}
\noindent The following result, whose proof we omit, can be proved in the same way as the analogous result for standard Fukaya categories in \cite[Chapter 10]{Seidel2008FukayaTheory}. Notice that it may be proved also by slightly extending the construction of continuation functors developed in Section \ref{contfunc}.
\begin{lem}\label{inv1}
    Let $p,q\in E_\textnormal{reg}$. Then $\mathcal{F}uk(M;p)$ is quasi equivalent to $\mathcal{F}uk(M;q)$.
\end{lem}
\noindent It can moreover be proved that our construction of the Fukaya category is quasi-equivalent to the standard one contained in \cite{Seidel2008FukayaTheory}. This may be proved as in \cite{Sheridan2011OnPants} or via the machinery of PSS functors, which will appear in the author's PhD thesis.


\subsection{Weakly-filtered structure on $\mathcal{F}uk(M)$}\label{wfs}

The content of this subsection is an adaptation of \cite[Section 3.3]{Biran2021LagrangianCategories} to the cluster setting. Fix a regular perturbation datum $p\in E_\textnormal{reg}$. Let $(L_0,L_1)$ be a tuple of Lagrangians in $\lagmd(M,\omega)$. We define the $p$-action functional \label{actionfunctionaldefinition}$$\mathbb{A}_p: CF(L_0,L_1;p)\rightarrow \mathbb{R}$$ for $(L_0,L_1)$ as follows: consider a generator $g\in CF(L_0,L_1;p)$ and a Novikov series $P:=\sum a_iT^{\lambda_i}\in \Lambda$ with $a_i\neq 0$ for all $i$, ordered in such a way that $\lambda_0<\lambda_i$ for any $i\in \Z_{\geq 1}$, then we set $$\mathbb{A}_p(Pg):=-\lambda_0+\int_0^1 H^{L_0,L_1}\circ g \ dt, \ \ \ \mathbb{A}_p(0):=-\infty$$
and extend it for $\sum P_ig_i\in CF(L_0,L_1:p)$ as $$\mathbb{A}_p\Bigl(\sum P_ig_i\Bigr):= \max_i \mathbb{A}_p(P_ig_i)$$
We use $\mathbb{A}_p$ to define an increasing $\R$-filtration on $CF(L_0,L_1;p)$ via $$CF^{\leq \alpha}(L_0,L_1;p):= \mathbb{A}_p^{-1}(-\infty, \alpha]$$ for any $\alpha\in \mathbb{R}$. It is well known that this filtration endows $(CF(L_0,L_1;p),\mu_1)$ with the structure of a filtered chain complex (see for instance \cite{Biran2021LagrangianCategories}) in the case $L_0\neq L_1$, and is trivial to see that it does also in the case $L_0=L_1$ (indeed in this case all generators lie at filtration level $0$ since we arbitrarily set $H^{L,L}=0$ at the beginning of Section \ref{pdgen}).
We have the following central result.
\begin{prop}\label{wfilt}
There is a non-empty subset $E^\textnormal{wf}_\textnormal{reg}\subset E_\textnormal{reg}$ such that for any $p\in E^\textnormal{wf}_\textnormal{reg}$ the filtrations described above on $p$-Floer complexes induce on $\mathcal{F}uk(M;p)$ the structure of a weakly-filtered $A_\infty$-category with units at filtration level $\leq 0$.
\end{prop}
\begin{rem}
    The above result is the analogous of \cite[Proposition 3.1]{Biran2021LagrangianCategories} in the cluster setting. The main difference is that we get units at vanishing filtration level as a result of the use of the cluster setting, as opposed to the standard setting used in the cited paper.
\end{rem}
\noindent In the remaining of this section, we sketch a proof of why in general the action functional only induces a weakly-filtered structure instead of a genuine filtered one (see Section \ref{ainfpre} for the definition of filtered and weakly-filtered $A_\infty$-categories). Central to this kind of results are energy computations for Floer-like curves with asymptotic conditions. Recall that the energy of a Floer like curve $u:S\rightarrow M$ (where $S$ is some Riemann surface with punctures) is defined via $$E(u):= \int_S|Du-X^{K}|^2\sigma$$ where $K$ is some Hamiltonian perturbation datum on $S$, $X^K$ is the indued Hamiltonian vector field and $\sigma$ is some area form on $S$ (from the choice of which the value of $E$ is independent, see \cite{Mcduff2012J-holomorphicTopology}) and the norm involved is defined on the space of linear maps $TS\rightarrow u^*TM$ as $$|L|:= \sqrt{\frac{\omega(L(v), J(L(v)))+\omega(L(jv), J(L(jv)))}{\sigma(v,jv)}}$$ where $j$ is the almost complex structure on $S$ (from which $E$ definitely depends on). Given a Floer cluster $(r,T,\lambda,u)$ (in some moduli space of clusters) we write $$E(u):=\sum_{v\in V(T)}E(u_v)+\sum_{e\in E(T)}E(u_e)= A+E(u_F)$$ where $u_F$ is defined on page \pageref{uf}.

We start by showing that $\mu_1$ preserves filtration. Let $(L_0,L_1)$ be a couple of different Lagrangians in $\mathcal{L}^{m,\textbf{d}}(M,\omega)$. Notice that for the standard strip $\R\times [0,1]$, the norm in the definition of the energy is defined with respect to the Riemannian metric $g:=g_{\omega, J^{L_0,L_1}_p}$ induced by $\omega$ and $J^{L_0,L_1}_p$. Consider two orbits $\gamma_-, \gamma_+\in \mathcal{O}(H^{L_0,L_1}_p)$, a class $A\in \pi_2(M,L_0\cup L_1)$ such that $\mu(\gamma_-,\gamma_+;A)-1=0$ and a Floer strip $u\in \mathcal{M}^2(\gamma_-,\gamma_+;A;p)$, then:
\begin{align*}
    E(u) & = \int_{\R\times [0,1]} |\partial_su|^2 \ dsdt = \int_{\R\times [0,1]} g(\partial_su, J^{L_0,L_1}_p(-J^{L_0,L_1}_p\partial u - \nabla H^{L_0,L_1}_p(u)) dsdt  \\ &= \int_{\R\times [0,1]} \omega(\partial_su, \partial_t u) \, dsdt + \int_0^1H^{L_0,L_1}_p(t,u(s,t))\ dt |_{s=- \infty}^{s=+ \infty} \\ & = \omega(u)-\int_0^1H^{L_0,L_1}_p\circ \gamma_+\ dt + \int_0^1 H^{L_0,L_1}_p\circ \gamma_-\ dt \\
    &= \mathbb{A}_p(\gamma_-)-\mathbb{A}_p(T^{\omega(u)}\gamma_+).
\end{align*}

\noindent From this and the definition on the filtration on Floer complexes it follows that $$\mathbb{A}_p(\mu_1(\gamma_-))\leq \mathbb{A}_p(\gamma_-)$$ as claimed.\\ Consider now a couple $(L,L)$ of identical Lagrangians in $\lagmd(M,\omega)$. Consider critical points $x_-,x_+\in \Crit(f^L_p)$ and a class $A\in \pi_2(M,L)$ such that $\mu(A)+|x_-|-|x_+|-1=0$ and consider a pearly trajectory $u\in \mathcal{M}^2(x_-;x_+;A;p)$. Then $$0\leq E(u)= \omega(u) = \mathbb{A}_p(x_-)-\mathbb{A}_p(T^{\omega(u)}x_+)$$ and hence $\mathbb{A}_p(\mu_1(x_-))\leq\mathbb{A}_p(x_-)$.\\ 
We perform similar calculations in the case of higher operations (see \cite{Seidel2008FukayaTheory, Biran2014LagrangianCategories}). Let $d\geq 2$ and consider first the $(d+1)$-tuple $\vec{L}:=(L,\ldots,L)$ where $L\in\mathcal{L}^{m,\textbf{d}}(M,\omega)$. Consider critical points $x_1,\ldots,x_d,x^+\in \Crit(f^L_p)$ and a class $A\in \pi_2(M,L)$ such that $d_{A}^{(x_i)_i,x^+}=0$ and pick a Floer cluster $$(r,T,\lambda,u)\in \mathcal{M}^{d+1}(x_-^1,\ldots,x_-^d;x^+;A;p).$$ This case is very similar to the case of a couple of equal Lagrangians above: we have $$0\leq E(u)=\omega(u)=\omega(A) = \sum_{i=1}^d\mathbb{A}_p(x_-^i)-\mathbb{A}_p(T^{\omega(u)}x^+)$$ by definition of the action functional, and hence $$\mathbb{A}_p(\mu^d(x_-^1,\ldots,x_-^d))\leq \sum_{i=1}^d\mathbb{A}_p(x_-^i)$$ that is, $\mu^d$ restricted to tuples made of equal Lagrangians preserves filtration.

Let $d\geq 2$ and $\vec{L}=(L_0,\ldots,L_d)$ be an arbitrary tuple of Lagrangians in $\lagmd(M,\omega)$. Assume $L_0\neq L_d$ first and consider $\vec{x}_i:=(x^i_1,\ldots,x^i_{m_i-1})$, $i=0,\ldots,d^R$, where $x^i_j\in \Crit(f^{\overline{L_i}})$ are critical points, orbits $\gamma_j\in \mathcal{O}(H^{\overline{L_{j-1}},\overline{L_j}})$, $i=1,\ldots,d^R$ and $\gamma_+\in \mathcal{O}(H^{L_0,L_d}_p)$ and a class $A\in \pi_2(M,\vec{L})$ such that $d_{A}^{(\vec{x}_i)_i, (\gamma_j)_j,\gamma^+}=0$ (see page \pageref{tc1}). Pick a Floer cluster $$(r,T,\lambda, u)\in \mathcal{M}^{d+1}_0\left(\vec{x}_0, \gamma_1, \vec{x}_1,\ldots,\gamma_{d^R}, \vec{x}_{d^R}; \gamma^+;A;p\right).$$ For any $v\in V(T)$ choose conformal coordinates $(s,t)$ on $S_{r,T,\lambda}(v)$ and write the area form as $\sigma_{r,T,\lambda}(v)=\rho_{r,T,\lambda}(v)ds\wedge dt$ in those coordinates. Locally, we can write the Hamiltonian term of the perturbation datum on $v$ induced by $p$ (see page \pageref{pdgen2}) as $$K^p_{r,T,\lambda}(v)= F_{r,T,\lambda}(v)ds + G_{r,T,\lambda}(v)dt$$ for domain dependent functions of the form $$F_{r,T,\lambda}(v)_{s,t}, G_{r,T,\lambda}(v)_{s,t}: M\to \R$$ for local coordinates $(s,t)$. Recall that $F_{r,T,\lambda}(v)=G_{r,T,\lambda}(v)=0$ for $v\notin V(T_\textnormal{red})$ by definition of perturbation data for Floer clusters (see page \pageref{pdgen2}). For any $v\in V(T_\textnormal{red})$ we have that on conformal patches \begin{align*}
    \int|du_v-X^{K^p_{r,T,\lambda}(v)}|^2\rho_{r,T,\lambda} \ dsdt = & \int \omega(\partial_su_v, \partial_t u_v) + \omega(X^{F_{r,T,\lambda}(v)}, X^{G_{r,T,\lambda}(v)}) +\\ &+ dF_{r,T,\lambda}(v)(\partial_tu_v)- dG_{r,T,\lambda}(v)(\partial_s u_v)\ ds dt 
\end{align*} holds.
With a couple more calculations and summing over different patches it can be shown that \begin{align*}
    0\leq E(u_v) & = \int_{S_{r,T,\lambda}(v)}|du_v-X^{K^p_{r,T,\lambda}(v)}|^2\rho_{r,T,\lambda} \ dsdt = \\ & =\omega(u_v) + \int_{S_{r,T,\lambda}(v)}d(u_v^*K^p_{r,T,\lambda}(v)) + \int_{S_{r,T,\lambda}(v)} R^p_{r,T,\lambda}(v)\circ u_v
\end{align*}
where $R^p_{r,T,\lambda}(v)\in \Omega^2(S_{r,T,\lambda}(v),C^\infty(M))$ is the so called \textit{curvature form}\label{curvatureform} of $K^p_{r,T,\lambda}(v)$ which in conformal coordinates can be written as $$R^p_{r,T,\lambda}(v) = \left(\partial_sG_{r,T,\lambda}(v)-\partial_t F_{r,T,\lambda}(v) + \omega\left(X^{F_{r,T,\lambda}(v)},X^{G_{r,T,\lambda}(v)}\right)\right) \ ds\wedge dt$$
From this it follows $$0\leq E(u_F) = \omega(u_F) - \int_0^1 H^{L_0,L_d}_p \circ \gamma_+\ dt + \sum  \int_0^1 H^{L_{i-1},L_i}_p \circ \gamma_-^i\ dt + \sum_{v\in V(T)}\int_{S_{r,T,\lambda}(v)} R^p_{r,T,\lambda}(v)\circ u_v$$ and hence $$0\leq E(u) = \omega(u) - \int_0^1 H^{L_0,L_d}_p \circ \gamma_+\ dt + \sum  \int_0^1 H^{L_{i-1},L_i}_p \circ \gamma_-^i\ dt + \sum_{v\in V(T)} \int_{S_{r,T,\lambda}(v)} R^p_{r,T,\lambda}(v)\circ u_v $$ by the definition of $u_F$ on page \pageref{uf}.\\
We call the integral $\int_{S_{r,T,\lambda}(v)} R^p_{r,T,\lambda}(v)\circ u_v $ the \textit{curvature term of the Floer disk $u_v$}, while the sum $$\sum_{v\in V(T)} \int_{S_{r,T,\lambda}(v)} R^p_{r,T,\lambda}(v)\circ u_v$$ will be called the \textit{curvature term of the Floer cluster $u$}.\\
Applying the arguments contained in the proof of Proposition 3.1 in \cite{Biran2021LagrangianCategories} vertexwise, we get that there is a non-empty subset $E^\textnormal{wf}_\textnormal{reg}\subset E_\textnormal{reg}$ of regular perturbation data such that for any $p\in E^\textnormal{wf}_\textnormal{reg}$, any $d\geq 2$, any tuple of Lagrangians $\vec{L}=(L_0,\ldots, L_d)$ and any possible Floer cluster $u$ on $\vec{L}$ of index $0$, the curvature term of $u$ can be absolutely bounded by a term $\epsilon_d= \epsilon_d(p)\in \R$ which does only depend on the number $d$ and on the perturbation data $p\in E^\textnormal{wf}_\textnormal{reg}$. Hence we conclude $$\mathbb{A}_p(\mu_d(\vec{x}_0, \gamma_1, \vec{x}_1,\ldots,\gamma_{d^R}, \vec{x}_{d^R}))\leq \sum_{i=1}^{d^R} \mathbb{A}_p(\gamma_-^i)+ \sum_{i=0}^{d^R}\sum_{j=1}^{m_i}\mathbb{A}_p(x_j^i) + \epsilon_d$$ as desired. For tuples with $L_0=L_d$ the result follows by analogous means. As the (strict) units lie at filtration level $\leq 0$ by definition, this shows that the above defined action functional endows $\mathcal{F}uk(M;p)$ with the structure of weakly-filtered $A_\infty$-category.

\begin{rem}\label{rem214}
\begin{enumerate}
    \item As we will show in Section \ref{actualconstruction}, by choosing the perturbation data carefully we can arrange that the curvature term of a Floer clusters are non-positive.
    \item Proposition \ref{wfilt} applies to the standard (i.e. without Morse trees, see \cite{Seidel2008FukayaTheory}) construction of Fukaya categories too. What is not apparent from the statement of the proposition, but is the reason for which we developed the Morse-Floer model for $\mathcal{F}uk(M)$ is the following trivial but crucial observation: in our case, units lie at filtration level $\leq0$, while this is not true in the standard case, due to the presence of Hamiltonian perturbation data. More about this will be explained in Section \ref{addend}.
\end{enumerate}
\end{rem}

\newpage

\section{$\epsilon$-Perturbation data and filtered structure on $\mathcal{F}uk(M)$}\label{actualconstruction}

 Recall from Section \ref{pdgen} that, given a symplectic manifold $(M,\omega)$ and a choice of family of monotone Lagrangians $\lagmd(M,\omega)$, we denote by $E$ the space of perturbation data for the cluster model of $\mathcal{F}uk(M)$ and by $E_\textnormal{reg}\subset E$ the subset of regular perturbation data, that is those perturbation data leading to a well-defined strictly-unital Fukaya category (see Definition \ref{defEreg}). In this section we construct families of perturbation data $E^\varepsilon_\textnormal{reg}\subset E_\textnormal{reg}$, one for any positive real number $\varepsilon>0$, such that for any $\varepsilon>0$ and any $p\in E^\varepsilon_\textnormal{reg}$ the $A_\infty$-category $\mathcal{F}uk(M,p)$ is filtered, i.e. the discrepancies $\varepsilon_d(p)$ (defined in Section \ref{wfs}) of the maps $\mu_d$ can be taken to be zero (see Section \ref{ainfpre}). In order to achieve this, we will restrict ourselves to perturbation data encapsulating Floer data whose Hamiltonian part is uniformly bounded: this will be the role of the parameter $\varepsilon>0$. The following theorem, which is a more accurate version of Thorem \ref{thmA}, summarizes the results and constructions contained in this section.
\vspace{0.5mm}
\begin{thm}\label{mainthm}
    For any $\varepsilon>0$ there is a non-empty family of perturbation data $E^\varepsilon_\textnormal{reg}\subset E_\textnormal{reg}$ such that for any $p\in E^\varepsilon_{\textnormal{reg}}$ the associated Fukaya category $\mathcal{F}uk(M;p)$ developed in Section \ref{mbfuk} is a filtered $A_\infty$-category.
\end{thm}
\noindent The proof of Theorem \ref{mainthm} occupies Section \ref{d^R=2} and \ref{epspd3}. The idea for the proof is to construct $\varepsilon$-perturbation data for strips and triangles first, and then extend the construction to more complicated clusters via gluing. After the proof of Theorem \ref{mainthm}, we will explain in Section \ref{addend} why it seems necessary to work with a hybrid Floer-Morse model for the Fukaya category in order to obtain an $A_\infty$-category in which the $\mu_d$-operations are filtered and also its units lie in filtration zero.
\begin{rem}
    As hinted above and as it will be clear from the construction, the role of $\varepsilon$ in our construction will be to control the oscillation of Hamiltonians in the Floer data of our Lagrangians. Note that the filtered structure of $\mathcal{F}uk(M,p)$ (and hence the structure of triangulated persistence category \cite{Biran2023TriangulationCategories} of the derived Fukaya category) will depend on the choice of $\varepsilon>0$ and $p\in E^\varepsilon_\textnormal{reg}$ (although in a quantifiable way, see Section \ref{contfunc}). Still, for some fixed choice of $\varepsilon$-perturbation data, the filtered structure on $\mathcal{F}uk(M,p)$ does contain interesting informations about $M$ and $\lagmd(M,\omega)$ and the limit for $\varepsilon\to 0$ can be computed for some invariants arising from filtrations, as we will show in forthcoming work. Anyway we plan to use the technology of continuation functors to define a limit Fukaya category for $\varepsilon\to 0$ in future work.
\end{rem}




\subsection{$\varepsilon$-Perturbation data for $d^R\leq 2$}\label{d^R=2}
Fix $\varepsilon>0$ and $\delta\in \left(\frac{1}{2},1\right)$. The basic idea for the definition of our class $E^\varepsilon$ of perturbation data is to construct homotopies on strip-like ends between Floer data and the zero form, similarly to the case of continuation maps. Note that, as mentioned in Section \ref{mbfuk}, since Morse flowlines and homolomorphic disks do not carry Hamiltonian perturbations, we just have to choose special Floer data for couples made of different Lagrangians and special perturbation data for Floer polygons, which, via gluing, amounts to choose perturbation data for what we will call `fundamental' polygons: strips and $3$-punctured disks (with marked points).
\begin{rem}\label{remind}
From their definition in Section \ref{pdgen}, it is clear that perturbation data on some universal family depend on the a priori choice of (consistent) strip-like ends and system of ends on such family. However, the (explicit or inductive) definition of such perturbation data does not really depend on the choice of ends on a formal level, but only on the fact that such a choice has been done. In other words, if we perturb strip-like ends, then perturbation data change, altough their formal definition does not. 
\end{rem}

Let $\mathcal{F}$ be the set of smooth and increasing functions $\beta: \mathbb{R}\to [0,1]$ such that $\beta\vert_{(-\infty,0]}=0$ and $\beta\vert_{[1,\infty)}=1$. \label{definitionofF}
\begin{defn}
    Let $L_0,L_1\in \lagmd(M,\omega)$ such that $L_0\neq L_1$. An $(\varepsilon,\delta)$-Floer datum for $(L_0,L_1)$ is a choice of Floer datum $(H^{L_0,L_1},J^{L_0,L_1})$ as in Section \ref{pdgen} such that $$\text{image}(H^{L_0,L_1})\subset \left(\delta\varepsilon,\varepsilon\right)$$
Meanwhile, an $(\varepsilon,\delta)$-Floer datum for a couple $(L,L)$ of identical Lagrangians in $\lagmd(M,\omega)$ is just a choice of Floer datum as in Section \ref{pdgen}.
\end{defn}
\noindent Fix a choice of $(\varepsilon,\delta)$-Floer datum for any couple of Lagrangians in $\lagmd(M,\omega)$; all the perturbation data we will define in the following are to be taken with respect to this choice of Floer data.\\

Let $d\geq 2$ and consider a tuple $\vec{L}=(L_0,\ldots, L_d)$ of Lagrangians in $\lagmd(M,\omega)$ such that $d^R\leq 2$. We split the definition of $(\varepsilon,\delta)$-perturbation data for this case in five subcases. Recall that given a universal choice of strip-like ends and an element $(r,T,\lambda)\in \mathcal{R}^{d+1}_C(\vec{L})$ we denote by $\epsilon_i=\epsilon_i^{r,T,\lambda}$ for $i=0,\ldots, d$ the induced strip-like ends on $S_{r,T,\lambda}$ and by $\overline{\epsilon}_i$ for $i=0,\ldots, d^R$ the strip-like ends at positive marked points of type I (i.e. near punctures, see Definition \ref{defoftype}).

\begin{itemize}
    \item[\textbf{Case 1}] Assume first that $\vec{L_\textnormal{red}}=(L_0)$, i.e. $\vec{L}=(L_0,\ldots, L_0)$. Then a choice of $(\varepsilon,\delta)$-perturbation datum for $\vec{L}$ is just a choice of perturbation datum for $\vec{L}$ in the sense of Definition \ref{maindefofpd}.
    \item[\textbf{Case 2}] Assume now that $\vec{L_\textnormal{red}}=(\overline{L_0},\overline{L_1})$ and $m_0^e(\vec{L})=0$, that is $L_0\neq L_d$ (see page \pageref{defofmis} for the definition of the numbers $m_i$, $m_0^b$ and $m_0^e$). Then a choice of $(\varepsilon,\delta)$-perturbation datum for $\vec{L}$ is just a choice of perturbation datum for $\vec{L}$.
    
    \item[\textbf{Case 3}]\label{case3} Assume now that $\vec{L_\textnormal{red}}=(\overline{L_0},\overline{L_1})$ and $m_0^e(\vec{L})>0$, that is, in particular, $L_0=L_d$. Then a choice of $(\varepsilon,\delta)$-perturbation datum for $\vec{L}$ is a choice $\left((f^{\vec{L}},g^{\vec{L}}),(K^{\vec{L}},J^{\vec{L}})\right)$ of perturbation datum for $\vec{L}$ such that for any $(r,T,\lambda)\in \mathcal{R}^{d+1}_C(\vec{L})$ and for the unique vertex $v\in V(T_\textnormal{red})$ of $T_\textnormal{red}$ we have that: \begin{enumerate}
    \item $K^{r,T,\lambda}_v$ vanishes away from the strip-like end $\overline{\epsilon_i}$ for any $i\in \{0,1\}$;
    \item On the $i$th ($i=0,1,2$) negative strip-like end $\overline{\epsilon_i}$ we have $$K^{r,T,\lambda}_v=\overline{H^{L_{i-1},L_i}_{r,T,\lambda}}(s,t)dt$$ where $$\overline{H^{L_{i-1},L_i}_{r,T,\lambda}}: (-\infty,0]\times [0,1]\times M \longrightarrow \R$$ is of the form $$\overline{H^{L_{i-1},L_i}_{r,T,\lambda}}(s,t)= (1-\beta^{L_{i-1},L_i}_{r,T,\lambda}(s+1))H^{L_{i-1},L_i}_t$$ for some $\beta^{L_{i-1},L_i}_{r,T,\lambda}\in \mathcal{F}$.
\end{enumerate}
    \item[\textbf{Case 4}] Assume now that $\vec{L_\textnormal{red}}=(\overline{L_0},\overline{L_1}, \overline{L_2})$ and $m_0^e(\vec{L})=0$. A choice of $(\varepsilon,\delta)$-perturbation for $\vec{L}$ is a choice $\left((f^{\vec{L}},g^{\vec{L}}),(K^{\vec{L}},J^{\vec{L}})\right)$ of perturbation datum for $\vec{L}$ such that for any $(r,T,\lambda)\in \mathcal{R}^{d+1}_C(\vec{L})$ and for the unique vertex $v\in V(T_\textnormal{red})$ of $T_\textnormal{red}$ we have that: \begin{enumerate}
    \item $K^{r,T,\lambda}_v$ vanishes away from the strip-like end $\overline{\epsilon_i}$ for any $i\in \{0,1,2\}$;
    \item On the $i$th ($i=1,2$) negative strip-like end $\overline{\epsilon_i}$ we have $$K^{r,T,\lambda}_v=\overline{H^{L_{i-1},L_i}_{r,T,\lambda}}(s,t)dt$$ where $$\overline{H^{L_{i-1},L_i}_{r,T,\lambda}}: (-\infty,0]\times [0,1]\times M \longrightarrow \R$$ is of the form $$\overline{H^{L_{i-1},L_i}_{r,T,\lambda}}(s,t)= (1-\beta^{L_{i-1},L_i}_{r,T,\lambda}(s+1))H^{L_{i-1},L_i}_t$$ for some $\beta^{L_{i-1},L_i}_{r,T,\lambda}\in \mathcal{F}$;
    \item On the unique positive strip-like end we have $$K^{r}_v=\overline{H^{L_{0},L_d}_{r,T,\lambda}}dt$$ where $$\overline{H^{L_{0},L_d}_{r,T,\lambda}}: [0,+\infty)\times [0,1]\times M \longrightarrow \R$$ is of the form $$\overline{H^{L_0,L_d}_{r,T,\lambda}}(s,t)= \beta^{L_0,L_2}_{r,T,\lambda}(s)H^{L_0,L_2}_t$$ for some $\beta^{L_0,L_2}_{r,T,\lambda}\in \mathcal{F}$.
\end{enumerate}
    \item[\textbf{Case 5}] Assume now that $\vec{L_\textnormal{red}}=(\overline{L_0},\overline{L_1}, \overline{L_2})$ and $m_0^e(\vec{L})>0$. A choice of $(\varepsilon,\delta)$-perturbation for $\vec{L}$ is a choice $\left((f^{\vec{L}},g^{\vec{L}}),(K^{\vec{L}},J^{\vec{L}})\right)$ of perturbation datum for $\vec{L}$ such that for any $(r,T,\lambda)\in \mathcal{R}^{d+1}_C(\vec{L})$ and for the unique vertex $v\in V(T_\textnormal{red})$ of $T_\textnormal{red}$ we have that: \begin{enumerate}
    \item $K^{r,T,\lambda}_v$ vanishes away from the strip-like ends $\overline{\epsilon_i}$ for $i\in \{0,1,2\}$;
    \item On the $i$th ($i=0,1,2$) negative strip-like $\overline{\epsilon_i}$ end we have $$K^{r,T,\lambda}_v=\overline{H^{L_{i-1},L_i}_{r,T,\lambda}}dt$$ where $$\overline{H^{L_{i-1},L_i}_{r,T,\lambda}}: (-\infty,0]\times [0,1]\times M \longrightarrow \R$$ is of the form $$\overline{H^{L_{i-1},L_i}_{r,T,\lambda}}(s,t)= (1-\beta^{L_{i-1},L_i}_{r,T,\lambda}(s+1))H^{L_{i-1},L_i}_t$$ for some $\beta^{L_{i-1},L_i}_{r,T,\lambda}\in \mathcal{F}$. \end{enumerate}
    
\end{itemize}

     \begin{center}
     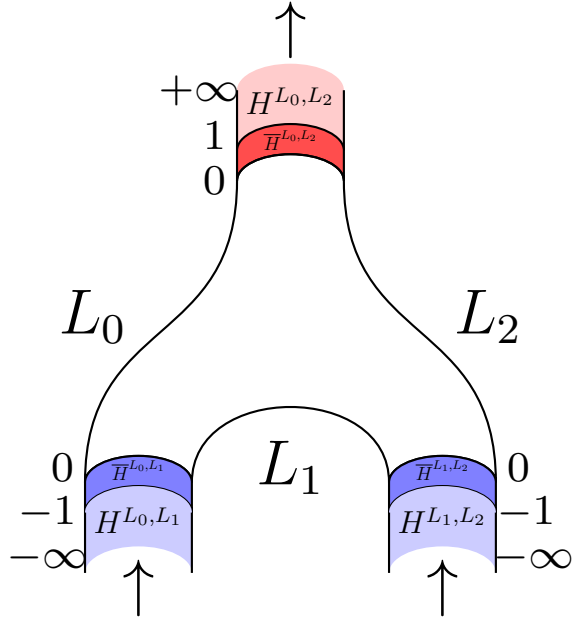
\begin{figure}
         \centering
 \scalebox{2}{
        \begin{tikzpicture}
     \pic[tqft,
     incoming boundary components=1,
     outgoing boundary components=1,
     cobordism edge/.style={draw},
     incoming boundary components/.style={draw},
     fill=red!20,
     name=c,
     cobordism height=0.4cm,
     ];
    \pic[tqft,
     incoming boundary components=1,
     outgoing boundary components=1,
     fill=red!70,
     name=b,
     draw,
     cobordism height=0.2cm, anchor=incoming boundary, at=(c-outgoing boundary)
         ];
         
    \pic[tqft/pair of pants,name=a, anchor=incoming boundary, at=(b-outgoing boundary), draw,];
    
    \pic[tqft,
     incoming boundary components=1,
     outgoing boundary components=1,
     fill=blue!50,
     name=d1,
     draw,
     cobordism height=0.2cm, anchor=incoming boundary, at=(a-outgoing boundary 1)
              ];
      \pic[tqft, 
     incoming boundary components=1,
     outgoing boundary components=1,
     fill=blue!50,
     name=d2,
     draw,
     cobordism height=0.2cm, anchor=incoming boundary, at=(a-outgoing boundary 2)
     
         ];
     \pic[tqft,
     incoming boundary components=1,
     outgoing boundary components=1,
     cobordism edge/.style={draw},
     fill=blue!20,
     name=e1,
     cobordism height=0.4cm, anchor=incoming boundary, at=(d1-outgoing boundary)
         ];
         
    \pic[tqft,
     incoming boundary components=1,
     outgoing boundary components=1,
     cobordism edge/.style={draw},
     fill=blue!20,
     name=e2,
     cobordism height=0.4cm, anchor=incoming boundary, at=(d2-outgoing boundary),
         ];
        
    \draw node at (-1.3,-1.5){$L_0$};
    \draw node at (1.3,-1.5){$L_2$};
    \draw node at (0,-2.5){$L_1$};
    \draw node at (-0.6,0){\tiny $+\infty$};
    \draw node at (-0.5,-0.3){\tiny $1$};
    \draw node at (0,0.4){ $\uparrow$};
    \draw node at (-1,-3.3){ $\uparrow$};
    \draw node at (1,-3.3){ $\uparrow$};
    \draw node at (1,-2.85){\scalebox{.5}{$H^{L_1,L_2}$}};
    \draw node at (-1,-2.85){\scalebox{.5}{$H^{L_0,L_1}$}};
    \draw node at (-0.5,-0.6){\tiny $0$};
    \draw node at (0,-0.07){\scalebox{.5}{$H^{L_0,L_2}$}};
    \draw node at (0,-0.33){\scalebox{.3}{$\overline{H}^{L_0,L_2}$}};
    \draw node at (1,-2.52){\scalebox{.3}{$\overline{H}^{L_1,L_2}$}};
    \draw node at (-1,-2.52){\scalebox{.3}{$\overline{H}^{L_0,L_1}$}};
    \draw node at (-1.5,-2.5){\tiny $0$};
    \draw node at (-1.6,-2.8){\tiny $-1$};
    \draw node at (-1.6,-3.1){\tiny $-\infty$};
    \draw node at (1.5,-2.5){\tiny $0$};
    \draw node at (1.55,-2.8){\tiny $-1$};
    \draw node at (1.6,-3.1){\tiny $-\infty$};

        \end{tikzpicture}}

         \caption{\textbf{Case 4}: A schematic representation of the Hamiltonian part a choice of $(\varepsilon,\delta)$-perturbation datum on a punctured disk in $\mathcal{R}^3_C(\vec{L})$ for a triple $\vec{L}=(L_0,L_1,L_2)$ made of different Lagrangians. White corresponds to vanishing Hamiltonian perturbation.}
         \label{fig:my_label}
     \end{figure}
   
    \end{center}

Before defining $(\varepsilon,\delta)$-perturbation data for tuples of Lagrangians in $\lagmd(M,\omega)$ with reduced tuple of length greater than $4$, we stop for a moment to explain why the above construction is useful for the filtration point of view. Until the end of this subsection we assume transversality of all Floer and perturbation data. The map $\mu_d$ associated with tuples of Lagrangians as in \textbf{Case 1} and \textbf{Case 2} above preserves action filtrations, as discussed in Section \ref{wfs}. Let $\vec{L}$ be a tuple of Lagrangians in $\lagmd(M,\omega)$ as in \textbf{Case 4} above. As it will be apparent from the following computations, this is the most delicate case of the three remaining ones, as it is the only one with a positive contribution to the curvature term coming from the (unique) positive strip-like end. For simplicity, we will assume that $d^R=d=3$; the treatment of the general case is similar to this one, the only complication is a pure formalism: the presence of quantum-Morse trees, which do not interfere with curvature terms of Floer polygons. \\ Let $(K^{\vec{L}}, J^{\vec{L}})$ be a choice of $(\varepsilon,\delta)$-perturbation data for $\vec{L}$. We show that the associated map $$\mu_2: CF(L_0,L_1)\otimes CF(L_1,L_2)\to CF(L_0,L_2)$$ preserves filtrations. Consider orbits $\gamma_1\in \mathcal{O}(H^{L_0,L_1})$, $\gamma_2\in \mathcal{O}(H^{L_1,L_2})$ and $\gamma^+\in \mathcal{O}(H^{L_0,L_2})$ and a class $A\in \pi_2(M,\vec{L})$, and assume there is a Floer polygon with respect to the above chosen $(\varepsilon,\delta)$-perturbation data connecting $\gamma_1^-$ and $\gamma_2^-$ to $\gamma^+$ in the class $A$. Note that $\mathcal{R}^{3}_C(\vec{L})=\mathcal{R}^{3}(\vec{L})$ is a singleton (made of a $3$-punctured disk with no marked points), say $\{pt\}$ we write $K^{\vec{L}}= K^{pt}_v$ and omit $pt\in \mathcal{R}^{3}(\vec{L})$ from the notation from now on. We estimate the cuvature term $\int_SR^{K^{\vec{L}}}\circ u$ of $K^{\vec{L}}$ on $u$. For any $i\in\{0,1,2\}$ we denote by $\int_iR^{K^{\vec{L}}}\circ u$ the integral over the strip-like end $\overline{\epsilon_i}$ of $S:=S_{pt}$. By definition of $(\varepsilon,\delta)$-perturbation data, we have the equality $$\int_S R^{K^{\vec{L}}}\circ u = \sum_{i=0}^3 \int_iR^{K^{\vec{L}}}\circ u$$ Consider $i\in \{1,2\}$ first, then with respect to the coordinates on the respective strip-like ends (which are the negative one) we have

\begin{align*} \int_iR^{K^{\vec{L}}}\circ u &= \int_0^1\int_{-\infty}^0 R^{K^{\vec{L}}}\circ u = \int_0^1\int_{-1}^0 R^{K^{\vec{L}}}\circ u = \int_0^1\int_{-1}^0 \partial_s \overline{H^{L_0,L_1}_{s,t}}\circ u\ dsdt \\ & = -\int_0^1\int_0^1 \partial_s\beta^{L_0,L_1}H^{L_0,L_1}_t\circ u \ ds dt\leq -\int_0^1\int_0^1\partial_s\beta^{L_0,L_1}\ ds \ \min_MH^{L_0,L_1}_t\ dt\\ & < -\delta\varepsilon     \end{align*}

\noindent On the other hand, for $i=0$, i.e. the unique positive end, we have
\begin{align*} \int_0 R^{K^{\vec{L}}}\circ u &= 
        \int_0^1\int_0^{+\infty} R^{K^{\vec{L}}}\circ u = \int_0^1\int_{0}^1 R^{K^{\vec{L}}}\circ u = \int_0^1\int_{0}^1 \partial_s \overline{H^{L_0,L_2}_{s,t}}\circ u\ dsdt \\ & = \int_0^1\int_{0}^1 \partial_s\beta^{L_0,L_2}H^{L_0,L_2}_t\circ u \ ds dt \leq 
        \int_0^1\int_0^1\partial_s\beta^{L_0,L_2}\ ds \ \max_MH^{L_0,L_2}_t\ dt\\ & < \varepsilon
    \end{align*}
Hence, we conclude $$\int_SR^{K^{\vec{L}}}\circ u<\varepsilon(1-2\delta)<0$$ as $\delta>\frac{1}{2}$ by assumption. By the computations carried out in Section \ref{wfs} and the definition of the filtrations on Floer complexes via the action functional we conclude that $\mu_2$ is a filtration preserving map. As explained above, this implies that the $\mu_d$-maps defined for tuples of Lagrangians as in \textbf{Case 4} via $(\varepsilon,\delta)$-perturbation data are filtered (if we assume transversality). Moreover, similar (but easier) computations to the ones presented above show that the $\mu_d$-maps for tuples of Lagrangians as in \textbf{Case 3} and \textbf{Case 5} defined via $(\varepsilon,\delta)$-perturbation data are filtered too, because Floer polygons for such tuples only have entries and no exit (the exit will be a Morse flowline).

\subsection{$\varepsilon$-Perturbation data for $d^R>2$}\label{epspd3}
Let $\varepsilon>0$, $\delta\in \left(\frac{1}{2},1\right)$ and $d^R>2$. We assume that $(\varepsilon,\delta)$-perturbation data have been defined for any tuple of Lagrangians in $\lagmd(M,\omega)$ such that its reduced tuples has length $<d^R+1$. Let $\vec{L}=(L_0,\ldots,L_d)$ be a tuple of Lagrangians in $\lagmd(M,\omega)$ such that its reduced tuple $\vec{L_\textnormal{red}}$ has length $d^R+1$. Near the boundary $\partial \mathcal{R}^{d+1}_C(\vec{L})$ of $\overline{\mathcal{R}^{d+1}_C(\vec{L})}$, more precisely on an union of neighbourhoods of boundary strata where the associated gluing maps are diffeomorphisms onto their image, we define $(\varepsilon,\delta)$-perturbation data to be images of lower order $(\varepsilon,\delta)$-perturbation data for different tuples of Lagrangians under the gluing maps. In particular, near vertices of $\overline{\mathcal{R}^{d+1}_C(\vec{L})}$ $(\varepsilon,\delta)$-perturbation data are obtained by gluing of `fundamental' $(\varepsilon,\delta)$-perturbation data, i.e. those we explicitly defined in Section \ref{d^R=2} for strips and $3$-punctured disks with marked points. Note that consistency of strip-like ends and of system of ends implies that this construction is well-defined. Before extending the definition to the whole $\mathcal{R}^{d+1}_C(\vec{L})$ we make the following (obvious) remark, which is however fundamental from the point of view of our aim to achieve a filtered Fukaya category.

\begin{rem}\label{thinparts1}
 A first trivial observation is that Floer clusters defined on clusters lying near vertices of $\partial \mathcal{R}^{d+1}_C(\vec{L})$ are the most problematic from the filtration point of view, as they are the ones with most inherited positive strip-like ends on thin parts (which contribute positively to the total curvature term). The further away we go from vertices, the less positive contributions to the total curvature term a Floer cluster will inherit. In this remark, we want to show that, although being the most problematic case, cluster near vertices have negative curvature term. Assume first $L_0\neq L_d$. For simplicity we will assume that $\vec{L}$ is cyclically different, but the following generalizes to any tuple with $L_0\neq L_d$, as Morse flowlines and holomorphic disks do not contribute to curvature terms. Let $T\in \mathcal{T}_U^{d+1}(\vec{L})$ such that $|V(T_\textnormal{red})|=1$ (note that in this case $T_\textnormal{red}$ has $d^R+1=d+1$ exterior edges, among which exactly one is outgoing), and assume that $(r,T,\lambda)\in \mathcal{R}^{d+1}_C(\vec{L})$ lies near $\partial \mathcal{R}^{d+1}_C(\vec{L})$ in the sense above. Then $S_{r,T,\lambda}$ has at most $2d-1$ thin parts ($d+1$ `exterior' ones plus $d-2$ `interior' ones, one for each codimension of $\mathcal{R}^{d+1}_C(\vec{L})$). Indeed, $S_{r,T,\lambda}$ has $2d-1$ thin parts exactly when $(r,T,\lambda)$ lies near a vertex of $\overline{\mathcal{R}^{d+1}_C(\vec{L})}$. In this case (assuming transversality), a Floer polygon defined on $S_{r,T,\lambda}$ endowed with the above defined $(\varepsilon,\delta)$-perturbation data carries a curvature term strictly bounded above by $$\varepsilon-d\delta\varepsilon+(d-2)\left(\varepsilon-\delta\varepsilon\right)= (d-1)\varepsilon(1-2\delta)<0$$ as $\delta>\frac{1}{2}$ (the $\varepsilon-\delta\varepsilon$ term coming from thin parts in the interior). Note that the last term goes to $0$ as $\delta\to \frac{1}{2}$.\\
    Assume now $L_0=L_d$. For simplicity we will assume that $\vec{L}$ is almost cyclically different, i.e. $L_i\neq L_{i+1}$ for any $i\in \{0,\ldots, d-1\}$, but the following generalizes to any tuple with $L_0= L_d$, as Morse flowlines and holomorphic disks do not contribute to curvature terms. As evident from the discussion in Section \ref{d^R=2}, this case is less problematic from a filtration point of view. Let $T\in \mathcal{T}_U^{d+1}(\vec{L})$ such that $|V(T_\textnormal{red})|=1$ (note that in this case $T_\textnormal{red}$ has $d^R=d$ exterior edges, all of which are oriented towards the only vertex of $T_\textnormal{red}$), and assume that $(r,T,\lambda)\in \mathcal{R}^{d+1}_C(\vec{L})$ lies near $\partial \mathcal{R}^{d+1}_C(\vec{L})$. Then $S_{r,T,\lambda}$ has at most $2(d-1)-1=2d-3$ thin parts. Indeed, $S_{r,T,\lambda}$ has $2d-3$ thin parts exactly when $(r,T,\lambda)$ lies near a vertex of $\overline{\mathcal{R}^{d+1}_C(\vec{L})}$. In this case (assuming transversality), a Floer polygon defined on $S_{r,T,\lambda}$ endowed with the above defined $(\varepsilon,\delta)$-perturbation data carries a curvature term bounded above by $$-d\delta\varepsilon+(d-3)\left(\varepsilon-\delta\varepsilon\right)= (d-3)\varepsilon(1-2\delta)-3\delta\varepsilon<0$$ Note that the last term goes to $-\frac{3}{2}\varepsilon$ as $\delta\to \frac{1}{2}$. We conclude that all Floer clusters resulting from gluing carry negative curvature terms.
\end{rem}

\noindent In view of the above remark, after having defined $(\varepsilon,\delta)$-perturbation data near the boundary of $\mathcal{R}^{d+1}_C(\vec{L})$, we interpolate on the whole $\mathcal{R}^{d+1}_C(\vec{L})$ while keeping the requirement that the total curvature term is non-positive.\\

Summing things up, we just defined a family $E^{\varepsilon,\delta}\subset E$ of consistent perturbation data for any positive real $\varepsilon>0$ and any $\delta\in \left(\frac{1}{2},1\right)$. Indeed, notice that the inductive definition of $(\varepsilon,\delta)$-perturbation data implies that we get consistency for free. Moreover, our discussions above imply that, assuming transversality, for any $\varepsilon>0$, any $\delta\in \left(\frac{1}{2},1\right)$ and any $p\in E^{\varepsilon,\delta}$ the associated Fukaya category $\mathcal{F}uk(M;p)$ is a strictly unital and filtered $A_\infty$-category. We set $$E^\varepsilon:=\bigcup_{\frac{1}{2}<\delta<1} E^{\varepsilon,\delta}$$ and refer to elements of $E^\varepsilon$ as $\varepsilon$-perturbation data. In the remaining of the section we discuss transversality of our perturbation data, which is the last ingredient missing in the proof of Theorem \ref{mainthm}. We recall that a perturbation datum $p\in E$ is regular if all possible Floer clusters of index $0$ and $1$ defined via $p$ are regular in the sense of Definition \ref{defsimple}. Regularity of quantum trees is generic for $(\varepsilon,\delta)$-perturbation data, exactly as explained in the general case in Section \ref{transversality for floer clusters}. It only remains to deal with regularity of Floer polygons defined via $(\varepsilon,\delta)$-perturbation data. \\
Let $d\geq 2$, $\vec{L}=(L_0,\ldots, L_d)$ be a tuple of Lagrangians in $\lagmd(M,\omega)$ and $(K, J)$ the Floer part of $p$ restricted to $\mathcal{R}^{d+1}_C(\vec{L})$. We generalize \cite[Chapter 9k]{Seidel2008FukayaTheory} and define an admissible deformation of $(K, J)$ as a couple $$(\Delta K, \Delta J)\in \left(\Omega^1(S_{r,T,\lambda}(v), C^\infty(M))\times C^\infty(S_{r,T,\lambda}(v), T_J\mathcal{J})\right)_{v\in V(T_\textnormal{red}), (r,T,\lambda)\in \mathcal{R}^{d+1}_C(\vec{L})}$$ (where $\mathcal{J}$ is the space of $\omega$-compatible almost complex structures on $M$) smooth in the $\mathcal{R}^{d+1}_C(\vec{L})$ direction and such that $(\Delta K^{r,T,\lambda}_v, J^{r,T,\lambda}_v)$ (that is, the couple $(\Delta K, \Delta J)$ restricted to the polygon corresponing to the vertex $v\in V(T_\textnormal{red})$ on the cluster $S_{r,T,\lambda}$) is supported on the thick parts of the polygons $S_{r,T,\lambda}(v)$ and $\Delta K^{r,T,\lambda}_v$ vanishes along vectors tangent to the boundary of $S_{r,T,\lambda}(v)$. The deformation of $(K, J)$ via $(\Delta K, \Delta J)$ is defined on $S_{r,T,\lambda}(v)$, for $(r,T,\lambda)\in \mathcal{R}^{d+1}_C(\vec{L})$ and $v\in V(T_\textnormal{red})$ as $$\left(K^{r,T,\lambda}_v+ \Delta K^{r,T,\lambda}_v, J^{r,T,\lambda}_v\exp(-J^{r,T,\lambda}_v\Delta J^{r,T,\lambda}_v)\right).$$ This way we defined the concept of admissible Hamiltonian deformation for the elements of $E$. Note that the requirement that $(\Delta K, \Delta J)$ is supported on thick parts of polygons is fundamental in order to keep consistency for the deformed Hamiltonian perturbation datum.

 \begin{defn}\label{transversalityepsilon}
        We define $E^{\varepsilon,\delta}_\textnormal{reg}\subset E_\textnormal{reg}$ to be the space of regular perturbation data $p\in E_\textnormal{reg}$ such that the Floer parts of $p$ are obtained via an admissible deformation of the Floer parts of some $(\varepsilon,\delta)$-perturbation datum $q\in E^{\varepsilon,\delta}$ and such that the associated (well-defined) Fukaya category $\mathcal{F}uk(M;p)$ is filtered. Moreover, we define $$E^\varepsilon_\textnormal{reg}:=\bigcup_{\frac{1}{2}<\delta<1} E^{\varepsilon,\delta}_\textnormal{reg}$$ and refer to its elements as regular $\varepsilon$-perturbation data.
\end{defn}

\noindent To conclude the proof of Theorem \ref{mainthm} we show that for any $\varepsilon>0$ the space $E^\varepsilon_\textnormal{reg}$ is non-empty. Let $q\in E^{\varepsilon,\delta}$ for some $\varepsilon>0$ and $\delta\in \left(\frac{1}{2},1\right)$. Let $d\geq 2$, $\vec{L}=(L_0,\ldots, L_d)$ be a tuple of Lagrangians in $\lagmd(M,\omega)$ and $(K, J)$ the Floer part of $p$ restricted to $\mathcal{R}^{d+1}_C(\vec{L})$. As proved\footnote{To see that Seidel's proof keep on working in the cluster setting it is enough to apply it to any vertex of $T_\textnormal{red}$.} in \cite[Chapter 9k]{Seidel2008FukayaTheory}, a generic admissible deformation $(\Delta K, \Delta J)$ of $(K, J)$ turns $(K, J)$ into a regular Hamiltonian perturbation datum on $\mathcal{R}^{d+1}_C(\vec{L})$. Thus we can choose a generic Hamiltonian $1$-form $\Delta K$ supported on the thick parts of punctured disks such that the associated function $$(r,T,\lambda)\in \mathcal{R}^{d+1}_C(\vec{L})\longmapsto \sum_{v\in T_{\textnormal{red}}}\int_{S_{r,T,\lambda}(v)}\max_{x\in M}R^{\Delta K}_{r,T,\lambda}(v)$$ is bounded above by $$(d-1)\varepsilon(2\delta-1)$$ implying that the deformed Hamiltonian $K+\Delta K$ has a negative curvature term overall, as the supports of $K$ and $\Delta K$ are disjoint. In particular, the map $\mu_d$ defined via $K+\Delta K$ is filtered. The case $L_0=L_d$ is analogous.\\
This concludes the proof of Theorem \ref{mainthm}

\begin{rem}
\begin{enumerate}
\item The above recipe for transversality allows us to define $\mu_d$-maps shifting filtration by $\leq 0$. By restricting the possible choice of deformations to the $\Delta K$ with arbitrary small curvature, say with curvature function bounded by a small number $\eta>0$, we obtain for any $d\geq 2$ that the map $\mu_d$ shifts filtration by $$\leq (d-1)\varepsilon(1-2\delta)+\eta.$$ This control over the negative shift might be useful in some situations.
\item Note that in general $E^{\varepsilon,\delta}_\textnormal{reg}\not\subset E^{\varepsilon,\delta}$. However, we conjecture that it is possible to modify the above transversality argument to define a residual subset $\tilde{E}^{\varepsilon,\delta}_\textnormal{reg}\subset E^{\varepsilon,\delta}$ of regular $(\varepsilon,\delta)$-perturbation data. The idea is to fix the Floer part of the perturbation data and, via an extension of the arguments contained in \cite{Mcduff2012J-holomorphicTopology}, to show that there are adequate deformations of the almost complex structures turning $(\varepsilon,\delta)$-perturbation data into regular ones.  Moreover, in that case we could define $(\varepsilon,\delta)$-perturbation data for $\delta=\frac{1}{2}$ too.
\end{enumerate}
\end{rem}

\subsection{Observation: the case of the unit in the standard model of Fukaya categories}\label{addend} In this subsection we explain the main reason for working with a hybrid Floer-Morse (or `cluster') model for the Fukaya category, in order to obtain a genuinely filtered $A_\infty$-category and not only a weakly-filtered one, as in \cite{Biran2021LagrangianCategories}.

We recall that in the standard model of the Fukaya category (that is, the one presented in \cite{Seidel2008FukayaTheory}, see e.g. \cite{Biran2014LagrangianCategories} for a monotone version) \textit{all} Floer chain complexes are defined via Hamiltonian perturbations (also the ones associated to couple of identical Lagrangians) and the $\mu_d$-maps are defined by counting Floer polygons joining Hamiltonian orbits (and no Morse trees). Note that the construction of the space $E^\varepsilon$ of $\varepsilon$-perturbation data generalizes without much effort to this standard model and it is easy to see that the associated maps $\mu_d$ are filtered in this case too. However, one of the main differences between the hybrid Floer-Morse model defined in this paper and the standard one is that the definition of the (representatives of) the units in the latter involves counting homolomorphic disks with Hamiltonian perturbations, and this may a priori lead to curvature terms taking positive values and hence to representatives of units that may lie at positive filtration levels (see \cite[Proposition 3.1(ii)]{Biran2021LagrangianCategories}). Indeed, in the following we sketch a proof of the fact that one \textit{cannot }define representatives of the units lying at vanishing filtration level using $\varepsilon$-perturbation data on the standard model of $\mathcal{F}uk(M)$. The same proof can be expanded to show that it is not possible to have a filtered Fukaya category via the standard model: one has to give up either filtering the maps $\mu_d$ or having units at vanishing filtration levels. 
\begin{rem} 
Not having filtered units may seem a marginal fact compared to not having filtered $\mu_d$-maps. However, upcoming work will show why filtration-zero units are desiderable, leading to a `minimal energy' Fukaya category and interesting results. 
Moreover, for the derived Fukaya category of a monotone symplectic manifold to fit the definition of a triangulated persistence category (TPC) (see the recent work \cite{Biran2023TriangulationCategories}), vanishing filtration levels for units are crucial.
\end{rem}
We briefly recall how representatives of the units are constructed in \cite{Seidel2008FukayaTheory}. Let $\varepsilon>0$, $\delta\in \left(\frac{1}{2},1\right)$ and $L\in \lagmd(M,\omega)$ be a monotone Lagrangian. Choose a Floer datum $(H^{L,L},J^{L,L})$ for $(L,L)$ in the sense of Seidel, i.e. a time-dependent Hamiltonian $H^{L,L}$ on $M$ such that $\varphi_H^1(L)\pitchfork L$, where $\varphi_H^t$ denotes the Hamiltonian flow of $H$, and an $\omega$-compatible time dependent almost complex structure $J^{L,L}$ on $M$. We assume that $(H^{L,L},J^{L,L})$ fits the definition of a $(\varepsilon,\delta)$-Floer datum, i.e. $\textnormal{image}(H^{L,L})\subset \left(\delta\varepsilon, \varepsilon\right)$. We define $$CF^S(L,L):=\bigoplus_{\gamma\in \mathcal{O}(H^{L,L})}\Lambda \cdot \gamma$$ where the superscript $S$ stands for `standard', $\mathcal{O}(H^{L,L})$ denotes the set of Hamiltonian orbits of $H^{L,L}$ and $\Lambda$ the Novikov field over $\Z_2$. We further assume that $(H^{L,L},J^{L,L})$ is regular, i.e. $CF^S(L,L)$ is a chain complex when endowed with the standard differential counting Floer strips.
Consider the standard unit disk $D\subset \mathbb C$, define $D_1:=D\setminus\{1\}$ and pick a positive strip-like end (see page \pageref{striplikends}) on $D_1$ near the point $1\in D$. We define $\varepsilon$-perturbation data for $L$ on $D_1$ analogously to $\varepsilon$-perturbation data for $(d+1)$-punctured disks where $d\geq 2$ as introduced in Section \ref{d^R=2} and \ref{epspd3}. The notion of curvature term is also defined analogously. Let $(K^L,J^L)$ be such a perturbation datum. Let $\gamma\in \mathcal{O}(H^{L,L})$ and $A\in \pi_2(M,L)$ such that $\mu(\gamma;A)=0$ and consider the space $\mathcal{M}(\gamma;A)$ of $(K^L,J^L)$-perturbed Floer $1$-gons $u:D_1\to M$ such that on the unique strip-like end of $D_1$ we have $$\lim_{s\rightarrow \infty}u(s,t)=\gamma(t).$$ We assume regularity of the perturbation datum $(K^L,J^L)$, so that $\mathcal{M}(\gamma;A)$ is a smooth manifold of dimension $0$ for any orbit $\gamma\in \mathcal{O}(H^{L,L})$ and any class $A\in \pi_2(M,L)$ as above. Standard Gromov-compactness arguments show that in such cases $\mathcal{M}(\gamma;A)$ is compact. We define $$e_L:=\sum_{\gamma,A}\sum_u T^{\omega(u)}\gamma\in CF^S(L,L)$$ where the first sum runs over classes $\gamma\in \mathcal{O}(H^{L,L})$ and classes $A\in \pi_2(M,L)$ such that $\mu(\gamma;A)=0$ and the second sum runs over Floer $1$-gons $u\in \mathcal{M}(\gamma;A)$. It is well-known that $e_L$ is a representative of the homological unit of $L$ in the standard model of $\mathcal{F}uk(M)$ (see \cite[Section 8d]{Seidel2008FukayaTheory}).

In the remaining of this section we sketch a proof of the fact that $e_L$ lies at a positive filtration level in $CF^S(L,L)$, that is $\mathbb{A}(e_L)>0$. Assume by contradiction that that $\mathbb{A}(e_L)\leq 0$. Then by definition of the action functional $\mathbb A$ (see page \pageref{actionfunctionaldefinition}) we have $$C:=\max_{u\in \mathcal{M}_0(\gamma)}\int_{D_1}R^{K^L}\circ u\leq 0$$ where $R^{K^L}$ denotes the curvature two form (introduced on page \pageref{curvatureform}) associated to the Hamiltonian perturbation datum $K^L$.
Let now $\tilde{H}^{L,L}$ be another Hamiltonian on $M$ such that $\varphi_{\tilde{H}^{L,L}}(L)\pitchfork L$ and assume that $0<\tilde{H}<H$. It is well known that (assuming regularity) the Floer homology defined using $(H^{L,L},J^{L,L})$ is quasi-isomorphic as a chain complex to the one defined using $(\tilde{H}^{L,L},J^{L,L})$. However, those two homologies are \textit{not} isomorphic as persistence modules, as the filtrations at the chain level may differ dramatically in general. However we sketch a proof of the fact that given our assumption on $\mathbb{A}(e_L)$ we can construct an isomorphism of persistence modules, leading to a contradiction. \\
It is easy to see that there is a filtration preserving chain map $CF(L,L,H)\to CF(L,L,\tilde{H})$ (e.g. by choosing a monotone homotopy from $H$ to $\tilde{H}$, see \cite{Biran2001PropagationHomology}). We construct a filtration preserving map $$\psi: CF(L,L,\tilde{H})\to CF(L,L,H)$$ using our assumption on $\mathbb{A}(e_L)$. Pick $D_1$ and puncture it in $-1$ to get $D_{-1,1}\cong \mathbb{R}\times [0,1]$. Then, we can see $K^L$ as a one-form on $\mathbb{R}\times [0,1]$ such that $K^L=0$ (on a strip-like end) near $-\infty$. We define $\psi$ as the continuation map associated with the Hamiltonian perturbation defined as the concatenation of a monotone homotopy from $\tilde{H}$ to $0$ and $K^{L}$ (seen as a one-form on the strip). Standard methods show that $\psi$ is a chain map. Moreover, as $\tilde{H}$ is positive, the monotone homotopy has negative curvature term, and, as $K$ has non-positive curvature term, energies idendities similar to the ones in Section \ref{d^R=2} tell us that the map $\psi$ is a filtered chain map. Hence it follows that $HF(L,L,H)$ and $HF(L,L,\tilde{H})$ are isomorphic as persistence modules. This contradicts the assumption $\tilde{H}<H$ and hence proves that the representative $e_L\in CF^S(L,L)$ of the unit satisfies $\mathbb{A}(e_L)>0$, as claimed.

 \begin{figure}
         \centering
 \scalebox{1.8}{
        \begin{tikzpicture}
        [
    every tqft/.append style={
        tqft/boundary separation=3cm
      }
  ]
     \pic[tqft,
     incoming boundary components=1,
     outgoing boundary components=1,
     cobordism edge/.style={draw},
     incoming boundary components/.style={draw},
     fill=red!20,
     name=c,
     cobordism height=0.4cm,
     ];
    \pic[tqft,
     incoming boundary components=1,
     outgoing boundary components=1,
     fill=red!70,
     name=b,
     draw,
     cobordism height=0.2cm, anchor=incoming boundary, at=(c-outgoing boundary),boundary separation=2cm
         ];
         
    \pic[tqft/cup,name=a, anchor=incoming boundary, at=(b-outgoing boundary), draw,boundary separation=2cm];
  
    \draw node at (-0.65,0){\tiny $+\infty$};
    \draw node at (-0.5,-0.3){\tiny $1$};
    \draw node at (-0.5,-0.6){\tiny $0$};
    \draw node at (0,0.4){ $\uparrow$};
      \draw node at (0.5,-0.2){ \tiny $L$};
    \draw node at (0,-0.07){\scalebox{.5}{$H^{L,L}$}};
    \draw node at (0,-0.33){\scalebox{.3}{$\overline{H}^{L,L}$}};
        \end{tikzpicture}}

         \caption{\textbf{Case 4}: A schematic representation of a representative of the unit in $CF^S(L,L)$ when endowed with some choice of $\varepsilon$-perturbation data.}
         \label{fig:my_label}
     \end{figure}

     \newpage

\section{Continuation functors}\label{contfunc}

It is well known that between two Floer complexes of the same objects (ambient symplectic manifold or couple of Lagrangians) defined using different Floer data one can construct chain maps, called continuation maps, defined by counting strips which homotope between the different data. Moreover these chain maps are quasi-isomorphisms. In particular, this shows that Floer homology is well-defined in the sense that it is independent of the auxiliary Floer data. In \cite{Sylvan2019OnCategories}, Sylvan extended continuation maps to $A_\infty$-functors on (partially wrapped) Fukaya categories. In particular, he geometrically showed that Fukaya categories do not depend on choices up to quasi-equivalence of $A_\infty$-categories. This approach differs from the well-known proof of this fact contained in \cite[Chapter 10]{Seidel2008FukayaTheory}, which heavily relies on algebraic machineries. In this section we will construct continuation functors for monotone Fukaya categories defined using the Morse-Bott model developed in Section \ref{mbfuk} following the main ideas contained in Sylvain's work. In Section \ref{contfunandfiltr} we discuss how the filtered structure of filtered Fukaya categories defined using $\varepsilon$-perturbation data (as defined in Section \ref{actualconstruction}) behaves under continuation functors.
The main results of this section are summarized in the following theorem, which is an expanded version of Theorem \ref{thmB}.

\begin{thm}
    Consider two regular perturbation data $p,q\in E_\textnormal{reg}$, and assume that they share the same Morse perturbation data. Then there is a non-empty space $E^{p, q}$ of so-called interpolation data and a residual subset $E^{p, q}_\textnormal{reg}\subset E^{p, q}$ such that for any $h\in E_\textnormal{reg}^{p,q}$ there is a weakly-filtered $A_\infty$-functor $$\mathcal{F}^{p,q}_h\colon \mathcal{F}uk(M;p)\to \mathcal{F}uk(M;q)$$ called the continuation functor from $\mathcal{F}uk(M;p)$ to $\mathcal{F}uk(M;q)$ associated to $h$, which is a quasi-equivalence canonical up to quasi-isomorphism of $A_\infty$-functors. Moreover, if $p\in E^{\varepsilon_1,\delta_1}_\textnormal{reg}$ and $q\in E^{\varepsilon_2}_\textnormal{reg}$ for some $\varepsilon_1,\varepsilon_2>0$ are $\varepsilon$-perturbation data sharing the same Morse part, then there is a non-empty subset $E^{p,q; f}_\textnormal{reg}\subset E^{p, q}_\textnormal{reg}$ such that for any $h\in E^{p,q, f}_\textnormal{reg}$ the associated functor $\mathcal{F}^{p,q}_h$ shifts filtration by $\leq \varepsilon_2-\delta_1\varepsilon_2$, in the sense of the definition appearing at the end of Section \ref{ainfpre}. In particular if $\varepsilon_2\leq \delta_1\varepsilon_1$, then $\mathcal{F}^{p,q}_h$ is filtered for any choice of $h\in E^{p, q, f}_\textnormal{reg}$.
\end{thm}
\noindent Note that we define continuation functors only between Fukaya categories defined using the same Morse part of the perturbation data. This assumption may be easily dropped by defining Morse interpolation data, whose construction we omit to avoid notational complexity. In any case, Morse interpolation data play no role from the filtration point of view.


\subsection{Colored trees}\label{colored trees}

We introduce the notion of colored leafed tree following \cite{Mau2010GeometricMultiplihedra}. Let $T$ be a $d$-leafed tree with no vertices of valency equal to $1$ and finitely many interior edges. Given a metric $\lambda\in \lambda(T)$ and a vertex $v\in V(T)$ we define the distance $\lambda(v)$ of $v$ from the root vertex $v_T$ as the sum of the length $\lambda(e)$ of the edges $e\in E^\textnormal{int}(T)$ in the geodesic from $v$ to $v_T$.\\
A coloring on $T$ is a choice of a subset $V^\col(T)\subset V(T)$ of vertices, which we call colored vertices, and of a metric $\lambda\in \lambda(T)$ such that: \begin{enumerate}
    \item for any $i\in \{1,\ldots, d\}$, if we flow along the geodesic from $e_i(T)$ to the root $e_0(T)$ we meet exactly one colored vertex;
    \item each vertex of valency equal to $2$ is colored;
    \item each colored vertex lies at the same distance from the root vertex $v_T$ of $T$.
\end{enumerate}
Two coloring are said to be equivalent if they carry the same set of colored vertices $V^\col(T)$ (and hence the metric is there only to make point (3) well-defined). A $d$-colored tree is a $d$-leafed tree $T$ as above together with an equivalence class of colorings, which we denote by only writing the subset of colored vertices. We denote by $\mathcal{T}^{d+1}_\col$ the space of $d$-colored trees. Note that the notation might be a bit confusing, as $\mathcal{T}^{d+1}_\col$ contains non-stable trees, whereas $\mathcal{T}^{d+1}$ contains only stable trees by definition (see Definition \ref{hugedegfortrees}(4)). Notice that there is a natural partition $$V(T)=V^l(T)\sqcup V^\col(T)\sqcup V^r(T)$$ on colored trees by setting $V^l(T)$ to be the set of vertices of $T$ lying at distance strictly smaller than vertices in $V^\col(T)$ from $v_T$ (this notion does not depend on the choice of metric) and $V^r(T)$ to be the set of vertices of $T$ lying at distance stritly larger than vertices in $V^\col(T)$ from $v_T$.\\ Let $V^\col(T)\subset V(T)$, then we denote by $\lambda(T,V^\col(T))$ the space of metrics $\lambda\in \lambda(T)$ such that $(V^\col(T), \lambda)$ is a coloring for $T$. The following lemma is proved in \cite{Mau2010GeometricMultiplihedra}.
\begin{lem}\label{polcone}
    For any choice of $V^\col(T)\subset V(T)$ the space $\lambda(T,V^\col(T))$ is a polyhedral set of dimension $|E^\textnormal{int}(T)|+1-|V^\col(T)|$.
\end{lem}

Let $d\geq 2$ and $\vec{L}=(L_0,\ldots,L_d)$ be a tuple of Lagrangians in $\lagmd(M,\omega)$. We define by $\mathcal{T}^{d+1}_\col(\vec{L})$ the set of colored trees labelled by $\vec{L}$ and by $\mathcal{T}^{d+1}_{U,\col}(\vec{L})$ the set of colored trees labelled by $\vec{L}$ which can be represented by an unilabelled tree.
\begin{rem}
    As said above, our trees will be defined only between Fukaya categories constructed via perturbation data which share the same Morse part. As it will be apparent from the discussion below, this simplifies the definition of system of ends for colored labelled trees. The definition of universal system of ends for $\vec{L}$ (see Section \ref{seconsystemofends} for the definition of system of ends), as well as that of consistency, readily translate to the case of colored labelled trees, as condition (3) in the definition implies that there are finitely many of those. We skip the details.
\end{rem}


\subsection{Stacked disks}\label{stackeddisks}

To define the source spaces for continuation functors, we will adopt the same strategy as in Section \ref{sourcespaces}: consider primitive source spaces consisting of some configuration of disks, compactify the moduli spaces of those and add a collar neighbourhood. The primitive source spaces in this case will be `stacked' disks, whose moduli spaces will simply be a stack of the same moduli spaces of marked disks indexed by a positive real parameter.\\
We start from the case $d=1$. Let $(L_0,L_1)$ be a couple of Lagrangians in $\lagmd(M,\omega)$. We define $\mathcal{R}^{2,2}(L_0,L_1) = \{pt\}$ to be a singleton and $\mathcal{S}^{2,2}(L_0,L_1) = \{S^{2,2}\}$, where $S^{2,2}$ is a strip if $L_0\neq L_1$ and a disk with marked points in $-1$ and $+1$ if $L_0=L_1$. In both cases we assume that $S^{2,2}$ is endowed with two numbers $\varepsilon_-<0$ and $\varepsilon_+>0$. We think of $\varepsilon_-$ and $\varepsilon_+$ as two strip-like ends coming from conformal embeddings, and which may be positive or negative depending on the context.\\
Let now $d\geq 2$ and pick a tuple $\vec{L}=(L_0,\ldots,L_d)$ of Lagrangians in $\lagmd(M,\omega)$. We define $$\mathcal{R}^{d+1,2}(\vec{L}):=\mathcal{R}^{d+1}(\vec{L})\times (0,\infty)$$ Over $\mathcal{R}^{d+1,2}(\vec{L})$ we have a fiber bundle $$\pi^{d+1,2}(\vec{L})\colon \mathcal{S}^{d+1,2}(\vec{L})\to \mathcal{R}^{d+1,2}(\vec{L})$$ where the fiber $S_{r,w}:=\left(\pi^{d+1,2}(\vec{L})\right)^{-1}(r,w)$ over $(r,w)\in \mathcal{R}^{d+1,2}(\vec{L})$ equals $(S_r,w)$, where $S_r\in \mathcal{S}^{d+1}(\vec{L})$ was defined in Section \ref{sourcespaces}.\\
We fix a consistent choice of strip-like ends for the universal family $(\pi^{d+1}(\vec{L}))_{d,\vec{L}}$ for the rest of this section. A universal choice of strip-like ends on the universal family $(\pi^{d+1,2}(\vec{L}))_{d,\vec{L}}$ is said to be compatible with the above fixed universal choice of strip-like ends on $(\pi^{d+1}(\vec{L}))_{d,\vec{L}}$ if for any $d$, any $\vec{L}$ and any $(r,w)\in \mathcal{R}^{d+1,2}(\vec{L})$ the induced choice of strip-like ends on $(S_r,w)$ agrees up to a shift with that on $S_r$.\\

Fix a tuple $\vec{L}=(L_0,\ldots, L_d)$ of Lagrangians in $\lagmd(M,\omega)$. For any colored labelled tree $T\in \mathcal{T}^{d+1}_\col(\vec{L})$ we write $$\mathcal{R}^{T,2}:=\prod_{v\notin V^\col(T)}\mathcal{R}^{|v|}(\vec{L_v})\times \prod_{v\in V^\col(T)}\mathcal{R}^{|v|,2}(\vec{L_v}) $$ That is, to each uncolored vertex we associate a family of marked disks, while to each colored vertex we associate a family of stacked marked disks. We then define $$\overline{\mathcal{R}^{d+1,2}(\vec{L})}:=\bigcup_{T\in T^{d+1}_\col(\vec{L})}\mathcal{R}^{T,2}$$
We have the following standard-looking result.
\begin{lem}
    $\overline{\mathcal{R}^{d+1,2}(\vec{L})}$ admits the structure of a generalized manifold with corners of dimension $d-1$ which realizes Stasheff's multiplihedron.
\end{lem}
\noindent As the definition of $\overline{\mathcal{R}^{d+1,2}(\vec{L})}$ involves two kind of moduli spaces the construction of boundary charts is more delicate than in the standard case of $\overline{\mathcal{R}^{d+1}(\vec{L})}$.\\
We split the construction of the boundary charts for $\overline{\mathcal{R}^{d+1,2}(\vec{L})}$ in three cases. A colored labelled tree is called:
\begin{enumerate}
    \item a lower tree if the root has valency equal to $2$, and is hence in particular the only colored vertex;
    \item a middle tree if the root is colored and has valency greater than $2$;
    \item an upper tree if the root is not colored.
\end{enumerate}
We will see that constructing boundary charts near lower and middle trees is easy, while upper trees are more delicate to handle as there are multiple simultaneous splittings.\\

\textbf{(1) Lower trees.} First, we consider the colored labelled tree $T^2\in \mathcal{T}^{d+1}(\vec{L})$ with only two vertices and two-valent root. Notice that $\mathcal{R}^{T^2,2}=\mathcal{R}^{2,2}(L_0,L_d)\times \mathcal{R}^{d+1}(\vec{L})$ and recall that we defined $\mathcal{R}^{2,2}(L_0,L_d)$ to be a singleton. We define a map $$\gamma^{T^2,2}\colon(-1,0)\times \mathcal{R}^{T^2,2}\to \mathcal{R}^{d+1,2}(\vec{L})$$ as follows: given $r\in \mathcal{R}^{d+1}(\vec{L})$ and $\rho\in (-1,0)$ we set $$\gamma^{T^2,2}(\rho, pt,r,):=(\tilde{r},-\rho)$$ where $\tilde{r}\in \mathcal{R}^{d+1}(\vec{L})$ represents the surface obtained by gluing $pt\in \mathcal{R}^{2,2}(L_0,L_d)$ to the $0$th marked point of $r$ with length $\rho$. Notice that $r=\tilde{r}$ as elements of $\mathcal{R}^{d+1}(\vec{L})$, but they may come with different choices of strip-like ends. \\
Let now $T\in \mathcal{T}^{d+1}_\col(\vec{L})$ be an arbitrary colored labelled tree with $2$-valent root. Notice that $\mathcal{R}^{T,2}= \mathcal{R}^{2,2}(L_0,L_d)\times \mathcal{R}^{T'}$, where $T'$ is obtained by collapsing the positive pointing edge attached to the root of $T$ and can hence be viewed as an alement of $\mathcal{T}^{d+1}(\vec{L})$, as all the non root vertices of $T$ are stable by definition of colored tree. We define a map $$\gamma^{T,2}\colon (-1,0)\times (-1,0)^{|E^\textnormal{int}(T)|-1}\times \mathcal{R}^{T,2}\to \mathcal{R}^{d+1,2}(\vec{L})$$ as follows: given $(r_v)_{v\in V(T')}\in \mathcal{R}^{T'}$, $\rho\in (-1,0)$ and $\rho_e\in (-1,0)$ for any $e\in E^\textnormal{int}(T')$ we set $$\gamma^{T,2}\left(\rho, (\rho_e)_e, (pt, (r_v)_v)\right):=\gamma^{T^2,2}\left(\rho, pt, \gamma^{T'}((\rho)_e, (r_v))\right)$$ where $$\gamma^{T'}\colon (-1,0)^{|E^\textnormal{int}(T')|}\times \mathcal{R}^{T'}\to \mathcal{R}^{d+1}(\vec{L})$$ is the gluing map defined in \cite{Seidel2008FukayaTheory}. Notice that by considering trivial gluing $\gamma^{T,2}$ extends to a map $$\overline{\gamma^{T,2}}\colon (-1,0]\times (-1,0]^{|E^\textnormal{int}(T)|-1}\times \mathcal{R}^{T,2}\to \overline{\mathcal{R}^{d+1,2}(\vec{L})}.$$
\begin{rem}
    All these charts will correspond to genuine boundary charts in the sense of manifolds with corners. The only codimension one face arising in this subcase is parametrized by $T_2$, and in general, codimension $k$ faces arising here are parametrized by bicolored trees with two-valent root and $k$ uncolored vertices.
\end{rem}

\textbf{(2) Middle trees.} Let $T\in \mathcal{T}^{d+1}_\col(\vec{L})$ be a colored labelled tree with colored stable root. Notice that $$\mathcal{R}^{T,2}=\prod_{v\neq v_T}\mathcal{R}^{|v|}\times \mathcal{R}^{|v_T|,2}.$$ Since the valency of $v_T$ is at least $3$ by assumption and $v_T$ is colored (hence there are no two-valent vertices in $T$), we may see $T$ as an element of $\mathcal{T}^{d+1}(\vec{L})$ by forgetting the coloring. In particular, $\mathcal{R}^{T}$ is well-defined. We define a map $$\gamma^{T,2}\colon (-1,0)^{|E^\textnormal{int}(T)|}\times \mathcal{R}^{T,2}\to \mathcal{R}^{d+1,2}(\vec{L})$$ as follows: given $\rho_e\in (-1,0)$ for any $e\in E^\textnormal{int}(T)$, $r_v\in \mathcal{R}^{|v|}$ for any $v\neq v_T$ and $(r_{v_T},w)\in \mathcal{R}^{|v_T|,2}$ we set $$\gamma^{T,2}\left((\rho_e)_e, (r_{v_T},w), (r_v)_v\right):= \left(\gamma^{T}\left((\rho_e)_e, (r_v)_{v\in V(T)}\right), w\right).$$ By considering trivial gluing $\gamma^{T,2}$ extends to a map $$\overline{\gamma^{T,2}}\colon (-1,0]^{|E^\textnormal{int}(T)|}\times \mathcal{R}^{T,2}\to \overline{\mathcal{R}^{d+1,2}(\vec{L})}.$$
\begin{rem}
All these charts correspond to genuine boundary charts in the sense of manifolds with corners. The codimension one faces arising in this subcase are parametrized by colored trees with colored root and one uncolored vertex, and in general, codimension $k$ faces arising here are parametrized by bicolored trees with colored root and $k$ uncolored vertices.
\end{rem}

\textbf{(3) Upper trees.} As soon as the root is not colored, there may be complications in the choice of gluing lenghts, basically due to the fact that the coloring distance is unique, so that \textit{colored vertices should remember what came before them} (where the concept of before is determined by the aforedefined orientation on trees).
\begin{example}\label{weneedintrinsic}
Consider the colored $3$-tree labeled by the cyclically different (see page \pageref{cicldiff}) tuple $\vec{L}=(L_0,L_1,L_2,L_3)$ of Lagrangians in $\lagmd(M,\omega)$ and the associated stacked disk configuration in $\overline{\mathcal{R}^{4,2}(\vec{L})}$ depicted in Figure \ref{stackedtree1}.

     \begin{center}
     \begin{figure}[H]
\includegraphics[width=12cm, height=10cm]{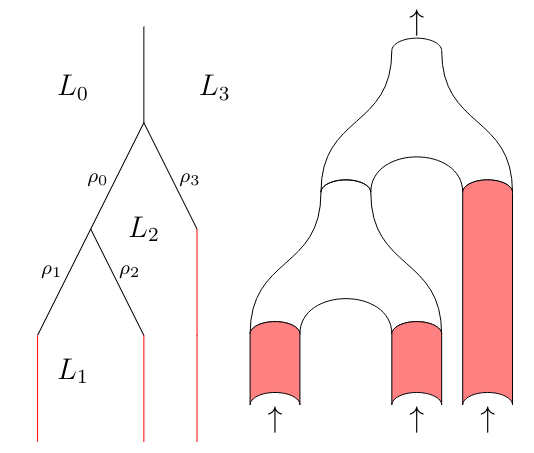}
     \caption{The colored tree we are discussing and the associated disks configuration in $\overline{\mathcal{R}^{4,2}(\vec{L})}$ for a cyclically different tuple $\vec{L}=(L_0,L_1,L_2, L_3)$. Black configurations correspond to elements of $\mathcal{R}^{k,1}$, while red configurations correspond to elements of $\mathcal{R}^{l,2}$ for some $k,l$.}
         \label{stackedtree1}
     \end{figure} \end{center}
The definition of colored tree forces $\rho_1=\rho_2$ and $\rho_0+\rho_1=\rho_3$. Gluing the two non-colored disks following the only interior edge which does not touch colored vertices we get the tree/configuration depicted in Figure \ref{stackedtree2}.
   \begin{center}
     \begin{figure}[H]
         \centering
\includegraphics[width=12cm, height=7.5cm]{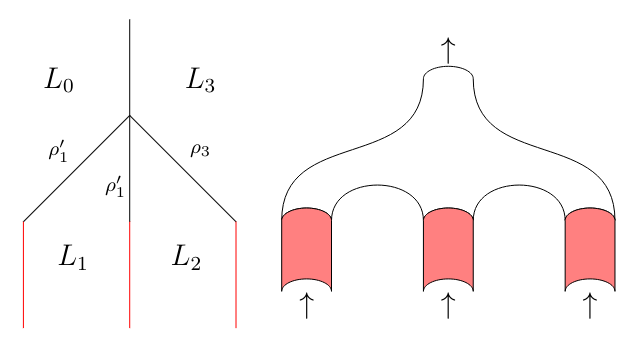}
\caption{}
 \label{stackedtree2}
     \end{figure}
    \end{center}
    
\noindent In this case, the definition of colored tree forces $\rho_1'=\rho_3=\rho_0+\rho_1$. On the other hand, if, starting again with the configuration from Figure \ref{stackedtree1}, we glue along the interior edges touching a colored vertex, we get the tree/configuration depicted in Figure \ref{stackeddisc3}.
 \begin{center}
     \begin{figure}[H]
         \centering
 \includegraphics[width=12cm, height=8cm]{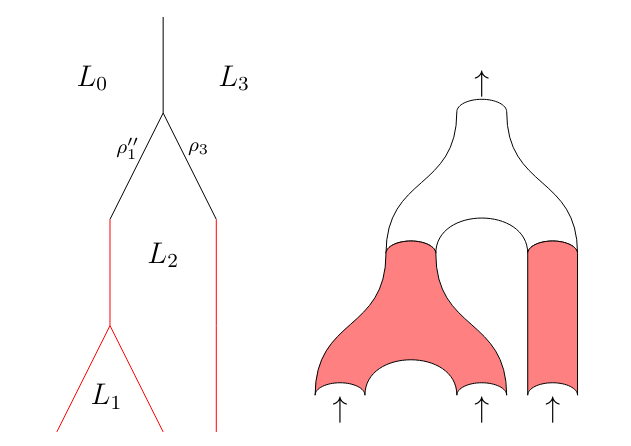}
         \caption{The colored tree we are discussing and the associated disks configuration in $\overline{\mathcal{R}^{3+1,2}}(\vec{L}$ for a cyclically different (see Page \pageref{cicldiff}) tuple $\vec{L}=(L_0,\ldots, L_3)$. Black configurations correspond to elements of $\mathcal{R}^{k,1}$, while red configurations correspond to elements of $\mathcal{R}^{l,2}$ for some $k,l$.}
         \label{stackeddisc3}
     \end{figure}
    \end{center}
    

\noindent Again the definition of colored tree forces $\rho_1''=\rho_3=\rho_0+\rho_1$. Performing one more gluing in both cases with lenght $\rho_1'=\rho_1''$ in order to get to an interior configuration in $\mathcal{R}^{3+1,2}$, we will in general get two different stacked disks, as in the first case it is $\rho_0$ that contributed to both gluings, while in the second one, it is $\rho_1$.
\end{example}

It follows from Example \ref{weneedintrinsic} that we have to refine the gluing process a bit, in order for non-root colored vertices to recall the gluing lenghts that were associated to preceedings (uncolored) vertices. The following definition comes from \cite{Sylvan2019OnCategories}.

\begin{defn}[\cite{Sylvan2019OnCategories}]
An intrinsic width function for the universal family $(\pi^{d+1}(\vec{L}))_{d,\vec{L}}$ consists of a family $\{w_i^{d,\vec{L}}\}$, $d\geq 2$ and $i\in \{1,...,d\}$, of functions $w_i^{d,\vec{L}}\colon\mathcal{R}^{d+1}(\vec{L})\rightarrow [0,\infty)$ such that:
\begin{enumerate}
    \item $w_1^2$ and $w_2^2$ are the zero function;
    \item let $S^{k+1}\in \mathcal{S}^{k+1}(\vec{L_1})$ and $S^{l+1}\in \mathcal{S}^{l+1}(\vec{L_2})$; if $S^{d+1}\in \mathcal{S}^{d+1}(\vec{L_1}\#_n\vec{L_2})$ is the disk\footnote{See page \pageref{cancelletto} for the definition of the operation $\#$ on tuples of Lagrangians.} (diffeomorphic to the surface) obtained by gluing the root of $S^{l+1}$ to the $n$th puncture of $S^{k+1}$ with lenght $\rho$ (assuming this gluing is admissible), then $$w_i^{d,\vec{L_1}\#_n\vec{L_2}}(S^{d+1})=\begin{cases} w_i^{k,\vec{L_1}}(S^{k+1}), \ \textnormal{ if } i<n\\ w_{i-n+1}^{l,\vec{L_2}}(S^{l+1})+\rho, \ \textnormal{ if } n\leq i<n+l\\ w_{i-l+1}^{k,\vec{L_1}}(S^{k+1}), \ \textnormal{ if }n+l\leq i\end{cases}$$
\end{enumerate}
\end{defn}
\begin{lem}
There is a unique choice of intrinsic width function.
\end{lem}
\begin{proof}
Build it by induction on $d$; the fact that the definition requires a fixed choice of function for $d=2$ implies uniqueness of the construction. See \cite{Sylvan2019OnCategories}.
\end{proof}

Back to the construction of boundary charts for $\overline{\mathcal{R}^{d+1,2}(\vec{L})}$ near upper trees. Let $k\geq 3$ and consider first a colored labelled tree $T^k\in \mathcal{T}^{d+1}_\col(\vec{L})$ with only one non-colored vertex, the root, of valency $k$. In particular, $T^k$ has $k$ colored vertices. Notice that in this case we have $$\mathcal{R}^{T^k,2}= \mathcal{R}^k\times \prod_{v\neq v_T}\mathcal{R}^{|v|,2}.$$ We define a map $$\gamma^{T^k,2}\colon(-1,0)\times \mathcal{R}^{T^k,2}\to \mathcal{R}^{d+1,2}(\vec{L})$$ as follows: given $\rho\in (-1,0)$, $r\in \mathcal{R}^k$ and $(r_v, w_v)\in \mathcal{R}^{|v|,2}$ for any $v\neq v_T$ we set $$\gamma^{T^k,2}\left(\rho, r, (r_v,w_v)_v\right):=(\tilde{r},e^\frac{-1}{\rho})$$ where $\tilde{r}$ corresponds to the configuration obtained by gluing the $i$th marked point of $r$ to the negative marked point of $S_{r_{v_i}}$, where $v_i\in V(T^k)$ is the $i$th colored vertex of $T^k$ in counterclockwise order starting from the root, with gluing length $$l_i:=e^{-{\frac{1}{\rho}}}-w_{v_i}-w_i^k(S_r)$$ for any $i\in \{1,\ldots,k\}$.\\
Let now $T\in \mathcal{T}^{d+1}_\col(\vec{L})$ be a colored labelled tree with uncolored root. Consider the decomposition $V(T)=V^l(T)\sqcup V^\col(T)\sqcup V^r(T)$ of the set of vertices of $T$ introduced in Section \ref{colored trees}: we define $T^l$ to be the (uncolored) tree generated by $V^l(T)$ and the interior edges attached to those vertices (some of which become leaves), and by $T_1^r,\ldots, T_n^r$ the connected components of the colored tree generated by $V^r(T)\sqcup V^\col(T)$ (this involves a slight change in the induced metrics, but it is a nuance). Note that each $T_i^r$ is either a lower or a middle tree. Since by construction the tree $T_i^r$ has a colored root for any $i\in \{1,\ldots, n\}$, it is either a lower or a middle tree. In particular any tree $T_i^r\setminus V^\col(T)$ is a union of uncolored trees $T_{i,j}^r$, $j\in \{0,\ldots, k_i\}$ for some $k_i$. Notice that we have $$\mathcal{R}^{T,2}= \mathcal{R}^{T^l}\times \prod_{i=0}^n\mathcal{R}^{T_i^r,2}= \mathcal{R}^{T^l}\times \prod_{v\in V^\col(T)}\mathcal{R}^{|v|,2}\times \prod_{i,j}\mathcal{R}^{T_{i,j}^r}$$

 \begin{rem}At this point, we would like to define a map $$\gamma^{T,2}\colon (-1,0)^{|E^\textnormal{int}(T^l)|}\times (-1,0)\times (-1,0)^{\sum_i|E^\textnormal{int}(T^r_i)|}\times \mathcal{R}^{T,2}\longrightarrow \mathcal{R}^{d+1,2}(\vec{L})$$ however, in this case, there may be colored trees parametrizing generalized corners, e.g. the colored $4$-tree associated to the configuration depicted in Figure \ref{singularity} (cfr. \cite{Mau2010GeometricMultiplihedra}).
 \begin{center}
     \begin{figure}[H]
         \centering
 \scalebox{1.5}{
\begin{tikzpicture}

    \pic[tqft/pair of pants,name=a, draw,];

         \pic[tqft,
     incoming boundary components=1,
     outgoing boundary components=2,
     name=b1,
     draw,
     cobordism height=2cm, anchor=incoming boundary, at=(a-outgoing boundary 1), offset=-1];

              \pic[tqft,
     incoming boundary components=1,
     outgoing boundary components=2,
     name=b2,
     draw,
     cobordism height=2cm, anchor=incoming boundary, at=(a-outgoing boundary 2)];

                   \pic[tqft,
     incoming boundary components=1,
     outgoing boundary components=1,
     name=c1,
     fill=red!50,
     draw,
     cobordism height=1cm, anchor=incoming boundary, at=(b1-outgoing boundary 1)];
                        \pic[tqft,
     incoming boundary components=1,
     outgoing boundary components=1,
     name=c2,
     fill=red!50,
     draw,
     cobordism height=1cm, anchor=incoming boundary, at=(b1-outgoing boundary 2)];
                        \pic[tqft,
     incoming boundary components=1,
     outgoing boundary components=1,
     name=c3,
     fill=red!50,
     draw,
     cobordism height=1cm, anchor=incoming boundary, at=(b2-outgoing boundary 1)];
                        \pic[tqft,
     incoming boundary components=1,
     outgoing boundary components=1,
     name=c4,
     fill=red!50,
     draw,
     cobordism height=1cm, anchor=incoming boundary, at=(b2-outgoing boundary 2)];

\end{tikzpicture}}
         \caption{The singularity in the space $\overline{\mathcal{R}^{5,2}}$.}
         \label{singularity}
     \end{figure}
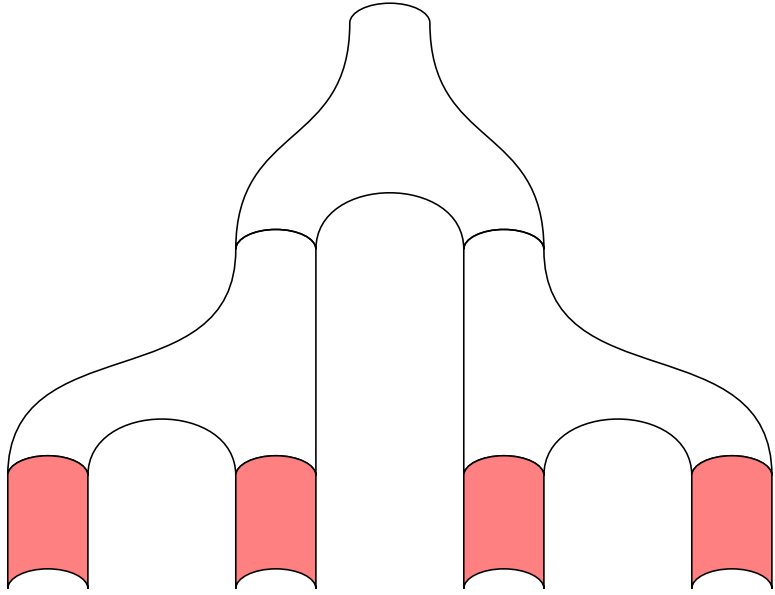
   
    \end{center}
The codimension one faces arising in this subcase are parametrized by colored trees with the root as the only uncolored vertex, and, in general, codimension $k$ faces are parametrized by colored trees with uncolored root and $k$ total uncolored vertices.
\end{rem}

Let $T$ be a $d$-leafed tree with no vertices of valency $1$ and finitely many interior edges, and let $V^\col(T)\subset V(T)$ be a subset of vertices. Recall that by Lemma \ref{polcone} the space $\lambda(T,V^\col(T))$ is a polyhedral cone. For each $\lambda\in \lambda(T,V^\col(T))$ define the map $\rho^\lambda\colon E^\textnormal{int}(T)\rightarrow (-1,0)^{|E^\textnormal{int}(T)|}$ by $$\rho^\lambda:= -e^{-\lambda}$$ and let $\rho(T,V^\col(T))$ to be the space of all such maps; it is easy to see that $\rho(T, V^\col(T))$ is a polyhedral cone in $(-1,0)^{|E^\textnormal{int}(T^l)|}\times (-1,0)\times (-1,0)^{\sum_i|E^\textnormal{int}(T^r_i)|}$, as by the decomposition above we get $$|E^\textnormal{int}(T)|-|V^\textnormal{col}(T)| = |E^\textnormal{int}(T^l)|+\sum_i|E^\textnormal{int}(T^r_i)|.$$
We remark that the manifold $\mathcal{R}^{T,2}$ does not depend on the choice of the metric $\lambda\in \lambda(T,V^\col(T))$ but only on the choice of the subset $V^\col(T)$ by definition of colored tree. We hence define a map $$\gamma^{T,2}\colon\rho(T,V^\col(T))\times \mathcal{R}^{T,2}\to \mathcal{R}^{d+1,2}(\vec{L})$$ as follows: let $$\begin{cases}
    (r^l_v)_{v\in V^l(T)}\in \mathcal{R}^{T^l},\\
    \left((r^r_{v_i},w_i),(r^r_{i,v})_{v\neq v_i\in V(T^r_i)}\right)\in \mathcal{R}^{T^r_i,2}= \mathcal{R}^{|v_i|,2}(\vec{L_{v_i}})\times \prod_{v\neq v_i\in V(T_i^r)}\mathcal{R}^{|v|}(\vec{L_v}) \textnormal{  for any }i\in \{1,\ldots, n\},
\end{cases}$$
and write $$\vec{r}:=\left( (r^l_v)_{v\in V^l(T)},  \left((r^r_{v_1},w_1),(r^r_{1,v})_{v\neq v_1}\right),\ldots, \left((r^r_{v_n},w_n),(r^r_{n,v})_{v\neq v_n}\right)\right)$$
and consider $\rho\in \rho(T,V^\col(T))$, which we can write as $$\rho = \left(\vec{\rho^l}, \tilde{\rho}, \vec{\rho_{1}^r},\ldots, \vec{\rho_n^r}\right)\in  (-1,0)^{|E^\textnormal{int}(T^l)|}\times (-1,0)\times \prod_{i=1}^n(-1,0)^{|E^\textnormal{int}(T^r_{i})|}$$
then we set $$\gamma^{T,2}(\rho, \vec{r}):=\gamma^{T^n,2}\left(\rho, \gamma^{T^l}(\vec{\rho^r}, (r^l_v)_v), \gamma^{T_1^r,2}(\vec{\rho^r_1}, (r^r_{v_i},w_i),(r^r_{1,v})_{v\neq v_1}), \ldots, \gamma^{T_n^r,2}(\vec{\rho^r_n}, (r^r_{v_n},w_i),(r^r_{n,v})_{v\neq v_n})\right)$$

In short, first we glue the all the uncolored disks, then we glue the resulting root with the colored disks following the receipt outlined just here above, and then we glue what's remaining, which is represented by a middle tree.
Again, exactly as before, we extend $\gamma^{T,2}$ to a map $$\overline{\gamma^{T,2}}\colon \mathcal{R}^{T,2}\times \overline{\rho(T)}\longrightarrow \overline{\mathcal{R}^{d+1,2}}$$ by considering trivial gluing.\\

We conclude this section with the following Lemma (cfr. \cite{Mau2010GeometricMultiplihedra, Sylvan2019OnCategories}).
\begin{lem}
    The space $\overline{\mathcal{R}^{d+1,2}(\vec{L})}$ admits the structure of generalized manifold with corners by declaring the maps $\gamma^{T,2}$, $T\in \mathcal{T}^{d+1}_\col(\vec{L})$, to be boundary charts for small enough choice of parameters. Moreover, $\overline{\mathcal{R}^{d+1,2}(\vec{L})}$ realizes Stasheff's $(d-1)$-multiplihedron.
\end{lem}

\subsection{Source spaces: moduli spaces of stacked clusters}
In this section we define moduli spaces of stacked clusters of disks starting from (the compactification of) moduli spaces of stacked disks, similarly to how we constructed moduli spaces of clusters of disks from (the compactification of) moduli spaces of disks in Section \ref{sourcespaces}. We will skip some details where the construction are the same.\\

Let $d\geq 1$ and pick a tuple $\vec{L}=(L_0,\ldots, L_d)$ of Lagrangians in $\lagmd(M,\omega)$. We define the moduli space of clusters of stacked disks (or stacked clusters of disks) with marked points as $$\mathcal{R}^{d+1,2}_C(\vec{L}):=\bigcup_{T\in \mathcal{T}^{d+1}_{U,\col}(\vec{L})} \mathcal{R}^{T,2}\times \lambda(T)$$ and also set $$\overline{\mathcal{R}^{d+1,2}_C(\vec{L})}:= \bigcup_{T\in \mathcal{T}^{d+1}_{\col}(\vec{L})} \mathcal{R}^{T,2}\times \overline{\lambda_U(T)}$$ We will write elements of $\mathcal{R}^{d+1,2}_C(\vec{L})$ as $(r,w,T,\lambda)$.\\
The following Lemma is the analogous of Lemma \ref{lemcluster} for stacked clusters. Its proofs combines the construction from Section \ref{stackeddisks} with gluing of line segments as in Section \ref{sourcespaces}.
\begin{lem}
    The space $\mathcal{R}^{d+1,2}_C(\vec{L})$ admits the structure of smooth manifold of dimension $d-1$. Moreover, the space $\overline{\mathcal{R}^{d+1,2}(\vec{L})}$ admits the structure of generalized manifold with corners and realizes Stasheff's $(d-1)$-multiplihedron.
\end{lem}
\noindent As in the case of Lemma \ref{lemcluster}, the fact that $\overline{\mathcal{R}^{d+1,2}(\vec{L})}$ realizes the multiplihedron comes for free from its construction.\\
We define bundles $$\pi^{d+1,2}_C(\vec{L})\colon \mathcal{S}^{d+1,2}_C(\vec{L})\to \mathcal{R}^{d+1,2}_C(\vec{L})$$ where fibers are defined as in Section \ref{mbfuk} by replacing nodal points by line segments of length controlled by the metric part of $\mathcal{R}^{d+1,2}_C(\vec{L})$.\\
A universal choice of strip-like ends for the universal family $(\pi^{d+1,2}_C(\vec{L}))_{d,\vec{L}}$ is the pullback of a universal choice of strip-like ends for the universal family $(\pi^{d+1,2}(\vec{L}))_{d,\vec{L}}$. A system of ends for (the trees involved in the definition of) $(\pi^{d+1,2}_C(\vec{L}))_{d,\vec{L}}$ is just a choice of system of ends for (the trees involved in the definition of) $(\pi^{d+1}_C(\vec{L}))_{d,\vec{L}}$ by forgetting about two valent vertices. This last definition relies on the fact that we will not construct continuation functors between perturbation data with different Morse parts (in particular, continuation maps between Floer complexes for a couple of identical Lagrangians will be the identity, so that we do not need system of ends for that case).


\subsection{Interpolation data}\label{intdata}

We define the concept of interpolation datum, which is nothing else than the analogous of a perturbation datum but for stacked disks.\\
Fix a universal and consistent choice of strip-like ends for the universal family of cluster of disks $(\pi^{d+1}_C(\vec{L}))_{d,\vec{L}}$ and a universal, consistent and compatible choice of strip-like ends for the universal family of stacked clusters of disks $(\pi^{d+1,2}_C(\vec{L}))_{d,\vec{L}}$. We fix Floer data $(f^L,g^L,J^L)$ for any couple $(L,L)$ of identical Lagrangians in $\lagmd(M,\omega)$ and associated Morse perturbation data for the universal family $(\pi^{d+1}_C(\vec{L}))_{d,\vec{L}}$ (in practice, for the family $(\mathcal{T}^{d+1}_U(\vec{L}))_{d,\vec{L}}$), which we denote by $(\boldsymbol{f}, \boldsymbol{g})=(f^{\vec{L}}, g^{\vec{L}})_{d,\vec{L}}$. We denote by $E^{(\boldsymbol{f},\boldsymbol{g})}_\textnormal{reg}\subset E_\textnormal{reg}$ the subfamily of regular perturbation data whose Morse part agrees with $(\boldsymbol{f},\boldsymbol{g})$. Note that for a generic choice of $(\boldsymbol{f}, \boldsymbol{g})$ this subfamily is non-empty, and we actually assume it is for our choice. We construct continuation functors between Fukaya categories defined via elements of $E^{(\boldsymbol{f},\boldsymbol{g})}_\textnormal{reg}$.\\
\begin{defn}
    Fix two perturbation data $p,q\in E^{(\boldsymbol{f},\boldsymbol{g})}_\textnormal{reg}$. Let $d\geq 1$ and consider a tuple $\vec{L}=(L_0,\ldots, L_d)$ of Lagrangians in $\lagmd(M,\omega)$ and an element $(r,w,T,\lambda)\in \mathcal{R}^{d+1,2}_C(\vec{L})$. An interpolation datum between the perturbation data $\mathcal{D}^{\vec{L}}_p(r,T,\lambda)$ and $\mathcal{D}^{\vec{L}}_q(r,T,\lambda)$ associated to $p$ and $q$ on $S_{r,T,\lambda}$ consists of a family of couples $$\left(K^{r,w,T,\lambda}_v, J^{r,w,T,\lambda}_v\right)$$ indexed by vertices $v\in V(T)$ such that:
\begin{itemize}
    \item if $v\in V^l(T)$ then:
    \begin{itemize}
        \item $K^{r,w,T,\lambda}_v\in \Omega^1(S_{r,T,\lambda}(v),C^\infty(M))$ is an Hamiltonian-valued one-form which vanishes identically if $v\notin T_\textnormal{red}$, while for $v\in T_\textnormal{red}$ is such that for any $i\in\{0,\ldots, |v|\}$ it satisfies $$K^{r,w,T,\lambda}_v|_{L^v_i} = 0 \textnormal{   for any   }\xi\in T(\partial_iS_{r,T,\lambda}(v)),$$ and for any $i\in\{0,\ldots, |v|\}$ on the strip-like end $\epsilon^v_i$ of $S_{r,T,\lambda}(v)$ we have\footnote{Recall that we set $H_p^{L,L}=H_q^{L,L}=0$ and $J^{L,L}_p=J^{L,L}_q=J^L$ for any $L\in \lagmd(M,\omega)$.} $$K^{r,w,T,\lambda}_v=H^{L_{i-1},L_{i}}_pdt \textnormal{   for any   }|s|\geq 1,$$
        \item $J^{r,w,T,\lambda}_v$ is a domain-dependent $\omega$-compatible almost complex structure such that if\footnote{See page \pageref{fundamentaldecomposition} for the definition of the subtrees $T_i^F$.} $v\in T_i^F\subset T\setminus T_\textnormal{red}$ for some $i\in\{0,\ldots, d^F\}$ it is identical to $J^{L_i^F}_p=J^{L_i^F}_q$, while if $v\in T_\textnormal{red}$ it is such that for any $i\in\{0,\ldots, |v|\}$ on the $i$th strip-like end $\epsilon_i^v$ of $S_{r,T,\lambda}(v)$ we have $$J^{r,w,T,\lambda}=J^{L_{i-1},L_{i}}_p;$$
    \end{itemize}
    \item if $v\in V^\col(T)$, then:
    \begin{itemize}
        \item $K^{r,w,T,\lambda}_v\in \Omega^1(S_{r,T,\lambda}(v),C^\infty(M))$ is an Hamiltonian-valued one-form which vanishes identically if $v\notin T_\textnormal{red}$, while for $v\in T_\textnormal{red}$ is such that for any $i\in\{0,\ldots, |v|\}$ it satisfies $$K^{r,w,T,\lambda}_v|_{L^v_i} = 0 \textnormal{   for any   }\xi\in T(\partial_iS_{r,T,\lambda}(v)),$$ and for any $i\in\{1,\ldots, |v|\}$ on the strip-like end $\epsilon^v_i$ of $S_{r,T,\lambda}(v)$ we have $$K^{r,w,T,\lambda}_v=H^{L_{i-1},L_{i}}_pdt \textnormal{   for any   }|s|\geq 1$$ while on the $0$th strip-like end we have $$K^{r,w,T,\lambda}_v=H^{L_0,L_d}_qdt,$$
        \item $J^{r,w,T,\lambda}_v$ is a domain-dependent $\omega$-compatible almost complex structure such that if $v\in T_i^F\subset T\setminus T_\textnormal{red}$ for some $i\in\{0,\ldots, d^F\}$ it is identical to $J^{L_i^F}_p=J^{L_i^F}_q$, while if $v\in T_\textnormal{red}$ it is such that for any $i\in\{1,\ldots, |v|\}$ on the $i$th strip-like end $\epsilon_i^v$ of $S_{r,T,\lambda}(v)$ we have $$J^{r,w,T,\lambda}=J^{L_i, L_{i+1}}_p,$$ while on the $0$th strip-like end we have $$J^{r,w,T,\lambda}_v=J^{L_0,L_d}_q;$$
    \end{itemize}
       \item if $v\in V^r(T)$ then:
    \begin{itemize}
        \item $K^{r,w,T,\lambda}_v\in \Omega^1(S_{r,T,\lambda}(v),C^\infty(M))$ is an Hamiltonian-valued one-form which vanishes identically if $v\notin T_\textnormal{red}$, while for $v\in T_\textnormal{red}$ is such that for any $i\in\{0,\ldots, |v|\}$ it satisfies $$K^{r,w,T,\lambda}_v|_{L^v_i} = 0 \textnormal{   for any   }\xi\in T(\partial_iS_{r,T,\lambda}(v)),$$ and for any $i\in\{0,\ldots, |v|\}$ on the strip-like end $\epsilon^v_i$ of $S_{r,T,\lambda}(v)$ we have $$K^{r,w,T,\lambda}_v=H^{L_{i-1},L_{i}}_qdt \textnormal{   for any   }|s|\geq 1,$$
        \item $J^{r,w,T,\lambda}_v$ is a domain-dependent $\omega$-compatible almost complex structure such that if $v\in T_i^F\subset T\setminus T_\textnormal{red}$ for some $i\in\{0,\ldots, d^F\}$ it is identical to $J^{L_i^F}_p=J^{L_i^F}_q$, while if $v\in T_\textnormal{red}$ it is such that for any $i\in\{0,\ldots, |v|\}$ on the $i$th strip-like end $\epsilon_i^v$ of $S_{r,T,\lambda}(v)$ we have $$J^{r,w,T,\lambda}_v=J^{L_{i-1},L_{i}}_q.$$
    \end{itemize}
\end{itemize}
\end{defn}
\begin{rem}
    The above definition is long but intuitive. Assume $T=T_\textnormal{red}$ for simplicity: interpolation data are just $p$- or $q$-perturbation data on uncolored vertices (depending on how far from the root the vertex lies), while on colored vertices interpolation data interpolate from Floer data with respect to $p$ to a Floer datum with respect to $q$.
\end{rem}
\begin{rem}
    Notice that altough we want to interpolate between the \textit{perturbation data} $p$ and $q$, the definition of interpolation data between them makes no explicit mention of any perturbation datum, but only Floer ones. The connection to perturbation data will appear in the definition of the consistency condition for interpolation data which we introduce below.
\end{rem}
We define an interpolation datum between $p$ and $q$ for $\vec{L}$ as a smooth choice of interpolation data on $\mathcal{R}^{d+1}_C(\vec{L})$. A universal choice of interpolation data from $p$ to $q$ is a choice of interpolation datum between $p$ and $q$ for any tuple $\vec{L}$ of Lagrangians in $\lagmd(M,\omega)$ of any length.\\

Similarly to the case of perturbation data, given a universal choice of interpolation data from $p$ to $q$, we might encounter some consistency problems (see Section \ref{pdgen}). In this case, those problems are a bit more subtle compared to the case of perturbation data, as in the gluing process for stacked disks involves gluing of two kinds of clusters (in particular, disks and stacked disks, see Section \ref{stackeddisks}).\\ Fix a universal choice of interpolation data from $p$ to $q$ on $(\pi^{d+1,2}_C(\vec{L}))_{d,\vec{L}}$. We say that this choice is consistent if for each $d\geq 1$ and for any tuple $\vec{L}=(L_0,\ldots, L_d)$ of Lagrangians in $\lagmd(M,\omega)$ we have:
\begin{enumerate}
    \item for any $T\in \mathcal{T}^{d+1}_\col(\vec{L})$ and any $k\in \Z_{\geq 0}$ there is a subset $$U\subset \mathcal{R}^{T,2}\times \lambda_U^k(T)\times [-1,0)^{|E^\textnormal{int}_F(T)|+k}$$ whose closure is a neighbourhood of the trivial gluing, where the gluing parameter are small such that the interpolation data for stacked clusters over $U$ agree with interpolation data induced by gluing on thin parts;
    \item all interpolation data extend smoothly to $\partial \overline{\mathcal{R}^{d+1,2}_C(\vec{L})}$ and agree there with perturbation data coming from (trivial) gluing.
\end{enumerate}
Assume for a brief moment that $\vec{L}$ is made of different Lagrangians, i.e. $\vec{L}=\vec{L^F}$ in the notation of Section \ref{tuplesoflags}. In more detail, point (1) of the above definition means in this case the following: let $T\in \mathcal{T}^{d+1}_\col(\vec{L})$, $(r,w)$ in the image of $\gamma^{T,2}$ restricted to a subset where the gluing parameters are small, and consider $e\in E^\textnormal{int}(T)$ (of course, $e$ is not unilabelled) and the summand $\varepsilon_{f^+(e)}([0,-\ln(-\rho_e)]\times [0,1])$ of the thin part of $(S_r,w)\in \mathcal{S}^{d+1,2}(\vec{L})$, then \begin{itemize}
    \item if $f^+(e)=(v,e)$ for some $v\in V^l(T)$, then here we do want interpolation data restricted to $(S_r,w)$ to match the perturbation data on $S_{r_v}\in \mathcal{S}^{|v|}(\vec{L_v})$ \textit{as prescriped by the universal choice $p$};
    \item if $f^+(e)=(v,e)$ for some $v\in V^\col(T)$, then here we do want interpolation data restricted to $(S_r,w)$ to match the interpolation data on $(S_{r_v}, w_v)\in \mathcal{S}^{|v|,2}(\vec{L_v})$;
    \item if $f^+(e)=(v,e)$ for some $v\in V^r(T)$, then here we do want interpolation data restricted to $(S_r,w)$ to match the perturbation data on $S_{r_v}\in \mathcal{S}^{|v|}(\vec{L_v})$ \textit{as prescriped by the universal choice $q$}.
\end{itemize}
The following result, the proof of which we omit, is an extension to stacked clusters of a result contained in \cite{Sylvan2019OnCategories}.
\begin{lem}
    Consistent choices of interpolation data exist.
\end{lem}
\noindent For any $p,q\in E^{(\boldsymbol{f},\boldsymbol{g})}_\textnormal{reg}$, we define the space $E^{p,q}$ of consistent universal choices of interpolation data for stacked clusters.

\subsection{Moduli spaces of stacked Floer clusters with Lagrangian boundary}
Fix regular perturbation data $p,q\in E^{(\boldsymbol{f},\boldsymbol{g})}_\textnormal{reg}$ sharing the same Morse part and an interpolation datum $h\in E^{p,q}$ between them. Let $d\geq 1$ and $\vec{L}=(L_0,\ldots, L_d)$ be a tuple of Lagrangians in $\lagmd(M,\omega)$. We define moduli spaces of stacked Floer clusters with boundary on $\vec{L}$ with respect to $h$.\\
Assume first $L_0\neq L_d$. Pick $\vec{x}_i:=(x^i_1,\ldots,x^i_{m_i-1})$ for any $i=0,\ldots,d^R$ (see page \ref{defofmis} for the definition of the numers $m_i$), where $x^i_j\in \Crit(f^{\overline{L^i}})$ are critical points, orbits $\gamma_j\in \mathcal{O}(H^{\overline{L_{j-1}},\overline{L_j}}_p)$ for any $i=1,\ldots,d^R$ and $\gamma_+\in \mathcal{O}(H^{\overline{L_0},\overline{L_{d^R}}}_q)$ and a class $A\in \pi_2(M,\vec{L})$. We define the moduli space $$\mathcal{M}^{d+1,2}\left(\vec{x}_0, \gamma_1, \vec{x}_1,\ldots,\gamma_{d^R}, \vec{x}_{d^R}; \gamma^+;A;h\right)$$ of stacked Floer clusters joining $\vec{x}_0, \gamma_1, \vec{x}_1,\ldots,\gamma_{d^R}, \vec{x}_{d^R}$ to $\gamma^+$ in the class $A$ as the space of tuples $((r,w,T,\lambda),u)$ where $$u=\left((u_v)_{v\in V(T)}, (u_e)_{e\in E(T)}\right)\colon S_{r,T,\lambda}\to M$$ satisfies

\begin{enumerate}
    \item for any vertex $v\in V(T)$, $u_v\colon S_{r,T,\lambda}(v)\to M$ satisfies the $(K^{r,w,T,\lambda}_v(h),J^{r,w,T,\lambda}_v(h))$-Floer equation and the boundary conditions $u(\partial_iS_{r,T,\lambda})\subset \overline{L_i}$,
    \item for any $i=1,\ldots,d^R$ we have $$\lim_{s\to \-\infty}u_{h(e_i(T_\textnormal{red}))}(\overline{\epsilon_i}(s,t))=\gamma_i(t)$$ and $$\lim_{s\to \infty}u_{h(e_0(T_\textnormal{red}))}(\epsilon_0(s,t))=\gamma^+(t),$$
    \item for any $i=0,\ldots,d^R$ and any interior edge $e\in E_i^\textnormal{int}(T)$ (uni)labelled by $\overline{L_i}$ there is a class $B_e\in \pi_2(M,L)$ such that $$u_e\in \mathcal{P}(u_{t(e)}(z_{t(e)}), u_{h(e)}(z_{h(e)});(\boldsymbol{f},\boldsymbol{g});e;B_e),$$
    \item for any $i=0,\ldots,d^R$ and any $j=1,\ldots,m_i$ there is a class $B_j^i\in \pi_2(M,\overline{L_i})$ such that $$u_{e^i_j(T_\textnormal{uni})}\in \mathcal{P}(x^i_{j}, u_{h(e_j(T_i^F))}(z_{h(e_j(T_i^F))});(\boldsymbol{f},\boldsymbol{g});e^i_j(T_\textnormal{uni});B_j^i),$$
    \item We have the relation $$A=\sum_{v\in V(T)}[u_v]+\sum_{e\in E^\textnormal{int}(T)}B_e+\sum_{i=0}^{d^R}\sum_{j=1}^{m_i}B_j^i$$ on $\pi_2(M,\vec{L})$.
\end{enumerate}

Assume now $L_0=L_d$. Pick $x^+\in \Crit(f^{\overline{L_0}})$, $\vec{x}_i:=(x_1^i,\ldots,x_{\overline{m_i}}^i)$ for $i=0,\ldots,d^R+1$, where $x^i_j\in \Crit(f^{\overline{L_i}})$ are critical points, orbits $\gamma_j\in \mathcal{O}(H^{\overline{L_{j-1}},\overline{L_j}}_p)$ for $i=1,\ldots,d^R+1$ and a class $A\in \pi_2(M,\vec{L})$. We define the moduli space $$\mathcal{M}^{d+1,2}\left(\vec{x}_0, \gamma_1, \vec{x}_2,\ldots,\gamma_{d^R+1}, \vec{x}_{d^R+1}; x^+;A;h\right)$$ of stacked Floer clusters joining $\vec{x}_0, \gamma_1, \vec{x}_1,\ldots,\gamma_{d^R+1}, \vec{x}_{d^R+1}$ to $x^+$ in the class $A$ as the space of tuples $((r,w,T,\lambda),u)$ where $$u=((u_v)_{v\in V(T)}, (u_e)_{e\in E(T)})\colon S_{r,T,\lambda}\to M$$ satisfies

\begin{enumerate}
    \item for any vertex $v\in V(T)$, $u_v\colon S_{r,T,\lambda}(v)\to M$ satisfies the $(K^{r,w,T,\lambda}_v(h),J^{r,w,T,\lambda}_v(h))$-Floer equation and the boundary conditions $u(\partial_iS_{r,T,\lambda})\subset \overline{L_i}$,
    \item for any $i=1,\ldots,d^R$ we have $$\lim_{s\to \-\infty}u_{h(e_i(T_\textnormal{red}))}(\overline{\epsilon_i}(s,t))=\gamma_i(t),$$
    \item for any $i=0,\ldots,d^R$ and any interior edge $e\in E_i^\textnormal{int}(T)$ (uni)labelled by $\overline{L_i}$ there is a class $B_e\in \pi_2(M,L)$ such that $$u_e\in \mathcal{P}(u_{t(e)}(z_{t(e)}), u_{h(e)}(z_{h(e)});(\boldsymbol{f},\boldsymbol{g});e;B_e),$$
    \item for any $i=0,\ldots,d^R$ and any $j=1,\ldots,m_i$ there is a class $B_j^i\in \pi_2(M,\overline{L_i})$ such that $$u_{e^i_j(T_\textnormal{uni})}\in \mathcal{P}(x^i_{j}, u_{h(e^i_j(T_\textnormal{uni}))}(z_{h(e^i_j(T_\textnormal{uni}))});p;e^i_j(T_\textnormal{uni});B_j^i)$$ and there is a class $B_0^0\in \pi_2(M,\overline{L_0})$ such that $$u_{e_0(T)}\in \mathcal{P}(u_{t(e_0(T))}, x^+;(\boldsymbol{f},\boldsymbol{g});e_0(T);B_0^0),$$
    \item We have the relation $$A=\sum_{v\in V(T)}[u_v]+\sum_{e\in E^\textnormal{int}(T)}B_e+\sum_{i=0}^{d^R}\sum_{j=1}^{m_i}B_j^i+B_0^0$$ on $\pi_2(M,\vec{L})$.
\end{enumerate}

\noindent We have the following transversality result for stacked Floer clusters.
\begin{prop}\label{stacktrans}
    Let $(\boldsymbol{f}, \boldsymbol{g})$ be a choice of a Morse perturbation data in the sense of Section \ref{intdata} and $p,q\in E^{(\boldsymbol{f}, \boldsymbol{g})}_\textnormal{reg}$. Then there is a generic subset $E^{p,q}_\textnormal{reg}\subset E^{p,q}$ such that for any $h\in E^{p,q}_\textnormal{reg}$ and any tuple $\vec{L}=(L_0,\ldots, L_d)$ of Lagrangians in $\lagmd(M,\omega)$ of any length $d\geq 1$ the following hold:
    \begin{enumerate}
        \item if $L_0\neq L_d$ then for any $\vec{x}_i:=(x^i_1,\ldots,x^i_{m_i-1})$, $i=0,\ldots,d^R$, where $x^i_j\in \Crit(f^{\overline{L^i}})$ are critical points, any orbits $\gamma_j\in \mathcal{O}(H^{\overline{L_{j-1}},\overline{L_j}}_p)$, $i=1,\ldots,d^R$, and $\gamma_+\in \mathcal{O}(H^{\overline{L_0},\overline{L_{d^R}}}_q)$ and any class $A\in \pi_2(M,\vec{L})$ satisfying $$d_{A}^{(\vec{x}_i)_i, (\gamma_j)_j,\gamma^+}+1\leq 1$$ then the moduli space $$\mathcal{M}^{d+1,2}\left(\vec{x}_0, \gamma_1, \vec{x}_1,\ldots,\gamma_{d^R}, \vec{x}_{d^R}; \gamma^+;A;h\right)$$ is a smooth manifold of dimension $d_{A}^{(\vec{x}_i)_i, (\gamma_j)_j,\gamma^+}+1$.
        \item if $L_0=L_d$ then for any $x^+\in \Crit(f^{\overline{L_0}})$, $\vec{x}_i:=(x_1^i,\ldots,x_{\overline{m_i}}^i)$. $i=0,\ldots,d^R+1$, where $x^i_j\in \Crit(f^{\overline{L_i}})$ are critical points, any orbits $\gamma_j\in \mathcal{O}(H^{\overline{L_{j-1}},\overline{L_j}}_p)$, $i=1,\ldots,d^R+1$ and any class $A\in \pi_2(M,\vec{L})$ satisfying $$d_{A}^{(\vec{x}_i)_i, (\gamma_j)_j,x^+}+1\leq 1$$ then the moduli space $$\mathcal{M}^{d+1,2}\left(\vec{x}_0, \gamma_1, \vec{x}_2,\ldots,\gamma_{d^R+1}, \vec{x}_{d^R+1}; x^+;A;h\right)$$ is a smooth manifold of dimension $d_{A}^{(\vec{x}_i)_i, (\gamma_j)_j,x^+}+1$
    \end{enumerate}
\end{prop}
\noindent The proof is a combinations of the arguments contained in \cite{Sylvan2019OnCategories} with those sketched for the proof of Proposition \ref{tc1} (cfr. Proposition 4.7 in \cite{Sheridan2011OnPants}).
\begin{rem}
    Let $h\in E^{p,q}_\textnormal{reg}$ and assume $\vec{L}$ is made of $d+1$ identical Lagrangians, i.e. $\vec{L^F}=(L)$, and pick $\vec{x}$, $x^+$ and $A\in \pi_2(M,L)$ as above such that $d_{A}^{\vec{x},x^+}+1=0$. Then by a simple dimension argument all the moduli spaces $\mathcal{M}^{d+1,2}(\vec{x},x^+;A;h)$ are empty except if $d=1$, $x=x^+$ and $A=0$, in which case $\mathcal{M}^{2,2}(x^+;x^+)= \mathcal{M}^2(x^+,x^+)=\{x^+\}$, as hidden in the definition of interpolation data there is the fact that we do not perturb Morse perturbation data.
\end{rem}

\subsection{Definition of continuation functors}
Fix regular perturbation data $p,q\in E^{(\boldsymbol{f},\boldsymbol{g})}_\textnormal{reg}$ sharing the same Morse part and a regular interpolation datum $h\in E^{p,q}_\textnormal{reg}$ between them. By the results of Section \ref{mbfuk} associated to $p$ and $q$ there are strictly unital $A_\infty$-categories $\mathcal{F}uk(M;p)$ and $\mathcal{F}uk(M;q)$ which have the same set of objects. In this section we construct an $A_\infty$-functor $$\mathcal{F}^{p,q}\colon \mathcal{F}uk(M;p)\to \mathcal{F}uk(M;q)$$ which will depend on the choice of $h$.\\
First, we declare $\mathcal{F}^{p,q}$ to be the identity on objects.\\
Let $d\geq 1$ and $\vec{L}=(L_0,\ldots, L_d)$ be a tuple of Lagrangians in $\lagmd(M,\omega)$. We define $$\mathcal{F}^{p,q}_d\colon CF(L_0,L_1;p)\otimes \cdots\otimes CF(L_{d-1},L_d;p)\to CF(L_0,L_d;q)$$ as follows. First, we assume $L_0\neq L_d$. Let $\vec{x}_i:=(x^i_1,\ldots,x^i_{m_i-1})$, $i=0,\ldots,d^R$, where $x^i_j\in \Crit(f^{\overline{L_i}})$ are critical points, $\gamma_j\in \mathcal{O}(H^{\overline{L_{j-1}},\overline{L_j}}_p)$, $i=1,\ldots,d^R$ are orbits. Then we define $$\mathcal{F}^{p,q}_d(\vec{x}_0, \gamma_1, \vec{x}_1,\ldots,\gamma_{d^R}, \vec{x}_{d^R}):=\sum_{\gamma^+,A}\sum_u T^{\omega(u)}\cdot \gamma^+$$ where the first sum runs over orbits $\gamma^+\in \mathcal{O}(H^{L_0,L_d})$ and classes $A\in \pi_2(M,\vec{L})$ such that $$d_{A}^{(\vec{x}_i)_i, (\gamma_j)_j,\gamma^+}+1=0$$ and the second sum runs over stacked Floer clusters $$(r,w,T,\lambda,u)\in \mathcal{M}^{d+1,2}\left(\vec{x}_0, \gamma_1, \vec{x}_1,\ldots,\gamma_{d^R}, \vec{x}_{d^R}; \gamma^+;A;h\right)$$

Assume now $L_0=L_d$. Let  $\vec{x}_i:=(x_1^i,\ldots,x_{\overline{m_i}}^i)$ for $i=0,\ldots,d^R+1$, where $x^i_j\in \Crit(f^{\overline{L_i}})$ are critical points, orbits $\gamma_j\in \mathcal{O}(H^{\overline{L_{j-1}},\overline{L_j}})$ for any $i=1,\ldots,d^R$. Then we define $$\mathcal{F}^{p,q}_d(\vec{x}_0, \gamma_1, \vec{x}_1,\ldots,\gamma_{d^R}, \vec{x}_{d^R}):=\sum_{x^+,A}\sum_u T^{\omega(u)}\cdot x^+$$ where the first sum runs over critical points $x^+\in \Crit(f^{L_0})$ and classes $A\in \pi_2(M,\vec{L}$ such that $$d_{A}^{(\vec{x}_i)_i, (\gamma_j)_j,x^+}+1=0$$ and the second sum runs over Floer clusters $$(r,w,T,\lambda,u)\in \mathcal{M}^{d+1,2}\left(\vec{x}_0, \gamma_1, \vec{x}_2,\ldots,\gamma_{d^R+1}, \vec{x}_{d^R+1}; x^+;A;h\right)$$
Hidden in the next Proposition there is a compactness-type statement which follows from standard Gromov-compactness arguments up to the case where two or more configurations break simultaneosly: this case is a bit more delicate and is handled in [Lemma 3.28]\cite{Sylvan2019OnCategories}.
\begin{prop}
Let $(\boldsymbol{f}, \boldsymbol{g})$ as above and $p,q\in E^{(\boldsymbol{f}, \boldsymbol{g})}_\textnormal{reg}$. Then, for any $h\in E^{p,q}_\textnormal{reg}$, $\mathcal{F}^{p,q}$ is a weakly-filtered unital $A_\infty$-functor.
\end{prop}
Unitality of $\mathcal{F}^{p,q}$ is obvious since by dimension arguments $\mathcal{F}^{p,q}_1$ applied to a couple $(L,L)$ of identical Lagrangians in $\lagmd(M,\omega)$ is the identity. The fact that $\mathcal{F}^{p,q}$ is weakly-filtered comes from energy-action identities involving curvature terms identical to those carried out in Section \ref{wfs} and uniform bounds can be found as in \cite{Biran2021LagrangianCategories}. The functor $\mathcal{F}^{p,q}$ will be called a continuation functor between the Fukaya categories $\mathcal{F}uk(M;p)$ and $\mathcal{F}uk(M;q)$.

Since, as it is well-known, continuation maps (that is, linear parts of continuation functors) induce isomorphisms in homology, we have the following result.
\begin{prop}\label{inv}
    Let $(\boldsymbol{f}, \boldsymbol{g})$ as above and $p,q\in E^{(\boldsymbol{f}, \boldsymbol{g})}_\textnormal{reg}$. Then, for any $h\in E^{p,q}_\textnormal{reg}$, $\mathcal{F}^{p,q}$ is a quasi-equivalence between $\mathcal{F}uk(M;p)$ and $\mathcal{F}uk(M;q)$. Moreover, $\mathcal{F}^{p,q}$ does not depend on $h$ up to quasi isomorphism of functors.
\end{prop}
\noindent The last part of the proposition has been proved in \cite{Sylvan2019OnCategories} by constructing `continuation homotopies'. As we mentioned above, Proposition \ref{inv} can be generalized to continuation functors between \textit{any} two perturbation data, not only those sharing the same Morse part (see Proposition \ref{inv1}). Given the above machinery, this is not hard to achieve, but goes beyond our aim to use continuation functors to estimate `distances' between filtered Fukaya categories.

\subsection{$(\varepsilon_1,\varepsilon_2)$-Continuation functors and filtrations}\label{contfunandfiltr}
Consider two positive real number $\varepsilon_1,\varepsilon_2>0$ and $\delta_1,\delta_2\in \left(\frac{1}{2},1\right)$ and two perturbation data $p\in E^{\varepsilon_1, \delta_1}$ and $q\in E^{\varepsilon_2, \delta_2}$. Recall that by the main result of this paper, Theorem \ref{mainthm}, the Fukaya categories $\mathcal{F}uk(M;p)$ and $\mathcal{F}uk(M;q)$ are strictly unital and filtered. Note that an arbitrary continuation functor defined via an element of $E^{p,q}$ might shift filtration in a way that is not controlled by  the parameters $\varepsilon_1$, $\varepsilon_2$, $\delta_1$ and $\delta_2$. In this section, we will construct a subclass $E^{p,q;f}\subset E^{p,q}$ of so called $(\varepsilon_1, \varepsilon_2)$-interpolation data such that the associated continuation functors are still weakly-filtered, but with discrepancies controlled by the above parameters. Moreover, we describe choices of parameters such that the associated continuation functors are genuinely filtered. The construction of $(\varepsilon_1, \varepsilon_2)$-interpolation data is almost identical to the construction of $\varepsilon$-perturbation data, as all the Hamiltonian perturbation will still lie on strip-like ends as in Section \ref{actualconstruction}, hence we will only explicitly define $(\varepsilon_1,\varepsilon_2)$-interpolation data for cyclically different tuples of Lagrangians of length $2$ and $3$, as the general case follows by obvious modifications and gluing.
Let $\vec{L}=(L_0,L_1)$ be a couple of different Lagrangians on $\lagmd(M,\omega)$ and consider the strip $S^{2,2}$ over $\mathcal{R}^{2,2}(\vec{L})$ coming with strip-like ends $\epsilon_1,\epsilon_2$. A choice of $(\varepsilon_1,\varepsilon_2)$-interpolation datum between $p$ and $q$ for $\vec{L}$ on the strip is a choice of interpolation datum $(K, J)$ as defined in Section \ref{intdata} but subject to the following restrictions: 
\begin{enumerate}
    \item $K$ vanishes away from the strip-like ends;
    \item On the negative strip-like end we have $$K = \overline{H^{L_{i-1},L_i}}(s,t)dt$$ where $$\overline{H^{L_{i-1},L_i}}\colon (-\infty, 0)\times [0,1]\times M\to \R$$ is of the form $$\overline{H^{L_{i-1},L_i}}(s,t)= (1-\beta^{L_{i-1},L_i}(s+1))H^{L_{i-1},L_i}_{p}(t)$$ for some $\beta^{L_{i-1},L_i}\in \mathcal{F}$ (see page \pageref{definitionofF} for the definition of the family $\mathcal{F}$);
    \item On the positive strip-like end we have $$K = \overline{H^{L_{0},L_2}}(s,t)dt$$ where $$\overline{H^{L_{0},L_2}}\colon[0,\infty)\times [0,1]\times M\to \R$$ is of the form $$\overline{H^{L_{0},L_2}}(s,t) = \beta^{L_0,L_2}(s)H^{L_0,L_2}_{q}(t)$$ for some $\beta^{L_0,L_2}\in \mathcal{F}$.
\end{enumerate}
Let now $\vec{L}:=(L_0,L_1,L_2)$ be a tuple of cyclically different monotone Lagrangians in $\lagmd(M,\omega)$. Notice that $\mathcal{R}^{3,2}_C(\vec{L})=\{pt\}\times (0,\infty)$. A choice of $(\varepsilon_1,\varepsilon_2)$-interpolation datum between $p$ and $q$ for $\vec{L}$ over $(pt,w)\in \mathcal{R}^{3,2}_C(\vec{L})$ is a choice of interpolation datum $(K^{pt,w}, J^{pt,w})$ as defined in Section \ref{intdata} but subject to the following restrictions: 
\begin{enumerate}
    \item $K^{pt,w}$ vanishes away from the strip-like ends;
    \item On the $i$th ($i\in \{1,2\})$ strip-like end we have $$K^{pt,w} = \overline{H^{L_{i-1},L_i}_{pt,w}}(s,t)dt$$ where $$\overline{H^{L_{i-1},L_i}_{pt,w}}\colon (-\infty, 0)\times [0,1]\times M\to \R$$ is of the form $$\overline{H^{L_{i-1},L_i}_{pt,w}}(s,t)= (1-\beta^{L_{i-1},L_i}_{pt,w}(s+1))H^{L_{i-1},L_i}_{p}(t)$$ for some $\beta^{L_{i-1},L_i}_{pt,w}\in \mathcal{F}$;
    \item On the unique positive strip-like end we have $$K^{pt,w} = \overline{H^{L_{0},L_2}_{pt,w}}(s,t)dt$$ where $$\overline{H^{L_{0},L_2}_{pt,w}}\colon[0,\infty)\times [0,1]\times M\to \R$$ is of the form $$\overline{H^{L_{0},L_2}_{pt,w}}(s,t) = \beta^{L_0,L_2}_{pt,w}(s)H^{L_0,L_2}_{q}(t)$$ for some $\beta^{L_0,L_2}_{pt,w}\in \mathcal{F}$.
\end{enumerate}
$(\varepsilon_1,\varepsilon_2)$-interpolation data for longer tuples of Lagrangians are constructed by gluing interpolation data for tuples $\vec{L}$ with reduced tuple of length $2$ and $3$. Assuming transversality, it follows by the construction of continuation functors that the $d$th term of a continuation functor defined via an interpolation datum with parameters $\varepsilon_1, \varepsilon_2, \delta_1$ and $\delta_2$ shifts filtration by at most\footnote{In the case $L_d\neq L_0$, while in the case $L_0=L_d$ we can refine the estimate to $(d-3)$} $$(d-1)\varepsilon_2(1-2\delta_2)+ d(\varepsilon_2-\delta_1\varepsilon_1)$$ since the least filtration preserving configurations are those endowed with interpolation data lying near the boundary part of $\partial\mathcal{R}^{d+1,2}(\vec{L})$ corresponding to $$\mathcal{R}^{d+1}_C(\vec{L})\times \prod_{i=1}^d \mathcal{R}^{2,2}(L_{i-1},L_i).$$
Similarly to the case of $\varepsilon$-perturbation data in Section \ref{epspd3}, we define regular $(\varepsilon_1,\varepsilon_2)$-interpolation data as follows:
 \begin{defn}
        We define $E^{p,q;f}_\textnormal{reg}\subset E^{p,q}_\textnormal{reg}$ to be the space of regular interpolation data $h$ such that the Floer parts of $h$ is obtained via a deformation of the Floer parts of some $(\varepsilon_1,\varepsilon_2)$-interpolation datum $h'\in E^{p,q;f}$ in the sense of \cite[Chapter 9k]{Seidel2008FukayaTheory} (that is, by deformation with support on the thich parts of the polygons, to keep consistency) and such that the associated (well-defined) continuation functor $\mathcal{F}^{p,q}_h$ shifts filtration by $\leq \varepsilon_2-\delta_1\varepsilon_1$ in the sense introduced at the end of Section \ref{ainfpre}.
\end{defn}
\noindent We show that $E^{p,q;f}_\textnormal{reg}$ is non-empty. Pick an interpolation datum $h\in E^{p,q;f}$. It is well-known that there is no need to perturb the induced interpolation data on strips in order to get regularity for $d=1$. Let $d\geq 2$, $\vec{L}=(L_0,\ldots, L_d)$ be a tuple of Lagrangians in $\lagmd(M,\omega)$ and $\left(K^{\vec{L}}, J^{\vec{L}}\right)$ the Floer part of $p$ restricted to $\mathcal{R}^{d+1,2}_C(\vec{L})$. As a generic deformation $(\Delta K^{\vec{L}}, \Delta J^{\vec{L}})$ (vanishing on thin parts) turns $\left(K^{\vec{L}},J^{\vec{L}}\right)$ into a regular perturbation datum $(K^{\vec{L}}+\Delta K^{\vec{L}},J^{\vec{L}}\exp(-J^{\vec{L}}\Delta J^{\vec{L}})) $ we can choose a generic Hamiltonian deformation $\Delta K^{\vec{L}}$ such that the associated function $$(r,w,T,\lambda)\in \mathcal{R}^{d+1,2}_C(\vec{L})\longmapsto \sum_{v\in T_{\textnormal{red}}}\int_{S_{r,w,T,\lambda}(v)}\max_{x\in M}R^{\Delta K^{\vec{L}}}_{r,w,T,\lambda}(v)$$ is bounded above by the positive number $$(d-1)\varepsilon(2\delta_2-1).$$ It follows that for any $h\in E^{p,q;f}_\textnormal{reg}$ and any $d\geq 1$ the $d$th term $\left(\mathcal{F}^{p,q}_h\right)^d$ shifts filtration by $\leq d(\varepsilon_2-\delta_1\varepsilon_1)$, that is, $\mathcal{F}^{p,q}_h$ shifts filtration by $\leq \varepsilon_2-\delta_1\varepsilon_1$.

\newpage

\printbibliography

\end{document}